\documentclass[12pt]{article}
\usepackage{amsmath,amssymb,amsfonts,amsthm}
\usepackage{yhmath,relsize}
\usepackage{eucal,oldgerm}
\usepackage[dvips]{graphics}
\usepackage{pstricks,pst-plot,pst-node,pst-coil}
\newpsobject{showgrid}{psgrid}{subgriddiv=1,griddots=5,gridlabels=0pt}
\usepackage[arrow,matrix,tips,2cell]{xy}

\newcommand{\rang}{\right\rangle}

\newcommand{\zz}{{\mathfrak{z}}}

\newcommand{\com}{{\mathbb C}}

\newcommand{\proj}{\mathbf{P}}

\newcommand{\bW}{\mathsf{W}}

\newcommand{\PP}{{\mathbf{P}}}

\newcommand\FF{\mathbb F}
\newcommand\ZZ{\mathsf Z}
\newcommand{\oh}{{\mathcal O}}
\newcommand{\T}{{\mathbf{T}}}
\newcommand{\OO}{{\mathcal{O}}}
\newcommand{\w}{{\mathsf{w}}}

\newcommand{\LL}{{\mathbb{L}}}

\newcommand{\C}{\mathbb{C}}
\newcommand{\Q}{\mathbb{Q}}
\newcommand{\Z}{\mathbb{Z}}

\newcommand{\cO}{\mathcal{O}}

\newcommand{\Pp}{{\mathbf{P}^1}}

\newcommand{\rarr}{\rightarrow}

\newcommand{\bZ}{\mathsf{Z}}
\newcommand{\oM}{\overline{M}}

\DeclareMathOperator{\Hilb}{Hilb}

\newtheorem{Theorem}{Theorem}
\newtheorem{Lemma}{Lemma}
\newtheorem{Corollary}{Corollary}

\newtheorem{Proposition}[Lemma]{Proposition}
\newtheorem{Conjecture}{Conjecture}

\begin{document}
\title{Gromov-Witten/Pairs correspondence
for the quintic
3-fold}
\author{R. Pandharipande and A. Pixton}
\date{January 2016}
\maketitle

\begin{abstract}
We use the Gromov-Witten/Pairs descendent
correspondence for toric 3-folds and 
degeneration arguments to establish the GW/P
correspondence for several compact Calabi-Yau
3-folds (including all 
CY complete intersections in products of projective
spaces). A crucial aspect of the proof is the study of
the GW/P correspondence for descendents in relative
geometries. 
Projective bundles
over surfaces relative to a section play a special role.

The GW/P correspondence for Calabi-Yau
complete intersections provides a  
structure result for the Gromov-Witten invariants
in a fixed curve class.
After change of variables, the Gromov-Witten series is a
rational function in the variable $-q=e^{iu}$ invariant
under $q \leftrightarrow q^{-1}$.
\end{abstract}

\maketitle

\setcounter{tocdepth}{2} 
\tableofcontents


\setcounter{section}{-1}
\section{Introduction}

\subsection{Overview}
The main result of the paper is a proof of the Gromov-Witten/Pairs
correspondence for several compact Calabi-Yau 3-folds (including
all Calabi-Yau complete intersections in products of projective
spaces). The Gromov-Witten/Pairs correspondence was first
stated in terms of the Donaldson-Thomas theory of ideal sheaves in
\cite{MNOP1,MNOP2} and is often referred to as the {\em MNOP
conjecture}. 
\begin{enumerate}
\item[(i)]
Via the Gromov-Witten theory
of the moduli of stable maps to a 3-fold $X$,
the generating series of curve counts is defined with a 
genus parameter $u$. 
\item[(ii)]
Via the Donaldson-Thomas theory
of the moduli of ideal sheaves on $X$, the generating series of 
sheaf counts is defined with an Euler characteristic parameter $q$.
\end{enumerate}
The MNOP conjecture equates the  generating series (i) and (ii) after
the non-trivial change of variables
$$-q = e^{iu}\, .$$
For Calabi-Yau 3-folds, the formulation
in terms of stable pairs \cite{pt} was proven to be equivalent in 
\cite{Bridge,Toda}.

Our proof here uses much of the development of the Donaldson-Thomas theory
of 3-folds in  the past decade. The first input is
a series of papers culminating in \cite{moop} which establish the
MNOP conjecture for nonsingular, quasi-projective, {\em toric}
3-folds. The results essentially concern the 3-fold
toric vertex (first studied in the Calabi-Yau case in \cite{AKMV}).
The argument of \cite{moop} uses
\begin{enumerate}
\item[$\bullet$] the proof of the MNOP conjecture for
local curves established in the papers \cite{BP,hilb1,lcdt}, 
\item[$\bullet$] the proof of the MNOP conjecture for
3-folds $A_n\times \mathbf{P}^1$, where $A_n$ is the
the holomorphic symplectic resolution of the standard $A_n$ - surface
singularity, in the papers \cite{DM,mo1,mo2}.
\end{enumerate}
A basic idea introduced in \cite{moop} is the notion of 
the {\em capped} 3-fold vertex -- a reorganization of the
standard localization formula which respects the MNOP conjecture.

The second input is the foundation of the theory of stable
pairs developed in \cite{pt,pt2,pt3}. Stable pairs are much better
behaved than the Donaldson-Thomas theory of ideal sheaves
since there are no floating points. Real differences
between stable pairs and ideal sheaves appear in the study of
{\em descendent} invariants involving the integration
of the slant products of the Chern characters of the tautological
sheaves over the moduli space. The generating functions
of descendent invariants for stable pairs are conjectured to
be rational functions (while the  parallel generating functions
for ideal sheaves are known  to be irrational).

The third input is the study of descendent invariants
for the stable pairs theory of 3-folds:
\begin{enumerate}
\item[$\bullet$] the proof  of the rationality of the generating series
for {\em toric} 3-folds in the papers \cite{part1, PP2},
\item[$\bullet$] the formulation (and proof in the toric case) of
a GW/Pairs correspondence for descendents in \cite{PPDC}.
\end{enumerate}
For the toric arguments in \cite{PPDC},   
a {\em capped} 3-fold descendent vertex is introduced.

Given a 3-fold $X$ and a nonsingular divisor $D\subset X$, there
are {\em relative} stable pairs and Gromov-Witten theories \cite{L,LW,MNOP2,part1}.
The interaction of the descendent theory with the relative
theory plays a crucial role. 
We approach compact
Calabi-Yau 3-fold via degeneration to toric geometries. In
order to prove the GW/Pairs correspondence, we prove appropriate
GW/Pairs correspondence for {\em all} the simpler descendent and
relative geometries which arise in the degeneration process.

A crucial case concerns the geometry of a $\mathbf{P}^1$-bundle,
$$\pi: \mathbf{P}_S \rightarrow S\, ,$$
over a surface $S$ relative to a section of $\pi$.
We prove GW/Pairs correspondences in case $S$ is 
a toric surface, a $K3$ surface, or a projective bundle over a higher genus curve $C$.
The proofs systematically use the descendent theory of
3-folds. Once the GW/Pairs correspondences for these 
special relative geometries are established, then the
degeneration scheme of \cite{mptop} can be used
to prove the GW/Pairs correspondence for any compact
Calabi-Yau 3-fold which admits a good degeneration.

\subsection{Descendents in Gromov-Witten theory} \label{gwdess}
Let $X$ be a nonsingular projective 3-fold.
Gromov-Witten theory is defined via integration over the moduli
space of stable maps.
Let
 $\overline{M}_{g,r}(X,\beta)$ denote the moduli space of
$r$-pointed stable maps from connected genus $g$ curves to $X$ representing the
class $\beta\in H_2(X, \Z)$. Let 
$$\text{ev}_i: \overline{M}_{g,r}(X,\beta) \rarr X,$$
$$ \LL_i \rarr \overline{M}_{g,r}(X,\beta)$$
denote the evaluation maps and the cotangent line bundles associated to
the marked points.
Let $\gamma_1, \ldots, \gamma_r\in H^*(X,{\mathbb{Q}})$, and
let $$\psi_i = c_1(\LL_i) \in H^2(\overline{M}_{g,n}(X,\beta),\mathbb{Q}).$$
The {\em descendent} fields, denoted by $\tau_k(\gamma)$, correspond 
to the classes $\psi_i^k \text{ev}_i^*(\gamma)$ on the moduli space
of maps. 
Let
$$\Big\langle \tau_{k_1}(\gamma_{1}) \cdots
\tau_{k_r}(\gamma_{r})\Big\rangle_{g,\beta} = \int_{[\overline{M}_{g,r}(X,\beta)]^{vir}} 
\prod_{i=1}^r \psi_i^{k_i} \text{ev}_i^*(\gamma_{_i})$$
denote the descendent
Gromov-Witten invariants. Foundational aspects of the theory
are treated, for example, in \cite{Beh, BehFan, LiTian}.

Let $C$ be a possibly disconnected curve with at worst nodal singularities.
The genus of $C$ is defined by $1-\chi(\oh_C)$. 
Let $\overline{M}'_{g,r}(X,\beta)$ denote the moduli space of maps
with possibly {disconnected} domain
curves $C$ of genus $g$ with {\em no} collapsed connected components.
The latter condition requires 
 each connected component of $C$ to represent
a nonzero class in $H_2(X,{\mathbb Z})$. In particular, 
$C$ must represent a {nonzero} class $\beta$.

We define the descendent invariants in the disconnected 
case by
$$\Big\langle \tau_{k_1}(\gamma_{1}) \cdots
\tau_{k_r}(\gamma_{r})\Big\rangle'_{g,\beta} = \int_{[\overline{M}'_{g,r}(X,\beta)]^{vir}} 
\prod_{i=1}^r \psi_i^{k_i} \text{ev}_i^*(\gamma_{i}).$$
The associated partition function is defined 
by{\footnote{Our 
notation follows \cite{moop,MNOP2} and emphasizes the
role of the moduli space $\overline{M}'_{g,r}(X,\beta)$. 
The degree 0 collapsed contributions
will not appear anywhere in our paper.}} 
\begin{equation}
\label{abc}
\bZ'_{\mathsf{GW}}\Big(X;u\ \Big|\ \prod_{i=1}^r \tau_{k_i}(\gamma_{i})\Big)_\beta = 
\sum_{g\in{\mathbb Z}} \Big \langle \prod_{i=1}^r
\tau_{k_i}(\gamma_{i}) \Big \rangle'_{g,\beta} \ u^{2g-2}.
\end{equation}
Since the domain components must map nontrivially, an elementary
argument shows the genus $g$ in the  sum \eqref{abc} is bounded from below. 
The descendent insertions in \eqref{abc} should
match  the  (genus independent) virtual dimension,
\begin{equation}\label{k345}
\text{dim} \ [\overline{M}'_{g,r}(X,\beta)]^{vir} = \int_\beta c_1(T_X) + r.
\end{equation}

If $X$ is a nonsingular projective toric 3-fold, then the descendent
invariants can be lifted to equivariant cohomology.
Let 
$$\T=(\com^*)^3$$ be the 3-dimensional algebraic torus acting on $X$.
Let $s_1,s_2,s_3$ be the equivariant first Chern classes
of the standard representations of the three factors of $\T$. The
equivariant cohomology of the point is well-known to be
$$H^*_{\T}(\bullet) = \mathbb{Q}[s_1,s_2,s_3]\ .$$ 
For equivariant classes $\gamma_{i}\in H^*_{\T}(X,\mathbb{Q})$, the
descendent invariants
$$\Big\langle \tau_{k_1}(\gamma_{1}) \cdots
\tau_{k_r}(\gamma_{r})\Big\rangle'_{g,\beta} = \int_{[\overline{M}'_{g,r}(X,\beta)]^{vir}} 
\prod_{i=1}^r \psi_i^{k_i} \text{ev}_i^*(\gamma_{i})\ \in H^*_{\T}(\bullet)$$
are well-defined. In the equivariant setting, the descendent
insertions may exceed the virtual dimension \eqref{k345}.
The equivariant partition function
$$\bZ'_{\mathsf{GW}}\Big(X;u\ \Big|\ \prod_{i=1}^r \tau_{k_i}(\gamma_{i})\Big)^\T_\beta \in \mathbb{Q}[s_1,s_2,s_3]((u))$$ 
is a Laurent series in $u$ with coefficients in $H^*_\T(\bullet)$.

If $X$ is a nonsingular quasi-projective toric 3-fold, the 
equivariant Gromov-Witten invariants of $X$ are still
well-defined{\footnote{A quasi-projective toric
variety $X$ has a finite skeleton of $1$-dimensional
projective torus orbits. For a stable map to $X$
to be torus fixed, the image must lie in the $1$-dimensional
skeleton. Hence, the torus fixed locus of the moduli
space of stable maps is compact.}} by localization residues \cite{BP}. In
the quasi-projective case,
$$\bZ'_{\mathsf{GW}}\Big(X;u\ \Big|\ \prod_{i=1}^r \tau_{k_i}(\gamma_{i})\Big)^\T_\beta \in \mathbb{Q}(s_1,s_2,s_3)((u))\ . $$ 
For the study of the Gromov-Witten theory of toric 3-folds,
the open geometries play an important role.

\subsection{Descendents in the theory of stable pairs}\label{dess}
Let $X$ be a nonsingular projective 3-fold, and let
$\beta \in H_2(X,\mathbb{Z})$ be a nonzero class. We consider next the
moduli space of stable pairs{\footnote{See \cite{pt} for a foundational
development. An introduction to the subject of stable pairs can be found in \cite{rp13}.}}
$$[\OO_X \stackrel{s}{\rightarrow} F] \in P_n(X,\beta)$$
where $F$ is a pure sheaf supported on a Cohen-Macaulay subcurve of $X$, 
$s$ is a morphism with 0-dimensional cokernel, and
$$\chi(F)=n, \  \  \ [F]=\beta.$$
The space $P_n(X,\beta)$
carries a virtual fundamental class obtained from the 
deformation theory of complexes in
the derived category \cite{pt}.

Since $P_n(X,\beta)$ is a fine moduli space, there exists a universal sheaf
$$\FF \rightarrow X\times P_{n}(X,\beta),$$
see Section 2.3 of \cite{pt}.
For a stable pair $[\OO_X\to F]\in P_{n}(X,\beta)$, the restriction of
$\FF$
to the fiber
 $$X \times [\OO_X \to F] \ \subset\  
X\times P_{n}(X,\beta)
$$
is canonically isomorphic to $F$.
Let
$$\pi_X\colon X\times P_{n}(X,\beta)\to X,$$
$$\pi_P\colon X\times P_{n}(X,\beta)
\to P_{n}(X,\beta)$$
 be the projections onto the first and second factors.
Since $X$ is nonsingular
and
$\FF$ is $\pi_P$-flat, $\FF$ has a finite resolution 
by locally free sheaves.
Hence, the Chern character of the universal sheaf $\FF$ on 
$X \times P_n(X,\beta)$ is well-defined.
By definition, the operation
$$
\pi_{P*}\big(\pi_X^*(\gamma)\cdot \text{ch}_{2+i}(\FF)
\cap \pi_P^*(\ \cdot\ )\big)\colon 
H_*(P_{n}(X,\beta))\to H_*(P_{n}(X,\beta))
$$
is the action of the descendent $\tau_i(\gamma)$, where
$\gamma \in H^*(X,\Z)$.

For nonzero $\beta\in H_2(X,\Z)$ and arbitrary $\gamma_i\in H^*(X,\Q)$,
define the stable pairs invariant with descendent insertions by
\begin{eqnarray*}
\Big\langle \tau_{k_1}(\gamma_1)\ldots \tau_{k_r}(\gamma_r)
\Big\rangle_{\!n,\beta}&  = &
\int_{[P_{n}(X,\beta)]^{vir}}
\prod_{i=1}^r \tau_{k_i}(\gamma_i) \\
& = & 
\int_{P_n(X,\beta)} \prod_{i=1}^r \tau_{k_i}(\gamma_{i})
\Big( [P_{n}(X,\beta)]^{vir}\Big).
\end{eqnarray*}
The partition function is 
$$
\ZZ_{\mathsf{P}}\Big(X;q\ \Big|   \prod_{i=1}^r \tau_{k_i}(\gamma_{i})
\Big)_\beta
=\sum_{n} 
\Big\langle \prod_{i=1}^r \tau_{k_i}(\gamma_{i}) 
\Big\rangle_{\!n,\beta} q^n.
$$

Since $P_n(X,\beta)$ is empty for sufficiently negative
$n$, the partition function 
is a Laurent series in $q$. The following conjecture was made in 
\cite{pt2}.

\begin{Conjecture}
\label{111} 
The partition function
$\ZZ_{\mathsf{P}}\big(X;q\ |   \prod_{i=1}^r \tau_{k_i}(\gamma_{i})
\big)_\beta$ is the 
Laurent expansion of a rational function in $q$.
\end{Conjecture}

Let $X$ be a nonsingular quasi-projective toric 3-fold. 
The stable pairs descendent
invariants can be lifted to equivariant cohomology (and
defined by residues in the open case). For
equivariant classes $\gamma_i \in H^*_{\T}(X,\mathbb{Q})$, we see
$$\bZ_{\mathsf{P}}\Big(X;q\ \Big|\ \prod_{i=1}^r \tau_{k_i}(\gamma_{i})\Big)^\T_\beta \in \mathbb{Q}(s_1,s_2,s_3)((q))$$ 
is a Laurent series in $q$ with coefficients in $H^*_\T(\bullet)$.
A central result of \cite{part1,PP2} is the following rationality
property.

\vspace{10pt}
\noindent{\bf Toric rationality.} {\em
Let 
$X$ be a nonsingular quasi-projective toric 3-fold. The partition function
$$\ZZ_{P}\Big(X;q\ \Big|   \prod_{i=1}^r \tau_{k_i}(\gamma_i)
\Big)^\T_\beta$$ is the 
Laurent expansion in $q$ of a rational function in the field 
$\mathbb{Q}(q,s_1,s_2,s_3)$.}

\vspace{10pt}

The above rationality result implies Conjecture 1 when $X$ is
a nonsingular projective toric 3-fold. 
The corresponding statement for the equivariant
Gromov-Witten descendent partition function is expected
(from calculational evidence) to be false.

\subsection{Descendent correspondence}
%
Let $X$ be a nonsingular projective 3-fold.
Let $\widehat{\alpha}=(\widehat{\alpha}_1,\ldots,
\widehat{\alpha}_{\widehat{\ell}})$ be a partition
of length $\widehat{\ell}$.
Let 
$$\iota_\Delta:\Delta\rightarrow X^{\widehat{\ell}}$$
 be the inclusion of the
small diagonal in the product $X^{\widehat{\ell}}$.
For $\gamma\in H^*(X,\mathbb{Q})$, 
we write $$\gamma\cdot \Delta =\iota_{\Delta*}(\gamma) \in H^*(X^{\widehat{\ell}}, \mathbb{Q})\, .$$ 
By K\"unneth decomposition, we have
$$\gamma\cdot \Delta=
\sum_{{j_1, \ldots, j_{\hat{\ell}}}} c^\gamma_{j_1,\ldots, j_{\hat{\ell}}}\, 
\theta_{j_1} \otimes
\ldots\otimes \theta_{j_{\hat{\ell}}}\, ,$$
where $\{\theta_j\}$ is a basis of $H^*(X,\mathbb{Q})$.
We define the descendent insertion $\tau_{\widehat{\alpha}}(\gamma)$ by
\begin{equation}\label{j77833}
\tau_{\widehat{\alpha}}(\gamma)= 
\sum_{j_1,\ldots,j_{\hat{\ell}}} c^\gamma_{j_1,\ldots, j_{\hat{\ell}}}\,
\tau_{\widehat{\alpha}_1-1}(\theta_{j_1})
\cdots\tau_{\widehat{\alpha}_{\hat{\ell}}-1}(\theta_{j_{\hat{\ell}}})\ .
\end{equation}
Three basic examples are:
\begin{enumerate}
\item[$\bullet$] If $\widehat{\alpha}=(\widehat{a}_1)$, then 
$$\tau_{(\, \widehat{a}_1\,)}(\gamma)= \tau_{\widehat{a}_1-1}(\gamma)\, .$$
The convention of shifting
the descendent by $1$  allows us to index descendent insertions by
standard partitions $\widehat{\alpha}$. The shift by $1$ 
is natural from the point of view of 
relative/descendent correspondences and follows the notation of \cite{PPDC}.

\item[$\bullet$] If $\widehat{\alpha}=(\widehat{a}_1,\widehat{a}_2)$ 
and $\gamma=1$ is the identity class, then
$$\tau_{(\, \widehat{a}_1,\, \widehat{a}_2\, )}(1)= \sum_{j_1,j_2} c^1_{j_1,j_2} 
\tau_{\widehat{a}_1-1}(\theta_{j_1})\, \tau_{\widehat{a}_2-1}(\theta_{j_2})\, ,$$
where $\Delta = \sum_{j_1,j_2} c^1_{j_1,j_2}\, \theta_{j_1} \otimes \theta_{j_2}$
is the standard K\"unneth decomposition of the diagonal in $X^2$.
\item[$\bullet$] If 
$\gamma$ is the class of a point, then 
\[
\tau_{\widehat{\alpha}}(\mathsf{p})=
\tau_{\widehat{\alpha}_1-1}(\mathsf{p})\cdots\tau_{\widehat{\alpha}_{\hat{\ell}}-1}(\mathsf{p}).
\]
\end{enumerate}
By the multilinearity of descendent insertions, 
formula  \eqref{j77833} does not depend upon the
basis choice $\{\theta_j\}$.


A central result of \cite{PPDC} is
the construction of
a universal correspondence matrix $\widetilde{\mathsf{K}}$ 
indexed by partitions
$\alpha$ and $\widehat{\alpha}$ of positive size with{\footnote{Here, $i^2=-1$.}}
$$\widetilde{\mathsf{K}}_{\alpha,\widehat{\alpha}}\in 
\mathbb{Q}[i,c_1,c_2,c_3]((u))\ $$
and $\widetilde{\mathsf{K}}_{\alpha,\widehat{\alpha}}=0$ 
unless $|\alpha|\geq |\widehat{\alpha}|$.
Via the substitution
\begin{equation} \label{h3492}
c_i=c_i(T_X),
\end{equation}
the elements of $\widetilde{\mathsf{K}}$
act by cup product on the cohomology 
of $X$ with $\mathbb{Q}[i]((u))$-coefficients.
The coefficients $\widetilde{\mathsf{K}}_{\alpha,\widehat{\alpha}}$
are constructed from the capped descendent vertex \cite{PPDC}.

The matrix $\widetilde{\mathsf{K}}$ is 
used to define a correspondence
rule
\begin{equation}\label{pddff}
{\tau_{\alpha_1-1}(\gamma_1)\cdots
\tau_{\alpha_{\ell}-1}(\gamma_{\ell})}\ \  \mapsto\ \ 
\overline{\tau_{\alpha_1-1}(\gamma_1)\cdots
\tau_{\alpha_{\ell}-1}(\gamma_{\ell})}\ .
\end{equation}
The formula for the right side
of \eqref{pddff} requires a sum over all set
partitions $P$ of $\{ 1,\ldots, \ell \}$.
 For such a  set partition
$P$, each element $S\in P$
is a subset of $\{1,\ldots, \ell\}$.
Let $\alpha_S$ be the associated subpartition of
$\alpha$, and let
$$\gamma_S = \prod_{i\in S}\gamma_i.$$
In case all cohomology classes $\gamma_j$ are even,
we define the right side of \eqref{pddff} 
by
\begin{equation}\label{mqq23}
\overline{\tau_{\alpha_1-1}(\gamma_1)\cdots
\tau_{\alpha_{\ell}-1}(\gamma_{\ell})}
=
\sum_{P \text{ set partition of }\{1,\ldots,\ell\}}\ \prod_{S\in P}\ \sum_{\widehat{\alpha}}\tau_{\widehat{\alpha}}(\widetilde{\mathsf{K}}_{\alpha_S,\widehat{\alpha}}\cdot\gamma_S) \ .
\end{equation}
The last sum is over 
all partitions $\widehat{\alpha}$ of positive size, but by the vanishing
$$\widetilde{\mathsf{K}}_{\alpha_S,\widehat{\alpha}}=0 \ \ \text{unless}
\ \ |{\alpha_S}|\geq |\widehat{\alpha}|\, , $$
the summation index  may be restricted to partitions $\widehat{\alpha}$ of positive size bounded
by $|\alpha_S|$.

The leading term of the descendent correspondence is calculated
in \cite{PPDC},
\begin{equation*}
\overline{\tau_{\alpha_1-1}(\gamma_1)\cdots
\tau_{\alpha_{\ell}-1}(\gamma_{\ell})}
= (iu)^{\ell(\alpha)-|\alpha|}\, \tau_{\alpha_1-1}(\gamma_1)\cdots
\tau_{\alpha_{\ell}-1}(\gamma_{\ell}) +\ldots .
\end{equation*}
The leading term occurs in the contribution of the maximal set partition
$$\{1\} \cup \{2\} \cup \ldots \cup \{\ell\} = \{1,2,\ldots, \ell\}$$
in $\ell$ parts, see \cite[Section 7]{PPDC}. 
In case $\alpha=1^\ell$ has all part equal to 1, the leading term is the entire formula,
\begin{equation*}
\overline{\tau_{0}(\gamma_1)\cdots
\tau_{0}(\gamma_{\ell})}
=  \tau_{0}(\gamma_1)\cdots
\tau_{0}(\gamma_{\ell})\, .
\end{equation*}

In the presence of odd cohomology, a natural sign
must be included in \eqref{mqq23}. We may write 
set partitions $P$ of $\{1,\ldots, \ell\}$ indexing 
the sum on the right side of \eqref{mqq23} 
as
$$S_1\cup \ldots \cup S_{|P|} = \{1,\ldots, \ell\}.$$
The parts $S_i$ of $P$ are unordered, but we choose an ordering
for each $P$.
We then 
obtain a permutation of $\{1, \ldots, \ell\}$
by moving the elements to the ordered parts $S_i$ (and
respecting the original order in each group).
The permutation, in turn, determines a sign $\sigma(P)$
determined by the anti-commutation of the associated
odd classes. We then write
\begin{equation*}
\overline{\tau_{\alpha_1-1}(\gamma_1)\cdots
\tau_{\alpha_{\ell}-1}(\gamma_{\ell})}
=
\sum_{P \text{ set partition of }\{1,\ldots,\ell\}}\ (-1)^{\sigma(P)}
\prod_{S_i\in P}\ \sum_{\widehat{\alpha}}\tau_{\widehat{\alpha}}(\widetilde{\mathsf{K}}_{\alpha_{S_i},\widehat{\alpha}}
\cdot\gamma_{S_i}) \ .
\end{equation*}
The descendent 
$\overline{\tau_{\alpha_1-1}(\gamma_1)\cdots
\tau_{\alpha_{\ell}-1}(\gamma_{\ell})}$ is easily seen to have the
same commutation rules with respect to odd
cohomology as ${\tau_{\alpha_1-1}(\gamma_1)\cdots
\tau_{\alpha_{\ell}-1}(\gamma_{\ell})}$.

To state the descendent correspondence proposed in
\cite{PPDC} for all nonsingular projective 3-folds $X$, the basic degree
$$d_\beta = \int_{\beta} c_1(X) \ \in \mathbb{Z}$$
associated to the class $\beta\in H_2(X,\mathbb{Z})$ will be required.

\begin{Conjecture}
\label{ttt222} 
For $\gamma_i \in H^{*}(X,\mathbb{Q})$, we have
\begin{multline*}
(-q)^{-d_\beta/2}\ZZ_{\mathsf{P}}\Big(X;q\ \Big|  
{\tau_{\alpha_1-1}(\gamma_1)\cdots
\tau_{\alpha_{\ell}-1}(\gamma_{\ell})}
\Big)_\beta \\ =
(-iu)^{d_\beta}\ZZ'_{\mathsf{GW}}\Big(X;u\ \Big|   
\ \overline{\tau_{\alpha_1-1}(\gamma_1)\cdots
\tau_{\alpha_{\ell}-1}(\gamma_{\ell})}\ 
\Big)_\beta 
\end{multline*}
under the variable change $-q=e^{iu}$.
\end{Conjecture}

By Conjecture \ref{111}, the stable pairs descendent series
on the left is expected to be a rational function in $q$, so the change
of variables is well-defined.

If $X$ is a nonsingular quasi-projective toric 3-fold,
all terms of the descendent correspondence
have $\T$-equivariant interpretations.
We take the equivariant K\"unneth decomposition in \eqref{j77833},
and the equivariant Chern classes $c_i(T_X)$ with respect to the
canonical $\T$-action on $T_X$ in \eqref{h3492}.
The toric case is proven in \cite{PPDC}.

\vspace{10pt}
\noindent{\bf Toric correspondence.} {\em
For $\gamma_i \in H^{*}_{\mathbf{T}}(X,\mathbb{Q})$,
we have
\begin{multline*}
(-q)^{-d_\beta/2}\ZZ_{\mathsf{P}}\Big(X;q\ \Big|  
{\tau_{\alpha_1-1}(\gamma_1)\cdots
\tau_{\alpha_{\ell}-1}(\gamma_{\ell})}
\Big)^\T_\beta \\ =
(-iu)^{d_\beta}\ZZ'_{\mathsf{GW}}\Big(X;u\ \Big|   \
\overline{\tau_{\alpha_1-1}(\gamma_1)\cdots
\tau_{\alpha_{\ell}-1}(\gamma_{\ell})}\
\Big)^\T_\beta 
\end{multline*}
under the variable change $-q=e^{iu}$ 
for all nonsingular quasi-projective toric 3-folds $X$.}

\subsection{Complete intersections}
Let $X$ be a Fano or Calabi-Yau complete intersection 
in a product of projective spaces,
$$ X \subset \PP^{n_1} \times \cdots \times \PP^{n_m}\ .$$
The main result of the paper is the proof of the descendent
correspondence for even classes.

\begin{Theorem}
\label{qqq111} 
Let $X$ be a Fano or Calabi-Yau complete intersection 3-fold in
a product of projective spaces, and let
$\gamma_i \in H^{2*}(X,\mathbb{Q})$ be even classes. Then,
$$\ZZ_{\mathsf{P}}\Big(X;q\ \Big|  
{\tau_{\alpha_1-1}(\gamma_1)\cdots
\tau_{\alpha_{\ell}-1}(\gamma_{\ell})}
\Big)_\beta \ \in \mathbb{Q}(q)\ ,$$
and we have the correspondence
\begin{multline*}
(-q)^{-d_\beta/2}\ZZ_{\mathsf{P}}\Big(X;q\ \Big|  
{\tau_{\alpha_1-1}(\gamma_1)\cdots
\tau_{\alpha_{\ell}-1}(\gamma_{\ell})}
\Big)_\beta \\ =
(-iu)^{d_\beta}\ZZ'_{\mathsf{GW}}\Big(X;u\ \Big|   
\ \overline{\tau_{\alpha_1-1}(\gamma_1)\cdots
\tau_{\alpha_{\ell}-1}(\gamma_{\ell})}\ 
\Big)_\beta 
\end{multline*}
under the variable change $-q=e^{iu}$.
\end{Theorem}

If we specialize Theorem \ref{qqq111} to the case
where all descendents are primary or stationary,
we obtain the explicit correspondence conjectured
first in \cite{MNOP2} for the Donaldson-Thomas theory of ideal sheaves.

\begin{Corollary} \label{ccc000}
Let $X$ be a Fano or Calabi-Yau complete intersection 3-fold in
a product of projective spaces, and let
$\gamma_i \in H^{2*}(X,\mathbb{Q})$ be even classes of positive
degree. Then,
$$\bZ_{\mathsf{P}}\left(X;q \
\Bigg| \ \prod_{i=1}^r {\tau}_0(\gamma_{i})
 \prod_{j=1}^s {\tau}_{k_j}(\mathsf{p}) \right)_{\beta} \ \in \mathbb{Q}(q)\ ,$$
and we have the correspondence
\begin{multline*}
(-q)^{-d_\beta/2}\ \bZ_{\mathsf{P}}\left(X;q \
\Bigg| \ \prod_{i=1}^r {\tau}_0(\gamma_{i})
 \prod_{j=1}^s {\tau}_{k_j}(\mathsf{p}) \right)_{\beta}=\\
(-iu)^{d_\beta} (iu)^{-\sum k_j}\ 
 \bZ'_{\mathsf{GW}}\left(X;u \ \Bigg| \ \prod_{i=1}^r \tau_0(\gamma_{i}) 
\prod_{j=1}^s {\tau}_{k_j}(\mathsf{p})  \right)_{\beta} \ 
\end{multline*}
under the variable change $-q=e^{iu}$.
\end{Corollary}

If we specialize Theorem \ref{qqq111} further to the Calabi-Yau case
(with no descendent insertions), we obtain the following result.

\begin{Corollary}
\label{ccc111} 
Let $X$ be a Calabi-Yau complete intersection 3-fold in
a product of projective spaces. Then,
$$\ZZ_{\mathsf{P}}\Big(X;q 
\Big)_\beta \ \in \mathbb{Q}(q)\ ,$$
and we have the correspondence
$$
\ZZ_{\mathsf{P}}\Big(X;q\Big)_\beta  =
\ZZ'_{\mathsf{GW}}\Big(X;u\Big)_\beta 
$$
under the variable change $-q=e^{iu}$.
\end{Corollary}

Corollary \ref{ccc111} together with the DT/PT correspondence
proven by Toda \cite{Toda} and Bridgeland \cite{Bridge} implies the
original GW/DT correspondence \cite{MNOP1} in case $X$ is a Calabi-Yau
complete intersection in a product of projective spaces.

\subsection{BPS counts}
For complete intersection Calabi-Yau 3-folds, Theorem \ref{qqq111}
is closely related to 
 the BPS structure conjectured by Gopakumar and Vafa \cite{GV}
in 1998. 


The method of \cite{GV} was to consider limits of type IIA string
theory which may be conjecturally analyzed in M-theory.
A remarkable proposal was made in \cite{GV} for the form of the
Gromov-Witten potential $\mathsf{F}^X$ of a Calabi-Yau 3-fold $X$.
Let $$\mathsf{F}^X(u,v)= 
\sum_{g\geq 0} u^{2g-2} \mathsf{F}^X_g(v),\ \ \ \ \ 
\mathsf{F}^X_g(v)= 
\sum_{0 \neq \beta \in H_2(X, \mathbb{Z})} N^X_{g,\beta}\ v^\beta,$$
where $N^X_{g,\beta}$ is the (connected) genus $g$ Gromov-Witten
invariant of $X$ in curve class $\beta$. 
For each curve class $\beta\in H_2(X, \mathbb{Z})$ and
genus $g$, there is conjecturally an {\em integer} $n^X_{g,\beta}$ 
 counting BPS states
in the associated M-theory. For fixed $\beta$, the count
$n^X_{g,\beta}$ is conjectured to be
nonzero for only finitely many $g$.
The formula predicted in \cite{GV} is:
\begin{equation}
\label{ttttt}
\mathsf{F}^X(u,v)
= \sum_{g\geq 0} \sum_{\beta\neq 0} \ n^X_{g,\beta} u^{2g-2} \sum_{d>0} \frac{1}{d}\Big(
\frac{\text{sin}
({du/2})}{u/2}\Big)^{2g-2} v^{d\beta}.
\end{equation}
The BPS form \eqref{ttttt} places integrality constraints
on the Gromov-Witten invariants. 


 We can uniquely
{\em define} invariants $n_{g,\beta}^X \in \mathbb{Q}$
by \eqref{ttttt}. Neither the integrality nor the
vanishing of $n_{g,\beta}^X$ for sufficiently high $g$ is
then clear. As a corollary of Theorem \ref{qqq111},
we obtain the following result.


\begin{Corollary}
\label{ccc333} 
Let $X$ be a Calabi-Yau complete intersection 3-fold in
a product of projective spaces, and let $\beta \in H_2(X,\mathbb{Z})$: 
\begin{enumerate}
\item[(i)] After the variable change $-q=e^{iu}$, 
$${\mathsf{F}}^X_{\beta}(q)= \text{\em Coeff}_{v^\beta}\left[ \mathsf{F}^X
\right]\in \mathbb{Q}(q)$$
is a rational function invariant under $q\leftrightarrow q^{-1}$.

\item[(ii)] If, for all divisors $\tilde{\beta}|\beta$,
$n_{g,\tilde{\beta}}^X$ vanishes for all sufficiently large $g$,
then $$n_{g, \beta}^X \in \mathbb{Z}, \ \ \forall g\geq 0\ .$$
\end{enumerate} 
\end{Corollary}

Corollary \ref{ccc333} follows easily from Theorem 1 and the 
results of Section 3 of \cite{pt}. 
The rationality of part (i) is slightly weaker
than the full Gopakumar-Vafa predicted BPS form, but becomes
equivalent with the vanishing assumed in (ii).
A proof of the integrality of $n_{g,\beta}^X$ has been recently
claimed in \cite{IIPP}. The method is analytic but eventually
reduces the integrality to the local curves calculation of \cite{BP}.
The vanishing (ii) is open.

\subsection{Plan of the paper}
We will prove Theorem \ref{qqq111} via the degeneration
scheme established in \cite{mptop}. To control the
Gromov-Witten and stable pairs theories of Fano and
Calabi-Yau complete
intersections in products of projective spaces,
we must prove GW/P correspondences for
relative and descendent insertions in 
several simpler geometries.

Let $D\subset X$ be a nonsingular divisor in a nonsingular
3-fold $X$.
The first step in the proof of Theorem \ref{qqq111} is to
formulate a
GW/P descendent correspondence for the
relative geometry $X/D$. The interaction of the
descendents with the relative divisor is
explained in Section \ref{relth} with a 
full correspondence
proposed in Conjecture \ref{ttt444} of Section \ref{pwwf}.

The degeneration scheme of
\cite{mptop} requires the study of
 $\mathbf{P}^1$-bundles
$$\pi: \mathbf{P}_S \rightarrow S$$
over surfaces $S$ relative to a section of $\pi$
where $S$ is either 
\begin{enumerate}
\item[(i)] a toric surface,
\item[(ii)] a $K3$ surface, 
\item[(iii)] or a $\mathbf{P}^1$-bundle over a
higher genus curve $C$.
\end{enumerate}
Sections \ref{n3n3}-\ref{phgc} are devoted to the proofs of
descendent correspondences for the relative surface
geometries (i)-(iii).

The toric case (i) is studied in Section \ref{n3n3}.
For the $K3$ surface, the results of Section 8  of \cite{PPDC}
establish special cases. The required descendent
correspondence for $\mathbf{P}_{K3}$ is proven in
Section \ref{xxx1} after the
fully equivariant relative descendent correspondence
for the 3-fold cap is established.

The technically most difficult results concern
the surface geometries (iii). We study
higher genus curves by degeneration to genus 0.
The method requires establishing
correspondences  for special surface geometries
in Section \ref{xxx2} 
and the introduction of bi-relative residue
theories in Section \ref{xxx3}. The
odd cohomology of the higher genus curves, discussed
in Section \ref{phgc}, is
controlled by the strategy first employed in
\cite{vir}.

The degeneration scheme and the proof
of Theorem \ref{qqq111} is presented in Section \ref{laman}.
In fact our methods are valid in any context in which
the Fano or Calabi-Yau 3-folds can be efficiently degenerated.
As an example, the GW/P correspondence for the
Enriques Calabi-Yau is discussed in Section \ref{ecy}.

The application of relative and descendent methods to
the GW/P correspondences for non-toric Calabi-Yau geometries
has been one of the major motivations for our work in
\cite{part1,PP2,PPstat,PPDC}. The recent proof \cite{kkv} of the
full Katz-Klemm-Vafa conjecture for the Gromov-Witten theory
of $K3$ surfaces uses the GW/P correspondences for 
non-toric hypersurface Calabi-Yau 3-folds established here.

\subsection{Acknowledgments}
Discussions with J. Bryan, S. Katz, D. Maulik, A. Oblomkov, A. Okounkov, 
R. Thomas, and Y. Toda 
about  Gromov-Witten theory, stable pairs,
and descendent invariants
played an important role. We thank the referee for suggesting several
improvements. 
R.P. was partially supported by NSF grant
DMS-1001154, SNF grant 200021-143274, and ERC grant
AdG-320368-MCSK. A.P. was supported by a NDSEG graduate fellowship.
The research reported here was pursued during several visits of A.P.
to ETH Z\"urich.

\section{Relative theories}\label{relth}
\subsection{Definitions}
Let $X$ be a nonsingular 3-fold with a nonsingular divisor
$D\subset X$. Relative Gromov-Witten and 
relative stable pairs
theories enumerate curves with specified tangency to
the divisor $D$. See \cite{L,LW,MNOP2,part1} for a technical discussion
of relative theories.

In Gromov-Witten theory, relative conditions  are
represented by a partition $\mu$ of the integer
$
\int_\beta [D],
$
each part $\mu_i$ of which is marked
by a cohomology class $\delta_i\in H^*(D,\mathbb{Z})$,
\begin{equation}
\label{mm33}
\mu=( (\mu_1,\delta_1), \ldots, (\mu_\ell,\delta_\ell))\, .
\end{equation}
The numbers $\mu_i$ record
the multiplicities of intersection with $D$
while the cohomology labels $\delta_i$ record where the tangency occurs.
More precisely, let $\oM_{g,r}'(X/D,\beta)_\mu$ be the
moduli space of stable relative maps with tangency conditions
$\mu$ along $D$. To impose the full boundary condition,
we 
pull-back the classes $\delta_i$
via the evaluation maps
\begin{equation}\label{gtth341}
\oM_{g,r}'(X/D,\beta)_\mu \to D
\end{equation}
at the points of tangency.
Also, the tangency points are
considered to be unordered.{\footnote{The evaluation
maps are well-defined only after ordering the points.
We define the theory first with ordered tangency points. 
The unordered theory is then defined by dividing by the
automorphisms of the cohomology weighted partition $\mu$.}}

In the stable pairs theory, the relative moduli space admits a natural
morphism to the Hilbert scheme of
$d$ points in $D$,
$$P_n(X/D,\beta) \to \Hilb(D,\int_\beta [D])\ .$$ 
Cohomology classes on $\Hilb(D,\int_\beta [D])$ may 
thus
be pulled-back to the relative moduli space. We will work in
the \emph{Nakajima basis} of $H^*(\Hilb(D,\int_\beta [D]),\mathbb{Q})$ indexed
by a partition $\mu$ of $\int_\beta [D]$
labeled by cohomology classes of $D$ as in \eqref{mm33}. For example, the
class
$$
\left.\big|\mu\rang \in H^*(\Hilb(D,\int_\beta [D]),\mathbb{Q})\,,
$$
with all cohomology labels equal to the identity,
 is $\prod \mu_i^{-1}$ times
the Poincar\'e dual of the closure of the subvariety formed by unions of
schemes of length
$$
\mu_1,\dots, \mu_{\ell(\mu)}
$$
supported at $\ell(\mu)$ distinct points of $D$.

The conjectural relative GW/P correspondence for primary
fields \cite{MNOP2}
equates the partition functions of the theories.

\begin{Conjecture}\label{htht}
For $\gamma_i \in H^{*}(X,\mathbb{Q})$,
we have
\begin{multline*}
(-q)^{-d_\beta/2}\,
\bZ_{\mathsf{P}}
\Big( X/D;q\ \Big|\
\tau_0(\gamma_1)\cdots\tau_0(\gamma_r)\, \Big| \mu \Big)_\beta
=\\
(-iu)^{d_\beta+\ell(\mu)-|\mu|}\,  \bZ'_{\mathsf{GW}}
\Big( X/D;u\ \Big|\
\tau_0(\gamma_1)\cdots\tau_0(\gamma_r) \, \Big| \mu  \Big )_\beta
\,,
\end{multline*}
after the change of variables $e^{iu}=-q$. 
\end{Conjecture}

As before,  $\bZ_{\mathsf{P}}
\left(\left. X/D;q\ \right|
\tau_0(\gamma_1)\cdots\tau_0(\gamma_r)\, \big| \mu \right)_\beta$
is
conjectured to be a rational function of $q$.  Conjecture \ref{htht} is
made for every boundary condition \eqref{mm33}.

\subsection{Diagonal classes}\label{diagclas}
To state our results for the Gromov-Witten/Pairs descendent
correspondence in the relative case, a discussion of 
diagonal classes is required.

For the absolute geometry $X$, the product $X^s$ naturally
parameterizes $s$ ordered (possibly coincident) points on $X$.
For the relative geometry $X/D$, the moduli space of
$s$ ordered (possibly coincident) points 
$$(p_1,\ldots, p_s) \in X/D$$
is a more  subtle space.
The points are not allowed to lie on the relative divisor $D$.
When the points approach $D$, the target $X$ degenerates.
The resulting moduli space $(X/D)^s$ is a nonsingular variety.
Let
$$\Delta_{\mathsf{rel}} \subset (X/D)^s$$
consisting of the small diagonal where all the points $p_i$ 
are coincident.
As a variety, $\Delta_{\mathsf{rel}}$ is isomorphic to $X$.

The space $(X/D)^s$
is a special case of well-known constructions in
relative geometry. 
For example, $(X/D)^2$ consists of
6 strata:

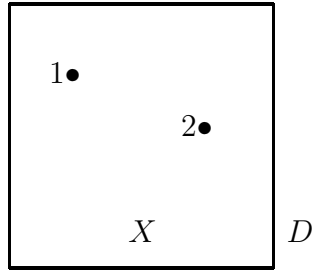
\begin{picture}(150,150)(-120,-5)
\thicklines
\put(10,10){\line(1,0){100}}
\put(10,110){\line(1,0){100}}
\put(10,10){\line(0,1){100}}
\put(110,10){\line(0,1){100}}
\put(25,80){$1\bullet$}
\put(75,60){$2\bullet$}
\put(55,20){$X$}
\put(115,20){$D$}
\end{picture}

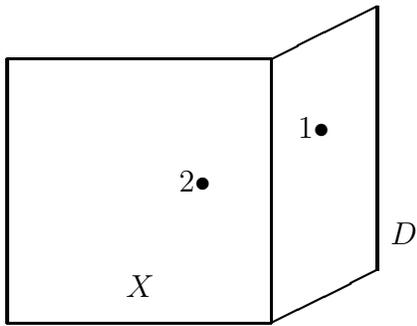
\begin{picture}(150,150)(0,-5)
\thicklines
\put(10,10){\line(1,0){100}}
\put(10,110){\line(1,0){100}}
\put(10,10){\line(0,1){100}}
\put(110,10){\line(0,1){100}}

\put(110,110){\line(2,1){40}}
\put(110,10){\line(2,1){40}}
\put(150,30){\line(0,1){100}}
\put(155,40){$D$}

\put(120,80){$1\bullet$}
\put(75,60){$2\bullet$}
\put(55,20){$X$}

\put(210,10){\line(1,0){100}}
\put(210,110){\line(1,0){100}}
\put(210,10){\line(0,1){100}}
\put(310,10){\line(0,1){100}}

\put(310,110){\line(2,1){40}}
\put(310,10){\line(2,1){40}}
\put(350,30){\line(0,1){100}}
\put(355,40){$D$}

\put(225,80){$1\bullet$}
\put(325,60){$2\bullet$}
\put(255,20){$X$}

\end{picture}

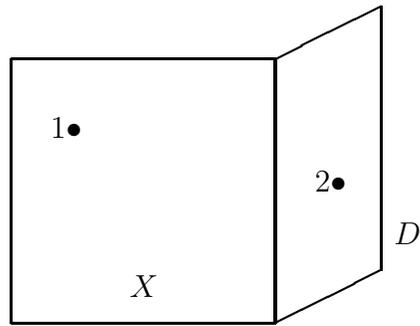
\begin{picture}(150,150)(-100,-5)
\thicklines
\put(10,10){\line(1,0){100}}
\put(10,110){\line(1,0){100}}
\put(10,10){\line(0,1){100}}
\put(110,10){\line(0,1){100}}

\put(110,110){\line(2,1){40}}
\put(110,10){\line(2,1){40}}
\put(150,30){\line(0,1){100}}
\put(155,40){$D$}

\put(120,80){$1\bullet$}
\put(120,50){$2\bullet$}
\put(55,20){$X$}

\end{picture}

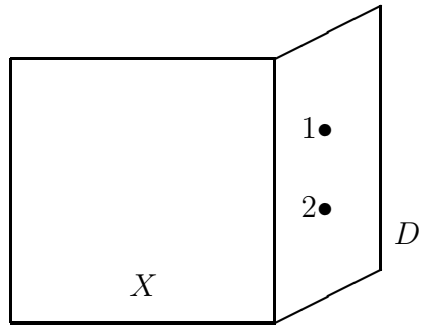
\begin{picture}(150,150)(-80,-5)
\thicklines
\put(10,10){\line(1,0){100}}
\put(10,110){\line(1,0){100}}
\put(10,10){\line(0,1){100}}
\put(110,10){\line(0,1){100}}

\put(110,110){\line(2,1){40}}
\put(110,10){\line(2,1){40}}
\put(150,30){\line(0,1){100}}

\put(150,30){\line(2,1){40}}
\put(150,130){\line(2,1){40}}
\put(190,50){\line(0,1){100}}
\put(195,60){$D$}

\put(160,80){$1\bullet$}
\put(120,50){$2\bullet$}
\put(55,20){$X$}

\thicklines

\end{picture}

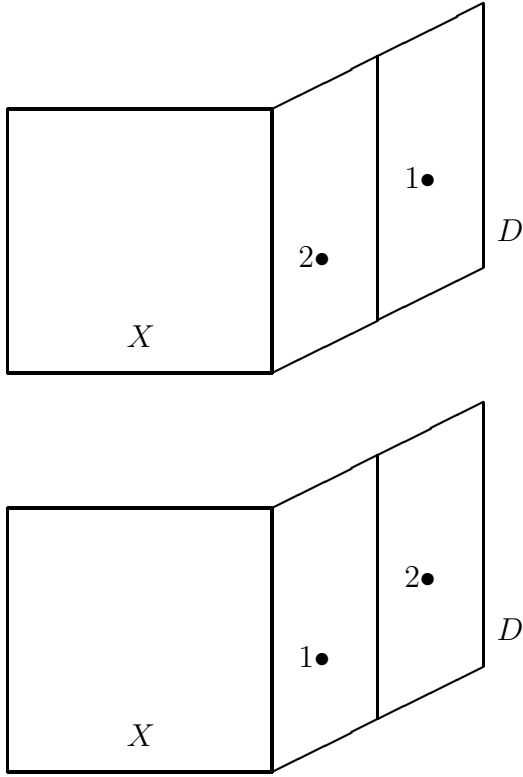
\begin{picture}(150,150)(-80,-5)
\thicklines
\put(10,10){\line(1,0){100}}
\put(10,110){\line(1,0){100}}
\put(10,10){\line(0,1){100}}
\put(110,10){\line(0,1){100}}

\put(110,110){\line(2,1){40}}
\put(110,10){\line(2,1){40}}
\put(150,30){\line(0,1){100}}

\put(150,30){\line(2,1){40}}
\put(150,130){\line(2,1){40}}
\put(190,50){\line(0,1){100}}
\put(195,60){$D$}

\put(160,80){$2\bullet$}
\put(120,50){$1\bullet$}
\put(55,20){$X$}

\thicklines

\end{picture}

\noindent As a variety, $(X/D)^2$ is the blow-up of $X^2$ along $D^2$.
And, $\Delta_{\mathsf{rel}} \subset (X/D)^2$ is the strict transform
of the standard diagonal.

Select a subset $S$ of cardinality $s$
from the $r$ markings of the moduli space
of maps.
Just as $\oM_{g,r}'(X,\beta)$
admits a canonical evaluation to $X^s$ via
the selected markings, 
the moduli space $\oM_{g,r}'(X/D,\beta)_\mu$
admits a canonical evaluation
 $$\text{ev}_S: \oM_{g,r}'(X/D,\beta)_\mu \rightarrow (X/D)^s , $$
well-defined by the definition of a relative stable
map (the markings never map to the relative divisor).
The class 
$$\text{ev}_S^*(\Delta_{\mathsf{rel}}) \in H^*(\oM_{g,r}'(X/D,\beta)_\mu)$$
plays a crucial role in the relative descendent
correspondence.

By forgetting the relative structure, we obtain a projection
$$\pi:(X/D)^s \rightarrow X^s\ .$$
The product contains the standard diagonal $\Delta\subset X^s$. However,
$$\pi^*(\Delta) \neq \Delta_{\mathsf{rel}}\ .$$
The former has more components in the relative boundary if
$D\neq \emptyset$.

\subsection{Relative descendent correspondence} \label{pwwf}
Let $\widehat{\alpha}$ be a partition
of length $\widehat{\ell}$.
Let $\Delta_{\mathsf{rel}}$ be the cohomology class of the
small diagonal in  $(X/D)^{\widehat{\ell}}$.
For a cohomology class $\gamma$ of $X$, let
$$\gamma\cdot \Delta_{\mathsf{rel}} \in H^*((X/D)^{\widehat{\ell}},\mathbb{Q}).$$
We define the relative descendent insertion $\tau_\alpha(\gamma)$ by
\begin{equation}\label{j9994}
\tau_{\widehat{\alpha}}(\gamma)= 
\psi_1^{\widehat{\alpha}_1-1} \cdots \psi_{\hat{\ell}}^{\widehat{\alpha}_{\hat{\ell}}-1} \cdot
\text{ev}^*_{1,\ldots,\hat{\ell}} ( \gamma\cdot \Delta_{\mathsf{rel}}) \ .
\end{equation}
In case, $D=\emptyset$, definition \eqref{j9994}
specializes to \eqref{j77833}.

Let $\Omega_X[D]$ denote the locally free sheaf of 
differentials with logarithmic poles along $D$.
Let $$T_{X}[-D] = \Omega_{X}[D]^{\ \vee}$$
denote the dual sheaf of tangent fields with logarithmic
zeros.

For the relative geometry $X/D$, we let the coefficients of
$\widetilde{\mathsf{K}}$ act on the cohomology of $X$ via the
substitution
$$c_i= c_i(T_{X}[-D])$$
instead of the substitution $c_i=T_X$ used in the absolute case.
Then, we define 
\begin{equation} \label{gtte4}
\overline{\tau_{\alpha_1-1}(\gamma_1)\cdots
\tau_{\alpha_{\ell}-1}(\gamma_{\ell})}
=
\sum_{P \text{ set partition of }\{1,\ldots,l\}}\ \prod_{S\in P}\ \sum_{\widehat{\alpha}}\tau_{\widehat{\alpha}}(\widetilde{\mathsf{K}}_{\alpha_S,\widehat{\alpha}}\cdot\gamma_S) \ 
\end{equation}
as before via \eqref{j9994} instead of 
\eqref{j77833}.

Definition \eqref{gtte4} is for  even classes $\gamma_i$.
In the presence of odd $\gamma_i$, a sign has to be included
exactly as in the absolute case.

\begin{Conjecture}
\label{ttt444} 
For $\gamma_i \in H^{*}(X,\mathbb{Q})$,
we have 
\begin{multline*}
(-q)^{-d_\beta/2}\ZZ_{\mathsf{P}}\Big(X/D;q\ \Big|  
{\tau_{\alpha_1-1}(\gamma_1)\cdots
\tau_{\alpha_{\ell}-1}(\gamma_{\ell})} \ \Big| \ \mu
\Big)_\beta \\ =
(-iu)^{d_\beta+\ell(\mu)-|\mu|}\ZZ'_{\mathsf{GW}}\Big(X/D;u\ \Big|   
\ \overline{\tau_{a_1-1}(\gamma_1)\cdots
\tau_{\alpha_{\ell}-1}(\gamma_{\ell})}
\ \Big| \ \mu\Big)_\beta 
\end{multline*}
under the variable change $-q=e^{iu}$.
\end{Conjecture}

In addition, the stable pairs descendent series on the left is
conjectured to be a rational function in $q$, so the change
of variables is well-defined.
Conjecture \ref{ttt444}
is also well-defined in the equivariant case with respect
to a group action on $X$ preserving the relative divisor $D$.
Definition \eqref{j9994} lifts canonically to
the equivariant cohomology. The coefficients of
$\widetilde{\mathsf{K}}$ act on the 
equivariant cohomology of $X$ via the
equivariant Chern classes
$c_i(T_{X}[-D])$.

\subsection{Degeneration}
There is no difficulty in proving the compatibility of 
Conjectures \ref{ttt222} and \ref{ttt444} with respect to the
degeneration formula. In fact, 
both definition \eqref{j9994} and
the replacement of
$T_X$ by $T_X[-D]$ are required for compatibility with
degeneration formula.
Definition \eqref{j9994} canonically lifts the diagonal
splittings which occur in the correspondence for the
absolute case.

The log tangent bundle arises for the following reason.
Let
$$
\pi:\mathfrak{X}\to B
$$
be a nonsingular $4$-fold fibered over an irreducible
nonsingular base curve $B$. Let $X$ be a nonsingular
fiber, and let
$$
X_1 \cup_{D} X_2
$$
be a reducible special fiber consisting of two nonsingular
$3$-folds intersecting transversally along a nonsingular
surface $D$.
Let $T_{\mathfrak{X}}[-X_1-X_2]$ be the tangent bundle 
of the total space
$\mathfrak{X}$ with logarithmic zeros along $X_1\cup_D X_2$.
The basic restriction property 
$$c( T_{\mathfrak{X}}[-X_1-X_2])|_{X_i} = c(T_{X_i}[-D])$$
holds on the special fiber.
The Chern classes of the tangent
bundle of a general fiber of
$\pi$ therefore are extended by the Chern classes 
of the log tangent bundle of the special fiber.

Since the compatibility with degeneration will play an 
important role in the paper, we state the result (a
formal consequence of the usual degeneration formula
in Gromov-Witten theory \cite{LR,L}).

\pagebreak
\noindent{\bf Compatibility with degeneration.}
{\em Let $\gamma_1, \ldots,\gamma_\ell$ be cohomology classes
on the total space $\mathfrak{X}$. We have
\begin{multline*}
\ZZ'_{\mathsf{GW}}\Big(X \Big|   
\ \overline{\tau_{a_1-1}(\gamma_1)\cdots
\tau_{\alpha_{\ell}-1}(\gamma_{\ell})}
\ \Big)_\beta 
=\\
{\mathlarger{\mathlarger{\mathlarger{\sum}}}}\  
\ZZ'_{\mathsf{GW}}\Big(X_1/D \Big|   
\ \overline{\prod_{i\in I_1}\tau_{a_i-1}(\gamma_i)}
\ \Big| \ \mu\Big)_{\beta_1}\ \zz(\mu) u^{2\ell(\mu)}\\ 
\cdot \ZZ'_{\mathsf{GW}}\Big(X_2/D \Big|   
\ \overline{\prod_{i\in I_2}\tau_{a_i-1}(\gamma_i)}
\ \Big| \ \mu^\vee\Big)_{\beta_2}\ .
\end{multline*}
The sum is over all marking distributions and curve class splittings 
$$I_1\cup I_2 = \{1, \ldots, \ell\}, \ \ \ \beta=\beta_1+\beta_2,$$
and all boundary conditions $\mu$ along $D$.}

\vspace{10pt}
The boundary conditions $\mu$ are partitions weighted by 
elements of a fixed basis of $H^*(D,\mathbb{Q})$.
The boundary condition
$\mu^\vee$ has the same parts as $\mu$ but with weights
given by dual elements of the dual{\footnote{With
respect to the intersection pairing.}} basis of $H^*(D,\mathbb{Q})$. 
The gluing factor is defined by
\begin{equation}
\zz(\mu) = \prod_{i=1}^{\ell(\mu)} \mu_i \cdot |\text{Aut}(\mu)|
\label{k233}
\end{equation}
The first factor in \eqref{k233} is simply the product of the
parts of $\mu$. The second term is the order
of the symmetry group of $\mu$ {\em as a weighted partition}.

\subsection{Relative results}\label{rr}
The first results about the descendent correspondence in the
relative case concern projective bundles over a nonsingular
surface $S$.
Let $$L_0, L_\infty\rightarrow S$$
be two line bundles. 
The projective bundle{\footnote{We always follow the convention
of projectivization by 1-dimension subspaces.}}
$$\mathbf{P}_S= \mathbf{P}(L_0 \oplus L_\infty) \rightarrow S$$
admits sections 
$$ S_i= \mathbf{P}(L_i) \subset \mathbf{P}_S\ .$$
We will establish 
the 
relative descendent correspondence 
of Conjecture \ref{ttt444}
for $\mathbf{P}_S/S_\infty$ and $\mathbf{P}_S/ S_0 \cup S_\infty$
when $S$ is a toric surface.

There is a canonical $\com^*$-action on $\mathbf{P}_S$ by scaling
the coordinates on the $\mathbf{P}^1$-fibers,
\begin{equation}\label{k45p}
\xi\cdot [l_0,l_\infty] = [\xi l_0, l_\infty], \ \ \ \ \ \xi\in \com^* \ .
\end{equation}
We denote by $t$ the generator of the equivariant cohomology
of $\com^*$.
We will prove Conjecture \ref{ttt444} for $\mathbf{P}_S$
equivariantly with respect to the fiberwise $\com^*$-action \eqref{k45p}.

\begin{Theorem}
\label{ttt999} 
Let $S$ be a nonsingular projective toric surface.
For classes $\gamma_i \in H^{*}_{\com^*}(\mathbf{P}_S,\mathbb{Q})$,
we have 
$$\ZZ_{\mathsf{P}}\Big(\mathbf{P}_S/S_\infty   ;q\ \Big|  
{\tau_{\alpha_1-1}(\gamma_1)\cdots
\tau_{\alpha_{\ell}-1}(\gamma_{\ell})} \ \Big| \ \mu
\Big)_\beta^{\com^*}\in \mathbb{Q}(q,t)$$
and the correspondence
\begin{multline*}
(-q)^{-d_\beta/2}\ZZ_{\mathsf{P}}\Big(\mathbf{P}_S/S_\infty   ;q\ \Big|  
{\tau_{\alpha_1-1}(\gamma_1)\cdots
\tau_{\alpha_{\ell}-1}(\gamma_{\ell})} \ \Big| \ \mu
\Big)_\beta^{\com^*} \\ =
(-iu)^{d_\beta+\ell(\mu)-|\mu|}\ZZ'_{\mathsf{GW}}\Big(\mathbf{P}_S/S_\infty
;u\ \Big|   
\ \overline{\tau_{a_1-1}(\gamma_1)\cdots
\tau_{\alpha_{\ell}-1}(\gamma_{\ell})}
\ \Big| \ \mu\Big)_\beta^{\com^*} 
\end{multline*}
under the variable change $-q=e^{iu}$.
\end{Theorem}

The parallel result holds when the projective
bundle geometry is taken relative to both sections.

\begin{Theorem}
\label{ttt9999} 
Let $S$ be a nonsingular projective toric surface.
Consider the relative geometry $\mathbf{P}_S/S_0 \cup S_\infty$.
For $\gamma_i \in H^{*}_{\com^*}(\mathbf{P}_S,\mathbb{Q})$,
we have 
$$\ZZ_{\mathsf{P}}\Big(\nu\  \Big|  
{\tau_{\alpha_1-1}(\gamma_1)\cdots
\tau_{\alpha_{\ell}-1}(\gamma_{\ell})} \ \Big| \ \mu
\Big)_\beta^{\com^*}\in \mathbb{Q}(q,t)$$
and the correspondence
\begin{multline*}
(-q)^{-d_\beta/2}\ZZ_{\mathsf{P}}\Big( \nu\ \Big|  
{\tau_{\alpha_1-1}(\gamma_1)\cdots
\tau_{\alpha_{\ell}-1}(\gamma_{\ell})} \ \Big| \ \mu
\Big)_\beta^{\com^*} \\ =
(-iu)^{d_\beta+\ell(\nu)-|\nu|
+\ell(\mu)-|\mu|
}\ZZ'_{\mathsf{GW}}\Big(\nu \ \Big|   
\ \overline{\tau_{a_1-1}(\gamma_1)\cdots
\tau_{\alpha_{\ell}-1}(\gamma_{\ell})}
\ \Big| \ \mu\Big)_\beta^{\com^*} 
\end{multline*}
under the variable change $-q=e^{iu}$.
\end{Theorem}

Theorems \ref{ttt999} and \ref{ttt9999} 
will be proven in Section \ref{n3n3}.
We will use the 
absolute toric correspondence and the relative
projective bundle geometries to prove Theorem \ref{qqq111} in Section \ref{laman}.

\section{Proofs of Theorems \ref{ttt999} and \ref{ttt9999}}
\label{n3n3}

\subsection{Conventions}
Localization with respect to the fiberwise $\com^*$-action
will play a central role in the proofs of the
descendent correspondence for the relative projective
bundle geometries. 
We will use the localization formula for $\mathbf{P}_S/S_\infty$
in a capped form following \cite{moop,PPDC}.
We review the constructions here. 

Since the fiberwise $\com^*$ acts trivially on $S$, we have
the simple characterization
$$H^*_{\com^*}(S,\mathbb{Q}) =
H^*(S,\mathbb{Q})\otimes_{\mathbb{Q}}{{\mathbb{Q}}}[t]\ .$$
Via the $\com^*$-invariant projection 
$$\pi: \mathbf{P}_S \rightarrow S\ ,$$
there is a canonical pull-back
$$\pi^*: H^*_{\com^*}(S,\mathbb{Q}) \rightarrow
H^*_{\com^*}(\mathbf{P}_S,\mathbb{Q})\ . $$

The localized
$\com^*$-equivariant cohomology of $\mathbf{P}_S$ is
a free module of rank 2 over the localized 
$\com^*$-equivariant cohomology of $S$,
\begin{equation}\label{htyy7}
H^*_{\com^*}(\mathbf{P}_S,\mathbb{Q})_{\frac{1}{t}}
\stackrel{\sim}{=}
H^*_{\com^*}(S,\mathbb{Q})_{\frac{1}{t}} \cdot[S_0] \ \oplus 
H^*_{\com^*}(S,\mathbb{Q})_{\frac{1}{t}}
\cdot[S_\infty].
\end{equation}
The normal bundles of $S_0$ and $S_\infty$ in $\mathbf{P}_S$ 
are
$$\mathcal{N}^*=L_\infty \otimes L_0^* \ \ \ \text{and} \ \ \
\mathcal{N}=L_0 \otimes L_\infty^*$$
respectively.
Under the isomorphism \eqref{htyy7}, we have
\begin{equation}\label{krr4}
\pi^*(\gamma) \ = \   \frac{\gamma}{-t-N}[S_0] + 
\frac{\gamma}{t+N}[S_\infty], \ \ \ \ \
\gamma\in H^*_{\com^*}(S,\mathbb{Q})
\end{equation}
where $N=c_1(\mathcal{N})\in H^*(S,\mathbb{Q})$.
Equation \eqref{krr4} is the Atiyah-Bott localization
formula for the fiberwise $\com^*$-action on $\mathbf{P}_S$.

Let $L\in H_2({S},\mathbb{Z})$ be a fixed ample polarization of $S$.
We will measure the $S$-degree of curve classes
on $\mathbf{P}_S$ via $\pi$ push-forward followed by intersection with
$L$,
$$
L_\beta = \int_S L\cdot \pi_*(\beta)\ .$$
Let $[\mathbf{P}]\in H_2(\mathbf{P}_S,\mathbb{Z})$ be the 
class of a fiber of $\pi$.
We have an exact sequence
\begin{equation}\label{fgbb}
0\longrightarrow \mathbb{Z}[\mathbf{P}] \longrightarrow
H_2(\mathbf{P}_S,\mathbb{Z}) \stackrel{\pi_*}{\longrightarrow}
H_2(S,\mathbb{Z})\longrightarrow 0\ .
\end{equation}
The only effective curve classes with $L_\beta=0$ are 
multiples of $[\mathbf{P}]$.

The inclusions of $S$ via $S_0$ and $S_\infty$  determine
two sections of the surjection in \eqref{fgbb}.
Let 
$$\text{Eff}(S_0),\ \text{Eff}(S_\infty) \subset H_2(\mathbf{P}_S,\mathbb{Z})$$
denote the effective curve classes supported on $S_0$ and $S_\infty$
respectively.

\subsection{Log tangent bundle}
The definition of  the descendent correspondence
$$\tau_{\alpha_1-1}(\gamma_1) \cdots  \tau_{\alpha_\ell-1}(\gamma_\ell)
\ \ \ 
\mapsto \ \ \
\overline{\tau_{\alpha_1-1}(\gamma_1) \cdots  \tau_{\alpha_\ell-1}(\gamma_\ell)}\ 
$$
for the relative geometry $\mathbf{P}_S/ S_\infty$
requires 
the Chern classes of the log tangent bundle $T_{\mathbf{P}_S}[-S_\infty]$.

Similarly, for the relative
geometry $\mathbf{P}_S/ S_0 \cup S_\infty$, the Chern classes
of $T_{\mathbf{P}_S}[-S_0-S_\infty]$ are required.

\begin{Lemma}\label{fvv}
The total Chern classes  are
\begin{eqnarray*}
c(T_{\mathbf{P}_S}[-S_\infty])& = &  c(\pi^*T_{S})\cdot (1+[S_0])\ ,\\
c(T_{\mathbf{P}_S}[-S_0-S_\infty])&= & c(\pi^* T_{S})
\end{eqnarray*}
in the $\com^*$-equivariant cohomology of $\mathbf{P}_S$
for the fiberwise action.
\end{Lemma}

In both cases, the restriction of the Chern classes to $S_\infty$
involves only classes pulled-back from $S$ via $\pi$.
We leave the elementary derivation of Lemma \ref{fvv} to the reader.

\subsection{Capped localization}

\subsubsection{Capping over $S_0$}
Let $P_n(\mathbf{P}_S/S_\infty,\beta)_\mu$ be the moduli space
of stable pairs with boundary condition given by $\mu$.
Let $\alpha$ be a partition of positive size, and let
 $$\Gamma=(\gamma_1, \ldots, \gamma_\ell), \ \ \ 
\gamma_i \in H^*(S,\mathbb{Q})
$$
be a vector of cohomology classes. Let 
$$ \tau_\alpha(\Gamma_0)=
\tau_{\alpha_1-1}(\gamma_1[S_0]) \ldots \tau_{\alpha_\ell-1}(\gamma_\ell[S_0])$$
be the associated descendent insertion over $S_0$. 
We can study the partition functions
\begin{equation} \label{nwn3}
   \sum_n q^n
\int_{[P_n(\mathbf{P}_S/S_\infty,\beta)_\mu]^{vir}}
\tau_\alpha(\Gamma_0), \ \ \ 
\end{equation}
\begin{equation*}
\sum_g u^{2g-2}
\int_{[\overline{M}'_{g,\star}(\mathbf{P}_S/S_\infty,\beta)_\mu]^{vir}}
\overline{\tau_\alpha(\Gamma_0)}
\end{equation*}
via localization with respect to  the fiberwise $\com^*$-action.
Recall, $\overline{\tau_\alpha(\Gamma_0)}$ is defined by \eqref{gtte4}
and is a sum of terms.
For the stable maps moduli space, the number of markings
depends upon the summand of $\overline{\tau_\alpha(\Gamma_0)}$,
and is denoted by $\star$.

The stable pairs capped descendent over $S_0$ is a sum of {particular} 
localization contributions to \eqref{nwn3}.
Let 
$$ U_{n,\beta,\mu} \subset P_n(\mathbf{P}_S/S_\infty,\beta)_\mu$$
be the open locus corresponding to stable pairs
which {\em do not carry components of positive $S$-degree
in the rubber over $S_\infty$}.
The open set $U_{n,\beta,\mu}$ is $\com^*$-invariant and has 
compact $\com^*$-fixed locus. Indeed, the fixed locus
$$U_{n,\beta,\mu}^{\com^*} \subset U_{n,\beta,\mu}$$
consists precisely of the $\com^*$-fixed loci of
$P_n(\mathbf{P}_S/S_\infty,\beta)_\mu$ with no components of
positive $S$-degree in the rubber over $S_\infty$.
Unless the curve class $\beta$ is of the form 
\begin{equation}\label{pd59}
\beta= \beta_0 + |\mu|[{\mathbf{P}}], \ \ \ \beta_0\in \text{Eff}(S_0),
\end{equation}
the open set $U_{n,\beta,\mu}$ is empty.
The stable pairs capped descendent over $S_0$ is
\begin{equation}\label{pdd4}
\mathsf{C}^{\mathsf{P}}_0(\tau_\alpha(\Gamma_0), \beta)_\mu = 
\sum_n q^n
\int_{[U_{n,\beta,\mu}]^{vir}}
\tau_\alpha(\Gamma_0)
\ \ \ \in \mathbb{Q}[t,\frac{1}{t}]((q)),
\end{equation}
well-defined by $\com^*$-residues.{\footnote{We have presented
the definition of the stable pairs capped descendent $\mathsf{C}^{\mathsf{P}}_0(\tau_\alpha(\Gamma_0), \beta)_\mu$ to parallel as closely as
possible the definition of
the Gromov-Witten
capped descendent $\mathsf{C}^{\mathsf{GW}}_0(\tau_\alpha(\Gamma_0), 
\beta)_\mu$. Instead of considering
$P_n(\mathbf{P}_S/S_\infty,\beta)_\mu$ as a space (after a
fixed representative of the Nakajima basis element $|\mu\rangle$
is chosen), we could alternatively arrive at the {\em same}
definition of $\mathsf{C}^{\mathsf{P}}_0(\tau_\alpha(\Gamma_0), \beta)_\mu$
via the $\com^*$-residue
$$\mathsf{C}^{\mathsf{P}}_0(\tau_\alpha(\Gamma_0), \beta)_\mu = 
\sum_n q^n
\int_{[U_{n,\beta}]^{vir}}
\tau_\alpha(\Gamma_0) \cup \text{ev}^*(\mu)\, ,$$
where  
$U_{n,\beta} \subset P_n(\mathbf{P}_S/S_\infty,\beta)$
is the open locus corresponding to stable pairs
which {\em do not carry components of positive $S$-degree
in the rubber over $S_\infty$} and 
$\text{ev}$ is the boundary map to the Hilbert scheme of points of
$S_\infty$.
}}
If condition \eqref{pd59} is not satisfied, 
$\mathsf{C}^{\mathsf{P}}_0(\tau_\alpha(\Gamma_0), \beta)_\mu$
vanishes.

For Gromov-Witten theory, we consider the parallel
open set
$$ \widetilde{U}_{g,\beta,\mu} \subset 
\overline{M}_{g,\star}'(\mathbf{P}_S/S_\infty,\beta)_\mu$$
corresponding to stable maps
which {\em do not carry curves of positive $S$-degree
in the rubber over $S_\infty$}.
The open set $\widetilde{U}_{g,\beta,\mu}$ is $\com^*$-invariant and has 
compact $\com^*$-fixed locus.
We again define the Gromov-Witten capped descendent over $S_0$ 
via $\com^*$-residues,
\begin{equation}\label{pddd44}
\mathsf{C}^{\mathsf{GW}}_0\Big(\overline{\tau_\alpha(\Gamma_0)},\beta\Big)_\mu = 
\sum_g u^{2g-2}
\int_{[\widetilde{U}_{g,\beta,\mu}]^{vir}}
\overline{\tau_\alpha(\Gamma_0)}
\ \ \ \in \mathbb{Q}[t,\frac{1}{t}]((u))
.
\end{equation}
The capped descendent \eqref{pddd44} vanishes unless
condition \eqref{pd59} is satisfied.

\subsubsection{Capping over $S_\infty$}
We can similarly define the capped contribution over $S_\infty$.
Let 
$$\tau_{\widehat{\alpha}}({\widehat{\Gamma}}_\infty) =
\tau_{\widehat{\alpha}_1-1}(\widehat{\gamma}_1[S_\infty]) \ldots 
\tau_{\widehat{\alpha}_\ell-1}
(\widehat{\gamma}_\ell[S_\infty])\ .$$
Consider the integrals
\begin{equation} \label{nwn33}
\sum_n q^n
\int_{[P_n(\mathbf{P}_S/S_0 \cup S_\infty,\beta)_{\nu,\mu}]^{vir}}
\tau_{\widehat{\alpha}}(\widehat{\Gamma}_\infty)
\end{equation}
\begin{equation*}
\sum_g u^{2g-2}
\int_{[\overline{M}'_{g,\star}(\mathbf{P}_S/S_0 \cup S_\infty,\beta)_{\nu,\mu}]^{vir}}
\overline{\tau_{\widehat{\alpha}}(\widehat{\Gamma}_\infty)}
\end{equation*}
via localization with respect to  the fiberwise $\com^*$-action.

The stable pairs
capped descendent over $S_\infty$ is again a sum of {particular} 
localization contributions to \eqref{nwn33}.
Let 
$$ W_{n,\beta,\nu,\mu} \subset P_n(\mathbf{P}_S/S_0\cup S_\infty,\beta)_{\nu,\mu}$$
be the open locus corresponding to stable maps
which {\em do not carry components of positive $S$-degree
in the rubber over $S_0$}.
The open set $W_{n,\beta,\nu,\mu}$ is $\com^*$-invariant and has 
compact $\com^*$-fixed locus. The fixed locus
$$W_{n,\beta,\nu,\mu}^{\com^*} \subset W_{n,\beta,\nu,\mu}$$
consists precisely of the $\com^*$-fixed loci of
$P_n(\mathbf{P}_S/S_0\cup S_\infty,\beta)_{\nu,\mu}$ with no components of
positive $S$-degree in $S_0$.
Unless the curve class $\beta$ satisfies
\begin{equation}\label{pd599}
\beta= |\nu|[{\mathbf{P}}]+ \beta_\infty, \ \ \ \beta_\infty\in 
\text{Eff}(S_\infty),
\end{equation}
the open set $W_{n,\beta,\nu,\mu}$ is empty.
The stable pairs capped descendent over $S_\infty$ is
\begin{equation}\label{pdd44}
\mathsf{C}^{\mathsf{P}}_\infty(\tau_{\widehat{\alpha}}(\widehat{\Gamma}_\infty),\beta)_{\nu,\mu} =
\sum_n q^n
\int_{[W_{n,\beta,\nu,\mu}]^{vir}}
\tau_{\widehat{\alpha}}(\widehat{\Gamma}_\infty)
\ \ \  \in \mathbb{Q}[t,\frac{1}{t}]((q))
\end{equation}
well-defined by $\com^*$-residues.
The capped descendent \eqref{pdd44} vanishes unless
condition \eqref{pd599} is satisfied.

For Gromov-Witten theory, we consider the parallel
open set
$$ \widetilde{W}_{g,\beta,\nu,\mu} \subset 
\overline{M}_{g,\star}'(\mathbf{P}_S/S_0\cup S_\infty,\beta)_{\nu,\mu}$$
corresponding to stable maps
which {\em do not carry curves of positive $S$-degree
in the rubber over  $S_0$}.
The open set $\widetilde{W}_{g,\beta,\nu,\mu}$ is $\com^*$-invariant and has 
compact $\com^*$-fixed locus.
We define the Gromov-Witten capped descendent over $S_\infty$ 
via $\com^*$-residues,
\begin{equation}\label{pded4}
\mathsf{C}^{\mathsf{GW}}_\infty\Big(\overline{\tau_{\widehat{\alpha}}
(\widehat{\Gamma}_\infty)},
\beta\Big)_{\nu,\mu} = 
\sum_g u^{2g-2}
\int_{[\widetilde{W}_{g,\beta,\nu,\mu}]^{vir}}
\overline{\tau_{\widehat{\alpha}}(\widehat{\Gamma}_\infty)}
\ \ \ \in \mathbb{Q}[t,\frac{1}{t}]((u))
 .
\end{equation}
The capped descendent \eqref{pded4} vanishes unless
condition \eqref{pd599} is satisfied.

\subsubsection{Capped localization formula} 
Let $\Phi=(\phi_1, \ldots, \phi_f)$ be a graded basis of $H^*(S,\mathbb{Q})$,
and let $\phi_1^\vee, \ldots, \phi_f^\vee$ be the dual basis satisfying
$$\int_S \phi_i \cdot \phi_j^\vee = \delta_{ij} \ .$$
We take the cohomological weights of the relative
boundary condition $\mu$ to lie in the basis $\Phi$.
Let $\mu^\vee$ then denote the boundary condition
obtained by replacing each $\phi_i$ by the Poincar\'e dual 
class $\phi_i^\vee$.

Let $\beta \in H_2(\mathbf{P}_S, \mathbb{Z})$ be a curve
class. A splitting of $\beta$ of type $d\geq 0$ is
a pair of curve classes $\beta_0,\beta_\infty$ of $\mathbf{P}_S$
satisfying
$$\beta_0\in \text{Eff}(S_0), \ \ 
\beta_\infty \in \text{Eff}(S_\infty), \ \ \text{and} \ \  
  \beta_0 + d[\mathbf{P}] + \beta_\infty = \beta.$$
We will often denote the type of a splitting by 
$$\beta=\beta_0+d[\mathbf{P}]+\beta_\infty\ .$$
A given $\beta \in  H_2(\mathbf{P}_S, \mathbb{Z})$
admits only finitely many such splittings.

The capped localization formula for $\mathbf{P}_S/S_\infty$ is is easy
to state in terms of the capped descendents over
$S_0$ and $S_\infty$.
First consider the stable pairs partition function{\footnote{We
depart slightly from the notation of the Introduction for
more efficient presentation of the data.}}
\begin{multline*}
\mathsf{Z}^{\mathsf{P}}_{\beta,\mu}\left(\tau_\alpha(\Gamma_0)\cdot
\tau_{\widehat{\alpha}}(
\widehat{\Gamma}_\infty)\right)^{\com^*} =\\
\sum_n q^n
\int_{[P_n(\mathbf{P}_S/S_\infty,\beta)_\mu]^{vir}}
\prod_i \tau_{\alpha_i-1}(\gamma_i[S_0])
\cdot
\prod_j \tau_{\widehat{\alpha}_j-1}(\widehat{\gamma}_j[S_\infty]).
\end{multline*}
The capped localization formula is
\begin{multline*}
\mathsf{Z}^{\mathsf{P}}_{\beta,\mu}\left(\tau_\alpha(\Gamma_0)\cdot
\tau_{\widehat{\alpha}}(
\widehat{\Gamma}_\infty)\right)^{\com^*}=\\
\sum \mathsf{C}^{\mathsf{P}}_0(\tau_\alpha(\Gamma_0),\beta_0+d[\mathbf{P}])_{\nu}
%
 \frac{(-1)^{|\nu|-\ell(\nu)} 
\zz(\nu)}{q^{|\nu|}} \
\mathsf{C}^{\mathsf{P}}_\infty
(\tau_{\widehat{\alpha}}(\widehat{\Gamma}_\infty),d[\mathsf{P}]+
\beta_\infty)_{\nu^\vee,\mu}).
\end{multline*}
The sum on the right side is the triple sum
$$\sum_{d\geq 0}\  \sum_{\beta_0+d[\mathsf{P}]+\beta_\infty=\beta} \ \sum_{|\nu|=d}\ .$$
The gluing factor $\zz(\nu)$ is defined by \eqref{k233}.

The parallel partition function{\footnote{Since 
$S_0$ and $S_\infty$ are disjoint, we have
$$\overline{\tau_\alpha(\Gamma_0) \cdot
\tau_{\widehat{\alpha}}(
\widehat{\Gamma}_\infty)}=
\overline{\tau_\alpha(\Gamma_0)}\cdot
\overline{\tau_{\widehat{\alpha}}(
\widehat{\Gamma}_\infty)}$$
by definition \eqref{gtte4}.}}
in Gromov-Witten theory is
\begin{multline*}
\mathsf{Z}'^{\, \mathsf{GW}}_{\beta,\mu}\Big(\overline{\tau_\alpha(\Gamma_0)}\cdot
\overline{\tau_{\widehat{\alpha}}(
\widehat{\Gamma}_\infty)}\Big)^{\com^*} =\\
\sum_g u^{2g-2}
\int_{[\overline{M}_{g,\star}'(\mathbf{P}_S/S_\infty,\beta)_\mu]^{vir}}
\overline{\prod_i \tau_{\alpha_i-1}(\gamma_i[S_0])}
\cdot
\overline{\prod_j \tau_{\widehat{\alpha}_j-1}(\widehat{\gamma}_j[S_\infty])},
\end{multline*}
and the capped localization formula is
\begin{multline*}
\mathsf{Z}'^{\, \mathsf{GW}}_{\beta,\mu}\left(
\overline{\tau_\alpha(\Gamma_0)}\cdot
\overline{\tau_{\widehat{\alpha}}(
\widehat{\Gamma}_\infty)}\right)^{\com^*}=\\
\mathsf{C}^{\mathsf{GW}}_0\Big(\overline{\tau_\alpha(\Gamma_0)},\beta_0
+d[\mathbf{P}]\Big)_{\nu}\ 
\,
{\zz(\nu)}\, {u^{2\ell(\nu)}} \
\mathsf{C}^{\mathsf{GW}}_\infty
\Big(
\overline{\tau_{\widehat{\alpha}}(\widehat{\Gamma}_\infty)},
d[\mathbf{P}]+\beta_\infty\Big)_{\nu^\vee,\mu}\ ,
\end{multline*}
where again the sum on the right is the triple sum
$$\sum_{d\geq 0}\  \sum_{\beta_0+d[\mathsf{P}]+\beta_\infty=\beta} \ \sum_{|\nu|=d}\ .$$

The idea of capping localization contributions has been used
extensively in \cite{moop,PPDC}. The main properties are the
following:
\begin{enumerate}
\item[$\bullet$] By definition, the capped contributions differ from the
bare residue contributions by just edge contributions and
$1$-legged $S$-degree 0 contributions on the far vertex.
\item[$\bullet$] The capped localization formula
is obtained from the standard localization formula by redistributing
edge and $1$-legged $S$-degree $0$ contributions (no new geometric
derivation is required).
\item[$\bullet$] The capped contributions, unlike the bare
contributions, are conjectured to have well-behaved
rationality and GW/P correspondence properties.
\end{enumerate}

\subsubsection{Capped edge}
In the capped localization formulas of 
\cite{moop,PPDC},  capped edge terms appear:
\begin{eqnarray*}
\mathsf{Z}^{\mathsf{P}}_{d,\nu,\mu} &=&
\sum_n q^n
\int_{[P_n(\mathbf{P}_S/S_0 \cup S_\infty,d[\mathbf{P}])_\mu]^{vir}}
1\ , \\
\mathsf{Z}'^{\, \mathsf{GW}}_{d,\nu,\mu} & = &
\sum_g u^{2g-2}
\int_{[\overline{M}_g(\mathbf{P}_S/S_0 \cup S_\infty,d[\mathbf{P}])_\mu]^{vir}}
1\ 
\end{eqnarray*}
where $d=|\nu|=|\mu|$.
By the following result,
the capped edges here are trivial, and hence need not be included
in the capped localization formulas in our geometry.

\begin{Lemma} We have the evaluations \label{kllw}
\begin{eqnarray*}
\mathsf{Z}^{\mathsf{P}}_{d,\nu,\mu^\vee} &=&
\delta_{\nu,\mu} \ \frac{(-1)^{|\nu|-\ell(\nu)}} 
{\zz(\nu)} \, q^{d}
\\
\mathsf{Z}'^{\, \mathsf{GW}}_{d,\nu,\mu^\vee} & = &
\delta_{\nu,\mu} \
\frac{1}{\zz(\nu)} \, u^{-2\ell(\nu)} \ .
\end{eqnarray*}
\end{Lemma}

\begin{proof}
We use the standard degeneration of
$\mathbf{P}_S/S_0 \cup S_\infty$ to  
$$\mathbf{P}_S/S_0 \cup S_\infty \ \cup \ \mathbf{P}_S/S_0 \cup S_\infty\ .$$
For the stable pairs, the degeneration formula for the
capped edges is
$$\mathsf{Z}^{\mathsf{P}}_{d,\nu,\mu^\vee} = 
\sum_{\lambda}
\mathsf{Z}^{\mathsf{P}}_{d,\nu,\lambda^\vee}\
(-1)^{|\lambda|-\ell(\lambda)} \,
\zz(\lambda) \, q^{-|\lambda|} \
\mathsf{Z}^{\mathsf{P}}_{d,\lambda,\mu^\vee}\ .
$$
The capped edge evaluation follows immediately. A parallel
argument is valid in Gromov-Witten theory.
\end{proof}

\subsection{Proof of Theorem \ref{ttt999}}

\subsubsection{Correspondence over $S_0$}
We will use
 the capped localization formulas together with
$\com^*$-equivariant descendent correspondences
for the capped contributions over $S_0$ and $S_\infty$
to prove Theorem \ref{ttt999}.

To study the contributions over $S_0$, we require
the full torus action.
Since $S$ is a toric surface, a 2-dimension torus $T$
acts on $S$. We lift $T$ to the line bundles $L_0$ and $L_1$.
Let $\mathbf{T}$ be the full 3-dimensional torus
acting on the relative geometry $\mathbf{P}_S/S_\infty$,
$$\mathbf{T}= T \times \com^*,$$
where the second factor is the fiberwise $\com^*$.

The capped contribution over $S_0$ is not difficult to understand.
The open sets 
\begin{equation}\label{5666}
 U_{n,\beta,\mu} \subset P_n(\mathbf{P}_S/S_\infty,\beta)_\mu, \ \ \ \ 
 \widetilde{U}_{g,\beta,\mu} \subset 
\overline{M}_{g,\star}'(\mathbf{P}_S/S_\infty,\beta)_\mu
\end{equation}
after localization with respect to the $\mathbf{T}$-action
yield 
only the standard capped descendent vertices at the $T$-fixed
points of $S_0$.


We consider capped contributions over $S_0$ in curve class
$$\beta = \beta_0 +d[\mathbf{P}], \ \ \ \beta_0\in \text{Eff}(S_0).$$ Let $\mu$ be a
boundary condition along $S_\infty$ with $|\mu|=d$. 

\begin{Proposition}
\label{aaa999}
The $\com^*$-equivariant descendent correspondence 
for the capped contributions over $S_0$ holds. We have
$\mathsf{C}^{\mathsf{P}}_0\left(\tau_\alpha(\Gamma_0), \beta\right)_\mu
\in \mathbb{Q}(q,t)$
and 
\begin{equation*}
(-q)^{-d_{\beta}/2}\,
\mathsf{C}^{\mathsf{P}}_0\left(\tau_\alpha(\Gamma_0), \beta\right)_\mu
 =
(-iu)^{d_{\beta}+\ell(\mu)-|\mu|}\,
\mathsf{C}^{\mathsf{GW}}_0\Big(\overline{\tau_\alpha(\Gamma_0)}, \beta\Big)_\mu
\end{equation*}
under the variable change $-q=e^{iu}$.
\end{Proposition}

\begin{proof}
We apply $\mathbf{T}$-equivariant localization
to the open sets \eqref{5666} to express
capped contributions in terms of 
descendent vertices \cite{PPDC}. We then apply  the GW/P correspondence
established in Theorem 8 of \cite{PPDC}. Since the $\com^*$-fixed locus is
compact (for the fiberwise $\com^*$-action), we may set the equivariant parameters
of $T$ to 0. 
\end{proof} 

\subsubsection{Correspondence over $S_\infty$}
The next step is to prove a descendent correspondence
for the capped contributions over $S_\infty$.
Consider capped contributions over $S_\infty$ in curve class
$$\beta = d[\mathbf{P}]+\beta_\infty, \ \ \ \beta_\infty\in \text{Eff}(S_\infty)
.$$ Let $\nu,\mu$ be 
boundary conditions along $S_0$ and $S_\infty$ with $|\nu|=d$. 

\begin{Proposition}
\label{aaa9999}
The $\com^*$-equivariant descendent correspondence 
for the capped contributions over $S_\infty$ holds. We have
$\mathsf{C}^{\mathsf{P}}_\infty\left
(\tau_{\widehat{\alpha}}(\widehat{\Gamma}_\infty), \beta\right)_{\nu,\mu} \in
\mathbb{Q}(q,s)$ and
\begin{equation*}
(-q)^{-d_{\beta}/2}\,
\mathsf{C}^{\mathsf{P}}_\infty\left
(\tau_{\widehat{\alpha}}(\widehat{\Gamma}_\infty), \beta\right)_{\nu,\mu}
 =
(-iu)^{d_{\beta}+\ell(\nu)-|\nu|+\ell(\mu)-|\mu|}\,
\mathsf{C}^{\mathsf{GW}}_\infty\Big(\overline{\tau_{\widehat{\alpha}}
(\widehat{\Gamma}_\infty)}, \beta
\Big)_{\nu,\mu}
\end{equation*}
under the variable change $-q=e^{iu}$.
\end{Proposition}

Theorem \ref{ttt999} is an immediate consequence of
Propositions \ref{aaa999} and \ref{aaa9999} and
the capped localization formulas for
$\mathbf{P}_S/S_\infty$. Proposition \ref{aaa9999} is harder to prove
than Proposition \ref{aaa999} because of the possibility of
curves of positive $S$-degree in the rubber over $S_\infty$. 
The GW/P correspondence for the descendent vertex \cite{PPDC}
does not directly apply. The proof of Proposition \ref{aaa9999}
is given in Sections \ref{instrat} - \ref{indstep}.

\subsubsection{Induction strategy}\label{instrat}
If $\beta$ is not an effective curve class, both capped
descendent contributions over $S_\infty$ vanish and
Proposition \ref{aaa9999} is trivial.

We will prove Proposition \ref{aaa9999}
for effective curve classes  by induction on
$L_\beta$ and the length $\ell(\widehat{\alpha})$.
The base case is
$$L_\beta =0  \ \ \ \text{and} \ \ \ \ell(\widehat{\alpha})=0 \ .$$
If $L_\beta=0$, then $\beta_\infty=0$.
If  $\ell(\widehat{\alpha})$ is also 0,
the capped contribution over $S_\infty$ is equal to the
capped edge term determined by Lemma \ref{kllw}.
Proposition \ref{aaa9999}
for $\beta_\infty=0$ and $\ell(\widehat{\alpha})=0$ is then easily seen to
hold.

Consider capped contributions over $S_\infty$ in curve class
$$\beta = d[\mathbf{P}]+\beta_\infty, \ \ \ \beta_\infty\in 
\text{Eff}(S_\infty)
.$$ Let $\nu$ and $\mu$ be 
relative conditions along $S_0$ and $S_\infty$ with $|\nu|=d$.
We take the cohomology weights of $\nu$ and $\mu$ to
lie in the basis 
$$\Phi=(\phi_1,\ldots, \phi_f)$$ of $H^*(S,\mathbb{Q})$.
Let $\text{deg}(\nu)$ and $\text{deg}(\mu)$
be the sum of the (complex) degrees{\footnote{We will always
use the complex grading (which is $\frac{1}{2}$ of the 
real grading).}}
 of the cohomology weights of
$\nu$ and $\mu$ respectively.
The codimensions of the relative conditions $\nu$ and $\mu$ are
$$\theta(\nu)= |\nu|-\ell(\nu)+\text{deg}(\nu) \ \ \ \text{and} \ \ \
\theta(\mu)=|\mu|-\ell(\mu)+\text{deg}(\mu) \ .$$

For the vector $\widehat{\Gamma}=(\widehat{\gamma}_1,\ldots,
\widehat{\gamma}_\ell)$
associated to the descendent insertion,
 define{\footnote{Again, we use the complex grading.}}
\begin{equation}\label{lee4}
\text{deg}(\widehat{\Gamma})= \frac{1}{2}\sum_{i=1}^\ell 
\text{deg}(\widehat{\gamma}_i), \ \ \ \widehat{\gamma}_i 
\in H^{\text{deg}(\widehat{\gamma}_i)}(S,\mathbb{Q}) \ .
\end{equation}
The maximum value of $\text{deg}(\widehat{\Gamma})$ is $2\ell$.

We will prove Proposition \ref{aaa9999} for the capped
contributions
\begin{equation}\label{dffd}
\mathsf{C}^{\mathsf{P}}_\infty(\tau_{\widehat{\alpha}}(\widehat{\Gamma}_\infty), \beta)_{\nu,\mu}, \ \ \ 
\ \ 
 \mathsf{C}^{\mathsf{GW}}_\infty\Big(\overline{\tau_{\widehat{\alpha}}
(\widehat{\Gamma}_\infty)}, \beta
\Big)_{\nu,\mu}\ .
\end{equation}
By induction, we assume Proposition \ref{aaa9999} has been established
for all capped contributions 
$$\mathsf{C}^{\mathsf{P}}_\infty(\tau_{\alpha'}(\Gamma'_\infty), \beta')_{\nu',\mu'}, \ \ \ \ \ 
 \mathsf{C}^{\mathsf{GW}}_\infty\Big(\overline{\tau_{\alpha'}(\Gamma'_\infty)}, 
\beta'\Big)_{\nu',\mu'}\ $$
satisfying at least 1 of the following 4 conditions:
\begin{enumerate}
\item[$\bullet$]
$L_{\beta'}< L_\beta\ ,$
\item[$\bullet$] $L_{\beta'}=L_\beta$ and  
$\ell(\alpha')<\ell(\widehat{\alpha})$, 
\item[$\bullet$]$L_{\beta'}=L_\beta$, $\ell(\alpha')=\ell(\widehat{\alpha})$, and
$\text{deg}(\Gamma')>\text{deg}(\widehat{\Gamma})$,
\item[$\bullet$]
$L_{\beta'}=L_\beta$, $\ell(\alpha')<\ell(\widehat{\alpha})$, 
$\text{deg}(\Gamma')=\text{deg}(\widehat{\Gamma})$, and
$\theta(\nu')< \theta(\nu)$.
\end{enumerate}
Via the third condition, we include a reverse induction over 
$\text{deg}(\widehat{\Gamma})$. 
Since $\text{deg}(\widehat{\Gamma})\leq 2\ell$, the reverse induction
is possible.

The proof of the induction step requires 
the $\com^*$-localization formula for the capped 
descendent contributions over $S_\infty$ in terms
of rubber moduli spaces. A review of the basic facts
is presented in Sections \ref{rubcon} and \ref{dimmm}.
 

\subsubsection{Rubber geometry} \label{rubcon}
The capped contributions \eqref{dffd} over $S_\infty$
are defined via $\com^*$-residues.
The $\com^*$-localization formula for the capped contributions
has three parts:
\begin{enumerate}
\item[(i)] rubber integrals over  $S_0$,
\item[(ii)] edge terms,
\item[(iii)] rubber integrals over $S_\infty$.
\end{enumerate}
The edge terms for stable pairs and stable maps are determined
by Lemma \ref{kllw}.
We discuss the rubber integrals here.

Consider first 
{rubber}{\footnote{We
follow the terminology and conventions of the
rubber discussion in \cite{part1}
for stable pairs and \cite{mptop} for Gromov-Witten theory.}} geometry
for the moduli of stable pairs.
Let
$$P_n(\mathbf{P}_S/S_0 \cup S_\infty,\beta)^\circ_{\epsilon,\delta} 
\subset P_n(\mathbf{P}/S_0\cup S_\infty,\beta)_{\epsilon,\delta}$$
denote the open set with finite stabilizers for the fiberwise
$\C^*$-action
and {\em no} destabilization over $S_\infty$.
The rubber moduli space,
$${P_n(\mathbf{P}_S/S_0\cup S_\infty,\beta)}^\sim_{\epsilon,\delta}  
= P_n(\mathbf{P}/S_0 \cup S_\infty,\beta)^\circ_{\epsilon,\delta}\ /\ \C^*,$$
denoted by a superscripted tilde,
is determined by the (stack) quotient. 
The rubber moduli space
carries a virtual fundamental class,
 $$[{P_n(\mathbf{P}_S/S_0\cup S_\infty,\beta)}^\sim_{\epsilon,\delta} ]^{vir}.$$
The fiberwise $\com^*$-action is lost after the quotient, the
fiberwise $\com^*$ acts trivially on the rubber moduli space.

The rubber moduli space $P_n(\mathbf{P}_S/S_0\cup S_\infty, \beta)^\sim
_{\epsilon,\delta}$ carries
cotangent lines associated to $S_0$ and $S_\infty$. 
A construction can be found in Section 1.5.2 of \cite{mptop}.
Let
\begin{equation*} 
\Psi_0,\Psi_\infty\in H^2({P_n(\mathbf{P}_S/S_0\cup S_\infty,\beta)}^\sim
_{\epsilon,\delta},
{\mathbb Q})
\end{equation*}
denote the associated cotangent line classes.

The $\com^*$-localization formula for the capped
descendent contribution over $S_\infty$ for stable pairs
is:
\begin{multline}\label{crubp}
\mathsf{C}^{\mathsf{P}}_\infty(\tau_{\widehat{\alpha}}(\widehat{\Gamma}_\infty), \beta)_{\nu,\mu}
=\\
\sum_{|\lambda|=d} {\mathsf R}^{\mathsf{P}}_{d[\mathbf{P}]}
\Big( \frac{1}{-\Psi_\infty-t}\Big)_{\nu,\lambda} \ 
\frac{{(-1)^{|\lambda|-\ell(\lambda)}} 
\zz(\lambda)}{q^{d}}\ \cdot \\ 
{\mathsf R}^{\mathsf{P}}_{\beta}\Big( \frac{1}{-\Psi_0+t}
\cdot \prod_{i=1}^\ell \tau_{\widehat{\alpha}_i-1}\big((t+N)
\widehat{\gamma}_i\big)\Big)
_{\lambda^\vee,\mu}\ .
\end{multline}
Here, ${\mathsf R}^{\mathsf{P}}_{d[\mathbf{P}]}$ 
denotes the generating series for
rubber integrals over $S_0\subset \mathbf{P}_S$ 
of curve class $d[\mathbf{P}]$
with inverse normal 
factor{\footnote{The normal factor
is the tensor of the tangent line $-\Psi_\infty$ of the
rubber moduli with the tangent line $-t$ on the target fiber
(which is pure weight).}}  $\Big( \frac{1}{-\Psi_\infty-t}\Big)$.
Similarly, ${\mathsf R}^{\mathsf{P}}_{\beta}$
denotes the generating series of
rubber integrals over $S_\infty\subset \mathbf{P}_S$  of
 curve class $\beta$ with inverse normal factor
$\Big( \frac{1}{-\Psi_0+t}\Big)$ and descendent insertions.
There are only two virtual normal directions in the relative
geometry here.

For Gromov-Witten theory, a parallel discussion yields the
$\com^*$-localization formula:
\begin{multline}\label{crubgw}
\mathsf{C}^{\mathsf{GW}}_\infty\Big
(\overline{\tau_{\widehat{\alpha}}(\widehat{\Gamma}_\infty)}, \beta\Big)_{\nu,\mu}
=\\
\sum_{|\lambda|=d} {\mathsf R}'^{\, \mathsf{GW}}_{d[\mathbf{P}]}
\Big( \frac{1}{-\Psi_\infty-t}\Big)_{\nu,\lambda} \ 
\zz(\lambda){u^{2\ell(\lambda)}}\ \cdot \\ 
{\mathsf R}'^{\, \mathsf{GW}}_{\beta}\Big( \frac{1}{-\Psi_0+t}
\cdot \overline{\prod_{i=1}^\ell \tau_{\widehat{\alpha}_i-1}\big((t+N)
\widehat{\gamma}_i\big)}\Big)
_{\lambda^\vee,\mu}\ .
\end{multline}

\subsubsection{Virtual dimensions} \label{dimmm}

The virtual dimensions of the stable pairs and
stable map spaces are
\begin{eqnarray*}
\text{dim} \ 
[P_n(\mathbf{P}_S/S_0 \cup S_\infty,\beta)_{\nu,\mu}]^{vir} & = &
d_\beta- \theta(\nu)
-\theta(\mu)       \ , \\
\text{dim}\
[\overline{M}_{g,\ell}'(\mathbf{P}_S/S_\infty,\beta)_{\nu,\mu}]^{vir}
& = &
d_\beta+\ell-\theta(\nu)
-\theta(\mu)\ .
\end{eqnarray*}
The virtual dimensions of the rubber moduli space
are 1 less,
\begin{eqnarray*}
\text{dim} \ 
[P_n(\mathbf{P}_S/S_0 \cup S_\infty,\beta)_{\nu,\mu}^\sim]^{vir} & = &
d_\beta- \theta(\nu)
-\theta(\mu)-1       \ , \\
\text{dim}\
[\overline{M}_{g,\ell}'(\mathbf{P}_S/S_\infty,\beta)_{\nu,\mu}^\sim]^{vir}
& = &
d_\beta+\ell-\theta(\nu)
-\theta(\mu)-1\ .
\end{eqnarray*}

\subsubsection{Proof of the induction step} \label{indstep}
We return to the proof of Proposition \ref{aaa9999}
via the induction strategy of Section \ref{instrat}.
We must prove the descendent correspondence for the
capped contributions \eqref{dffd} assuming the induction
hypothesis.
The analysis divides into two cases. 

\vspace{15pt}
\noindent {\bf Case I.}  $|\widehat{\alpha}|-2\ell(\widehat{\alpha})
+\text{deg}(\widehat{\Gamma}) \geq d_\beta-\theta(\nu)
-\theta(\mu)$
\vspace{15pt}

Under the hypothesis of Case I,
we will prove  the vanishing of both sides of the descendent
correspondence of Proposition \ref{aaa9999}
 for capped contributions over $S_\infty$ by a straightforward
dimension analysis. 

First, consider the moduli space of stable pairs.
Formula \eqref{crubp} expresses
$$\mathsf{C}^{\mathsf{P}}_\infty\left(\tau_\alpha(\Gamma_\infty), 
\beta\right)_{\nu,\mu}, \ \ \ 
\beta= d[\mathbf{P}]+\beta_\infty
$$
in terms of integrals over rubber moduli spaces.
The rubber over $S_0$ carries curve classes
with $S$-degree 0. In formula \eqref{crubp}, if
$$\theta(\nu) +\theta(\lambda) > 2d,$$
then the rubber integrals over $S_0$ vanish (since the
the virtual dimension{\footnote{The
leading $q$ term of ${\mathsf R}^{\mathsf{P}}_{d[\mathbf{P}]}
\Big( \frac{1}{-\Psi_\infty+t}\Big)_{\nu,\lambda}
$, given by the intersection pairing
between $\nu$ and $\lambda$, is degenerate.}} of the rubber moduli spaces
over $S_0$
is $2d-1$). 
Therefore,
$$\theta(\lambda^\vee) \geq \theta(\nu)\ .$$
As a consequence, the 
 virtual dimensions of the rubber moduli spaces over $S_\infty$
in \eqref{crubp}
never exceed
$$d_\beta-\theta(\nu)
-\theta(\mu)-1\ .$$
The dimension of the integrand on the rubber of $\infty$
is at least the dimension of  
$$\text{dim}\Big(\prod_{i=1}^\ell \tau_{\widehat{\alpha}_i-1}\big(\widehat{\gamma}_i\big) \Big)
=|\widehat{\alpha}|-2\ell(\widehat{\alpha})+\text{deg}
(\widehat{\Gamma}) >d_\beta-\theta(\nu)
-\theta(\mu)-1\ ,$$
where the inequality is by the hypothesis of Case I.
We conclude {\em every} rubber integral{\footnote{All the rubber
integrals are non-equivariant (there is no $\com^*$-action).
For a nonvanishing result,
the integrand can not exceed the virtual dimension.}}
over $S_\infty$
in \eqref{crubp} vanishes and hence
$$\mathsf{C}^{\mathsf{P}}_\infty\left(\tau_{\widehat{\alpha}}(\widehat{\Gamma}_\infty), 
\beta\right)_{\nu,\mu}=0 \ .$$

The argument for the vanishing of 
$\mathsf{C}^{\mathsf{GW}}_\infty\Big(\overline{\tau_{\widehat{\alpha}}
(\widehat{\Gamma}_\infty)}, \beta\Big)_{\nu,\mu}$
is identical. We use the compatibility of the correspondence
with grading established in Proposition 24 of \cite{PPDC}
and the identification of the log tangent bundle of Lemma \ref{fvv}.
Degree can be interchanged between the cotangent
lines and Chern class of $T_{\mathbf{P}_S}(-S_0-S_\infty)$. However,
since 
$$c(T_{\mathbf{P}_S}[-S_0-S_\infty])=  c(\pi^* T_{S}),$$
the dimension calculus for the vanishing remains unchanged.
We conclude 
$$\mathsf{C}^{\mathsf{GW}}_\infty\Big(\overline{\tau_{\widehat{\alpha}}
(\widehat{\Gamma}_\infty)}, 
\beta\Big)_{\nu,\mu}=0 \ .$$
Proposition \ref{aaa9999} is established in Case I.

\vspace{15pt}
\noindent {\bf Case II.}  $|\widehat{\alpha}|-2\ell(\widehat{\alpha})+\text{deg}(\widehat{\Gamma}) < d_\beta
-\theta(\nu)
-\theta(\mu)
$
\vspace{15pt}
 
The capped contributions need not vanish 
under the hypothesis of Case II.
However, we will find parallel inductive
relations to establish the descendent
correspondence of Proposition \ref{aaa9999}.


To each 
partition $\lambda$ weighted by cohomology classes of $S$
in the basis $\Phi$,
\begin{equation}\label{rr6rr}
( (\lambda_1,\delta_1), \ldots, (\lambda_{\ell(\lambda)},
\delta_{\ell(\lambda)}))\ ,  \ \ \
|\lambda|=\sum_{i=1}^{\ell(\lambda)} \lambda_i, \ \ \ \delta_i 
\in \Phi\ , 
\end{equation}
we associate a descendent
insertion over $S_0$,
$$\tau[\lambda]= 
\tau_{\lambda_1-1}(\delta_1[S_0])\cdots
\tau_{\lambda_{\ell(\lambda)}-1}(\delta_{\ell(\lambda)}[S_0])
 \ .$$
The dimension of the descendent insertion $\tau[\lambda]$
equals $\theta(\lambda)$.

Let $\Lambda_\nu$ be the  set of cohomology weighted
partitions \eqref{rr6rr} defined by
$$\Lambda_\nu = \Big\{ \lambda \ \Big| \ |\lambda|=|\nu|, \ \ \theta(\lambda)= \theta(\nu)\ \Big\}\ .$$
Since there are only finitely many partitions \eqref{rr6rr}
satisfying $|\lambda|=|\nu|$, the set
$\Lambda_\nu$ is finite.

For each cohomology weighted partition 
$\lambda\in \Lambda_\nu$, consider the stable pairs
and Gromov-Witten
generating series 
\begin{multline} \label{vanser}
\mathsf{Z}^{\mathsf{P}}_{\beta,\mu}\left(\tau[\lambda]\cdot
\tau_{\widehat{\alpha}}(
\widehat{\Gamma}_{\text{Id}})\right)^{\com^*} =\\
\sum_n q^n
\int_{[P_n(\mathbf{P}_S/S_\infty,\beta)_\mu]^{vir}}
\prod_i \tau_{\lambda_i-1}(\delta_i[S_0])
\cdot
\prod_j \tau_{\widehat{\alpha}_j-1}(\widehat{\gamma}_j).
\end{multline}
\begin{multline*} 
\mathsf{Z}'^{\, \mathsf{GW}}_{\beta,\mu}\Big(\overline{\tau[\lambda]\cdot
\tau_{\widehat{\alpha}}(
\widehat{\Gamma}_{\text{Id}})}\Big)^{\com^*} =\\
\sum_g u^{2g-2}
\int_{[\overline{M}'_{g,\star}(\mathbf{P}_S/S_\infty,\beta)_\mu]^{vir}}
\overline{\prod_i \tau_{\lambda_i-1}(\delta_i[S_0])
\cdot
\prod_j \tau_{\widehat{\alpha}_j-1}(\widehat{\gamma}_j)}.
\end{multline*}

The dimension of the integrand in for the stable pair series  \eqref{vanser} 
is
$$\theta(\nu)+  
|\widehat{\alpha}|-2\ell(\widehat{\alpha})+\text{deg}(\widehat{\Gamma}),$$
and the dimension of the moduli space of stable pairs
is 
$d_\beta
-\theta(\mu)$.
By the hypothesis of Case II, the integrand dimension is {\em strictly
less} than the dimension of the moduli space. By compactness
of the geometry,
the series \eqref{vanser} vanishes identically,
\begin{equation} \label{vanser2}
\mathsf{Z}^{\mathsf{P}}_{\beta,\mu}\left(\tau[\lambda]\cdot
\tau_{\widehat{\alpha}}(
\widehat{\Gamma}_{\text{Id}})\right)^{\com^*} = 0 \ .
\end{equation}
An identical dimension count shows
\begin{equation} \label{vanser3}
\mathsf{Z}'^{\, \mathsf{GW}}_{\beta,\mu}\Big(\overline{\tau[\lambda]\cdot
\tau_{\widehat{\alpha}}(
\widehat{\Gamma}_{\text{Id}})}\Big)^{\com^*} =0\ .
\end{equation}
The relations \eqref{vanser2}-\eqref{vanser3} 
 will be used to uniquely determine
the capped contributions
\begin{equation} \label{pddww}
\mathsf{C}^{\mathsf{P}}_\infty\left(\tau_{\widehat{\alpha}}(\widehat{\Gamma}_\infty), \beta\right)_{\nu,\mu},\ \ \ \ 
\mathsf{C}^{\mathsf{GW}}_\infty\left(\overline{\tau_{\widehat{\alpha}}(\widehat{\Gamma}_\infty)}, \beta\right)_{\nu,\mu}, \ \ \
\beta= d[\mathbf{P}]+\beta_\infty
\ .
\end{equation}
Moreover, the determinations will be sufficiently compatible to
prove the correspondence of Proposition \ref{aaa9999}
for \eqref{pddww}.

We will expand relations \eqref{vanser2}-\eqref{vanser3}
 using the 
capped localization formula. First, we write
\begin{equation}\label{jj445}
\tau_{\widehat{\alpha}_j-1}(\widehat{\gamma}_j) =
\tau_{\widehat{\alpha}_j-1}\left(\frac{\widehat{\gamma}_j}{-t-N}[S_0]\right) + 
\tau_{\widehat{\alpha}_j-1}\left(
\frac{\widehat{\gamma}_j}{t+N}[S_\infty]
\right)
\end{equation}
using the basic identity \eqref{krr4}.
We have already proven the descendent correspondence
for 
almost all the terms of the 
parallel capped localization formulas for \eqref{vanser2}-\eqref{vanser3}.
The correspondence is proven
for all the capped contributions over $S_0$ by Proposition \ref{aaa999}.
Also, the correspondence is proven for
all the capped contributions over $S_\infty$ which are
covered by the induction hypothesis of Section \ref{instrat}.
We can write 
$$0 = 
\sum_{|\rho|=d} \mathsf{C}^{\mathsf{P}}_0(\tau[\lambda],d[\mathbf{P}])_{\rho}
%
 \frac{(-1)^{|\rho|-\ell(\rho)} 
\zz(\rho)}{q^{|\rho|}} \ \left(\frac{1}{t}\right)^{\ell(\widehat{\alpha})}
\mathsf{C}^{\mathsf{P}}_\infty
(\tau_{\widehat{\alpha}}(\widehat{\Gamma}_\infty),\beta)_{\rho^\vee,\mu})+ \ldots,$$
$$0 = 
\sum_{|\rho|=d} \mathsf{C}^{\mathsf{GW}}_0(\overline{\tau[\lambda]},d[\mathbf{P}])_{\rho}\
%
\zz(\rho){u^{2\ell(\rho)}} \ \left(\frac{1}{t}\right)^{\ell(\widehat{\alpha})}
\mathsf{C}^{\mathsf{GW}}_\infty
(\overline{\tau_{\widehat{\alpha}}(\widehat{\Gamma}_\infty)},
\beta)_{\rho^\vee,\mu})+ \ldots,$$
where the sums are over all cohomology weighted
partitions $\rho$ of $d$.
The dots stand for terms covered by the 
first 3 inductive conditions: 
\begin{enumerate}
\item[$\bullet$]
lower $S$-degree over $S_\infty$, 
\item[$\bullet$]
fewer descendent
insertions over $S_\infty$, 
\item[$\bullet$]
higher descendent degree over $S_\infty$.
\end{enumerate}
The induction condition over descendent degree is used to
 replace
$$\frac{\widehat{\gamma}_j}{t+N} = \frac{\widehat{\gamma}_j}{t}
- \frac{\widehat{\gamma}_jN}{t^2} +\frac{\widehat{\gamma}_jN^2}{t^3}$$
by the leading term (note $N^3=0$ in $H^*(S,\mathbb{Q})$).

Using the $4^{th}$ induction condition, the relations can
be simplified further.
The capped contributions over $S_0$,
$$\mathsf{C}^{\mathsf{P}}_0(\tau[\lambda],d[\mathbf{P}])_{\rho}, \ \ \
\mathsf{C}^{\mathsf{GW}}_0(\overline{\tau[\lambda]},d[\mathbf{P}])_{\rho}
$$
have curve class of $S$-degree 0. Hence, the capped
contributions equal the full stable pairs 
and Gromov-Witten partition functions
\begin{eqnarray*}
\mathsf{C}^{\mathsf{P}}_0(\tau[\lambda],d[\mathbf{P}])_{\rho} & =&
\mathsf{Z}^{\mathsf{P}}_{d[\mathbf{P}] ,\rho}(\tau[\lambda])^{\com^*}\ , \\
\mathsf{C}^{\mathsf{GW}}_0(\overline{\tau[\lambda]},d[\mathbf{P}])_{\rho} & =&
\mathsf{Z}'^{\, \mathsf{GW}}_{d[\mathbf{P}] ,\rho}(\overline{\tau[\lambda]})^{\com^*}\ , 
\end{eqnarray*}
Since the moduli space $P_n(\mathbf{P}_S,d[\mathbf{P}])_{\rho}$
has virtual dimension  $2d-\theta(\rho)$, we see only
terms with
$$\theta(\lambda)+ \theta(\rho) \geq 2d $$
occur in the stable pairs relation. A parallel dimension
count yields the same conclusion on the Gromov-Witten side. 
When the inequality is strict, we have
$$\theta(\lambda)+ \theta(\rho) > 2d   \ \ \implies \ \
\theta(\rho^\vee) < \theta(\lambda)=\theta(\nu), $$
so the terms are covered by the $4^{th}$ induction condition.

The final forms we find for the principal terms on the right side of
the
relations \eqref{vanser2}-\eqref{vanser3} are the following:
$$
\sum_{|\rho|=d,\ \theta(\rho)=\theta(\nu^\vee)} \mathsf{C}^{\mathsf{P}}_0(\tau[\lambda],d[\mathbf{P}])_{\rho}
%
 \frac{(-1)^{|\rho|-\ell(\rho)} 
\zz(\rho)}{q^{|\rho|}} \ \left(\frac{1}{t}\right)^{\ell(\widehat{\alpha})}
\mathsf{C}^{\mathsf{P}}_\infty
(\tau_{\widehat{\alpha}}(\widehat{\Gamma}_\infty),\beta)_{\rho^\vee,\mu})+ \ldots,$$
$$ 
\sum_{|\rho|=d,\ \theta(\rho)=\theta(\nu^\vee)} \mathsf{C}^{\mathsf{GW}}_0(\overline{\tau[\lambda]},d[\mathbf{P}])_{\rho}\
%
\zz(\rho){u^{2\ell(\rho)}} \ \left(\frac{1}{t}\right)^{\ell(\widehat{\alpha})}
\mathsf{C}^{\mathsf{GW}}_\infty
(\overline{\tau_{\widehat{\alpha}}(\widehat{\Gamma}_\infty)},
\beta)_{\rho^\vee,\mu})+ \ldots,$$
The capped contributions 
$$\mathsf{C}^{\mathsf{P}}_\infty
(\tau_{\widehat{\alpha}}(\widehat{\Gamma}_\infty),\beta)_{\rho^\vee,\mu}),\ \ \ 
\mathsf{C}^{\mathsf{GW}}_\infty
(\overline{\tau_{\widehat{\alpha}}(\widehat{\Gamma}_\infty)},
\beta)_{\rho^\vee,\mu})$$
as $\rho$ varies yield exactly $|\Lambda_\nu|$ unknowns.
As $\lambda$ varies, we obtain exactly $|\Lambda_\nu|$ equations.
The coefficients of the system are nonsingular by Proposition 6
of \cite{PP2} on the stable pairs side (and therefore also
on the Gromov-Witten side by Proposition \ref{aaa999}).
Hence, the relations uniquely determine all the unknowns including
\eqref{pddww}. Since the descendent correspondences have already been
proven for all of the
terms besides the unknowns,
we conclude Proposition \ref{aaa9999} holds for
\eqref{pddww}. The induction step has been established. \qed

\subsection{Proof of Theorem \ref{ttt9999}}
The capped localization
formulas for stable pairs and stable maps
for the relative geometry $\mathbf{P}_S/ S_0 \cup S_\infty$
have
contributions over $S_0$ and $S_\infty$.
Both take the form of the capped contributions over
$S_\infty$ for the relative geometry
$\mathbf{P}_S/ S_\infty$. Hence,
both are covered by the descendent correspondence of
Proposition \ref{aaa9999}.
Theorem \ref{ttt9999} follows immediately. \qed

\subsection{Non-toric surfaces}\label{ntss}
Let $S$ be a nonsingular projective surface (not
necessarily toric) with line bundles 
$$L_0, L_\infty \rightarrow S \ .$$
As a consequence of Conjecture \ref{ttt222}, 
Theorems \ref{ttt999} and \ref{ttt9999} should hold
for non-toric $S$ exactly as stated. 

In fact, our proofs of Theorems \ref{ttt999} and \ref{ttt9999}
are valid for {\em any} nonsingular projective surface
$S$ for which Proposition \ref{aaa999}, concerning the
correspondence for capped descendent contributions of $\mathbf{P}_S/S_\infty$
over $S_0$, holds. The toric hypothesis for $S$ was only
used to establish Proposition \ref{aaa999} via the
descendent correspondence of \cite{PPDC} for capped vertices in toric
geometry.

In order to prove Theorem \ref{qqq111}, we will require Theorem \ref{ttt999}
for particular non-toric surfaces.
Let 
$$\epsilon: S\rightarrow C$$
 be a surface $S$ expressed as a $\mathbf{P}^1$-bundle
over a curve of genus $g$.
Let 
$$L^C_{0},  L^C_\infty \rightarrow C$$
be line bundles.
We will prove Proposition \ref{aaa999}  for $S$ and the line bundles
\begin{equation}\label{frrp}
\epsilon^*L^C_0, \ \epsilon^*L^C_\infty \rightarrow S\ 
\end{equation}
in Section \ref{phgc}. As a consequence, Theorem \ref{ttt999}
will also hold for the geometry determined by the data 
\eqref{frrp}.

In the proof of Theorem \ref{qqq111}, $K3$ surfaces will also appear.
A special case of Theorem \ref{ttt999} for $K3$ surfaces $S$
(in the non-equivariant limit) has been
established in Proposition 26 of \cite{PPDC}. 
In Proposition \ref{kk33} of Section \ref{kkk333}, we will prove
the results we require for $K3$ surfaces.

\section{Descendent correspondence for the cap}
\label{xxx1}
\subsection{Overview} 
The 1-leg {\em cap} is the total space of  
the trivial bundle,
\begin{equation}\label{pkk9}
N = \cO_{\PP^1} \oplus \cO_{\PP^1} \rightarrow \PP^1\ ,
\end{equation}
relative to the fiber
$$N_\infty \subset N$$
over $\infty \in \PP^1$.
The total space $N$ naturally carries an action of a 
3-dimensional torus $$\mathbf{T} = T \times \com^*\ .$$
Here, $T$ acts by  scaling the
factors of $N$ and preserving the relative divisor $N_\infty$. 
The $\com^*$-action
on the base $\PP^1$ which fixes the points  $0, \infty\in \PP^1$ 
lifts to an additional $\com^*$-action on $N$ fixing
$N_\infty$. Let the tangent weights at $0,\infty\in \PP^1$
with respect to the last $\com^*$-factor be $-s_3$ and $s_3$
respectively.{\footnote{The tangent weight conventions here match
\cite{part1}.}}

The equivariant cohomology 
ring $H_{\mathbf{T}}^*(\bullet)$ is generated by
the Chern classes $s_1$, $s_2$, and $s_3$
of the standard representation of the three $\com^*$-factors.
Following \cite{part1}, we define 
\begin{multline}\label{pppw}
{\mathsf Z}^{\mathsf{P}}_{d,\eta} 
\left(   \prod_{j=1}^k \tau_{i_j}([0]) \
\prod_{j'=1}^{{k'}} \tau_{i'_{j'}}([\infty])
\right)^{\mathsf{cap},\mathbf{T}}
 = \\ 
\sum _{n\in \Z }q^{n}
\int _{[P_{n} (N/N_\infty,d)_\eta]^{vir}}
  \prod_{j=1}^k \tau_{i_j}([0]) \
\prod_{j'=1}^{k'} \tau_{i'_{j'}}([\infty])
\ ,
\end{multline}
by $\mathbf{T}$-equivariant residues.
By Theorem 3 of \cite{part1}, the partition 
 function \eqref{pppw} is
a Laurent series in $q$ 
of a rational function{\footnote{By Theorem 5 of \cite{part1},
the poles in $q$ of the partition function occur only at
roots of unity.}} in
$\mathbb{Q}(q,s_1,s_2,s_3)$.
Let 
\begin{multline}\label{pppgw}
{\mathsf Z}'^{\, \mathsf{GW}}_{d,\eta} 
\left(\   \overline{\prod_{j=1}^k \tau_{i_j}([0]) \
\prod_{j'=1}^{{k'}} \tau_{i'_{j'}}([\infty])}\
\right)^{\mathsf{cap},\mathbf{T}}
 =\\ 
\sum _{g\in \Z } u^{2g-2}
\int _{[\overline{M}_{g,\star}' (N/N_\infty,d)_\eta]^{vir}}
  \overline{\prod_{j=1}^k \tau_{i_j}([0]) \
\prod_{j'=1}^{k'} \tau_{i'_{j'}}([\infty])}
\ ,
\end{multline}
be the parallel Gromov-Witten partition function.

Our goal here is to prove the relative descendent
correspondence of Conjecture \ref{ttt444}
for the fully $\mathbf{T}$-equivariant partition
functions \eqref{pppw} and \eqref{pppgw}.

\begin{Theorem}
\label{ppp444} 
For the cap geometry $N/N_\infty$, we have 
\begin{multline*}
(-q)^{-d}\ZZ^{\mathsf{P}}_{d,\eta}\Big(
\prod_{j=1}^k \tau_{i_j}([0]) \
\prod_{j'=1}^{{k'}} \tau_{i'_{j'}}([\infty])
\Big)^{\mathsf{cap},\mathbf{T}} \\ =
(-iu)^{|\eta|+\ell(\eta)}\ZZ'^{\,\mathsf{GW}}_{d,\eta}\Big(\
\overline{\prod_{j=1}^k \tau_{i_j}([0]) \
\prod_{j'=1}^{k'} \tau_{i'_{j'}}([\infty])}\
\Big) ^{\mathsf{cap},\mathbf{T}}
\end{multline*}
under the variable change $-q=e^{iu}$.
\end{Theorem}

The proof of Theorem \ref{ppp444}, given in Sections
\ref{depp} -- \ref{jajaja}, follows the strategy of the proof of
Theorem 3 of \cite{PPDC}. The main idea is to intertwine
an induction on the depth of
the descendent theories with the localization formula.

\subsection{$T$-depth} \label{depp}
For $N$ defined by \eqref{pkk9},
let $S\subset N$ be the relative divisor associated
to the points $p_1,\ldots, p_r\in \PP^1$.
We consider the $T$-equivariant stable pairs theory of $N/S$
with respect to the scaling action.

The $T$-{\em depth} $m$ theory of $N/S$ consists of all
$T$-equivariant series
\begin{equation}
\label{hkkq}
{\mathsf Z}^{\mathsf{P}}_{d,\eta^1,\dots,\eta^r} 
\left( \prod_{{j'}=1}^{k'}
\tau_{i'_{j'}}(\mathsf{1})  \  \prod_{j=1}^k \tau_{i_j}(\mathsf{p})   
\right)^{N/S,T}
\end{equation}
where $k' \leq m$.
Here, $\mathsf{p}\in H^2(\PP^1,\mathbb{Z})$ is the class of a point,
and the $\eta^i$ are partitions determining the relative conditions 
along $\pi^{-1}(p_i)$.
The $T$-depth $m$ theory has at most $m$ descendents
of $1$ and arbitrarily many descendents of $\mathsf{p}$ in the integrand.
The $T$-depth $m$ theory of $N/S$ is {\em correspondent} if 
Conjecture \ref{ttt444} holds for all
$T$-depth $m$ series \eqref{hkkq},
\begin{multline*}
\label{hkkq44}
(-q)^{-d}{\mathsf Z}^{\mathsf{P}}_{d,\eta^1,\dots,\eta^r} 
\left( \prod_{{j'}=1}^{k'}
\tau_{i'_{j'}}(\mathsf{1})  \  \prod_{j=1}^k \tau_{i_j}(\mathsf{p})   
\right)^{N/S,T} = \\
(-iu)^{2d+\sum_l\left(\ell(\eta^l)-|\eta^l|\right)}\
{\mathsf Z}'^{\,\mathsf{GW}}_{d,\eta^1,\dots,\eta^r} 
\left(\ \overline{\prod_{{j'}=1}^{k'}
\tau_{i'_{j'}}(\mathsf{1})  \  \prod_{j=1}^k \tau_{i_j}(\mathsf{p})}   
\ \right)^{N/S,T}
\end{multline*}

The $T$-depth 0 theory concerns only descendents of $\mathsf{p}$.
By taking the specialization $s_3=0$, we have 
$$
\ZZ_{d,\eta}^{\mathsf{P}}
\left(   \prod_{j=1}^k \tau_{i_j}(\mathsf{p})
\right)^{\mathsf{cap},T}=
\ZZ_{d,\eta}^{\mathsf{P}}
\left(   \prod_{j=1}^k \tau_{i_j}([0])
\right)^{\mathsf{cap},\mathbf{T}}\Big|_{s_3=0}\ . $$
The parallel relation holds for Gromov-Witten theory.
By the descendent correspondence for the
1-leg capped vertex \cite{PPDC}, we see the $T$-depth 0
theory of the cap is correspondent.

\begin{Lemma} The $T$-depth 0 theory of $N/S$ is correspondent.
\label{ht99}
\end{Lemma}

\begin{proof} By the degeneration formula, all the descendents
$\tau_{i_j}(\mathsf{p})$ can be degenerated on to a cap. 
The $T$-depth 0 theory of the cap is correspondent.
The 
theories of local curves without any insertions are
correspondent by \cite{mpt,lcdt}.
Hence, the result follows by the 
compatibility of Conjecture \ref{ttt444} with the
degeneration formula. 
\end{proof}

\subsection{Induction I}

To establish the descendent correspondence
 for the $T$-depth $m$ theory of
$N/S$, the following result
is required.

\begin{Lemma} The descendent correspondence for the
 $T$-depth $m$ theory of the cap implies
the descendent correspondence of the 
$T$-depth $m$ theory
of $N/S$. \label{rtt5}
\end{Lemma}

\begin{proof}
We must prove the descendent correspondence for the $T$-depth $m$ theories of
$N$
relative to $p_1,\ldots, p_r \in \Pp$.
If $r=1$, the geometry is the cap and the correspondence of the
$T$-depth $m$ theories is given.
Assume the correspondence holds for $r$. We will show 
the correspondence holds for
$r+1$.

Let $p(d)$ be the number of partitions of size $d>0$.
Consider the $\infty \times p(d)$ matrix $M_d$, indexed by 
monomials 
$$L= \prod_{i\geq 0} \tau_i (\mathsf{p})^{n_i} $$
in the descendents of $\mathsf{p}$ and partitions $\mu$ of $d$,
 with
coefficient 
$
{\mathsf Z}_{d,\mu}^{\mathsf{P}}
\left( L   
\right)^{{\mathsf{cap}},T}$
in position $(L,\mu)$.
The lowest Euler characteristic for a degree $d$
stable pair on the cap is $d$.  
The leading $q^d$
coefficients of $M_d$ are well-known to be of maximal
rank.{\footnote{ The leading $q^d$ coefficients
are obtained from the Chern characters
of the tautological rank $d$ bundle
on $\text{Hilb}(N_\infty,d)$.
The Chern characters generate the ring
$H^*_T(\text{Hilb}(N_\infty,d),\mathbb{Q})$ after
localization as can easily
be seen in the $T$-fixed point basis. 
A more refined result is
discussed in Proposition 9 of \cite{part1}.}}
Hence, the full matrix $M_d$ is also of maximal rank.

Consider $N$
relative to  $r+1$ points in $T$-depth $m$,
\begin{equation}
\label{yone}
{\mathsf Z}^{\mathsf{P}}_{d,\eta^1, \ldots, \eta^r,\mu} 
\left( \prod_{j'=1}^{k'}
\tau_{i'_{j'}}(1)  \  \prod_{j=1}^k \tau_{i_j}(\mathsf{p})   
\right)^{N/S,T}\ .
\end{equation}
We will determine the series \eqref{yone} from the
$T$-depth $m$ series relative to  $r$ points,
\begin{equation}
\label{yall}
{\mathsf Z}^{\mathsf{P}}_{d,\eta^1, \ldots, \eta^r}
\left( L \ \prod_{j'=1}^{k'}
\tau_{i'_{j'}}(1)  \  \prod_{j=1}^k \tau_{i_j}(\mathsf{p})   
\right)^{N/S,T}
\end{equation}
defined by all monomials $L$  in the descendents of $\mathsf{p}$.

Consider the $T$-equivariant degeneration of 
$N$ 
by bubbling off a single cap at a point not equal
to $p_1,\ldots, p_r$.
All the descendents of $\mathsf{p}$
remain on the original $N$ in the degeneration except for those 
in $L$ which distribute to the cap. 
By induction on $m$, we need only analyze the terms of the degeneration formula
in which the descendents of 1 distribute away from the cap.
Then,
since $M_d$ has full rank,
the invariants \eqref{yone} are determined
by the invariants \eqref{yall}.

The parallel inductive construction for Gromov-Witten theory determines
\begin{equation}
\label{gwyone}
{\mathsf Z}'^{\,\mathsf{GW}}_{d,\eta^1, \ldots, \eta^r,\mu} 
\left(\ \overline{\prod_{j'=1}^{k'}
\tau_{i'_{j'}}(1)  \  \prod_{j=1}^k \tau_{i_j}(\mathsf{p}) }  \
\right)^{N/S,T}\ 
\end{equation}
in terms of 
the
$T$-depth $m$ series relative to  $r$ points,
\begin{equation}
\label{gwyall}
{\mathsf Z}'^{\,\mathsf{GW}}_{d,\eta^1, \ldots, \eta^r}
\left(\ \overline{L \ \prod_{j'=1}^{k'}
\tau_{i'_{j'}}(1)  \  \prod_{j=1}^k \tau_{i_j}(\mathsf{p})}   
\ \right)^{N/S,T}\ ,
\end{equation}
the $T$-depth $m$ theory of the cap, and theories of lower $T$-depth.
By the compatibility of the descendent correspondence with
the degeneration formula, the determinations of the
$T$-depth $m$ theories of $N$ relative $r+1$ points in $\mathbf{P}^1$
respect the descendent correspondence.
\end{proof}

\vspace{10pt}

The 1-leg {\em tube} is the total space of  
the trivial bundle,
\begin{equation*}
N = \cO_{\PP^1} \oplus \cO_{\PP^1} \rightarrow \PP^1\ ,
\end{equation*}
relative to the fibers
$$N_0,N_\infty \subset N$$
over both $0,\infty \in \PP^1$.
The tube carries a fiberwise $T$-action as well as
a full $\mathbf{T}$-action.
Lemma \ref{rtt5} implies the following
result which will be  half of our induction argument
relating the descendent theory of the cap and
the tube.

\begin{Lemma}\label{nndd}
The descendent correspondence for the $T$-depth $m$ theory of the 
cap implies the descendent correspondence for the  $T$-depth $m$ theory of
the tube. 
\end{Lemma}

\subsection{$\mathbf{T}$-depth}

The $\mathbf{T}$-{\em depth} $m$ theories of the cap 
consists of all the
$\mathbf{T}$-equivariant series
\begin{equation}
\label{hkkqq}
{\mathsf Z}^{\mathsf{P}}_{d,\eta} 
\left( 
\prod_{j=1}^k \tau_{i_j}([0])  \  \prod_{{j'}=1}^{k'}
\tau_{i'_{j'}}([\infty])  
\right)^{\mathsf{cap},\mathbf{T}}\ ,
\end{equation}
\begin{equation*}
\label{hkkqqgw}
{\mathsf Z}'^{\,\mathsf{GW}}_{d,\eta} 
\left( \
\overline{\prod_{j=1}^k \tau_{i_j}([0])  \  \prod_{{j'}=1}^{k'}
\tau_{i'_{j'}}([\infty])}\  
\right)^{\mathsf{cap},\mathbf{T}}\ ,
\end{equation*}
where $k' \leq m$.
Here, $0\in \Pp$ is the non-relative $\mathbf{T}$-fixed point and
$\infty\in \Pp$ is the relative point.
The $\mathbf{T}$-depth $m$ theory of the cap 
is {\em correspondent} if Conjecture \ref{ttt444} holds
for all depth $m$  stable pairs and Gromov-Witten
partition functions \eqref{hkkqq}.

\begin{Lemma}\label{dummm} The 
descendent correspondence for the $\mathbf{T}$-depth $m$
theory of the cap implies the
descendent correspondence 
for  the $T$-depth $m$ theory of the cap.
\end{Lemma}

\begin{proof}
The identity class
$1\in H^*_T(\Pp,\mathbb{Z})$ may be
expressed in terms of the $\mathbf{T}$-fixed point classes,
$$1 = -\frac{[0]}{s_3} + \frac{[\infty]}{s_3}\ .$$
We can calculate at most $m$ descendents of $1$ in  the $T$-equivariant
theory via at most 
$m$ descendents of $[\infty]$ in the $\mathbf{T}$-equivariant
theory (followed by the specialization $s_3=0$).
\end{proof}

\subsection{Induction II}

\label{ind2}

The first half of our induction argument was established
in Lemma \ref{nndd}.
The second half relates the tube back to the 
cap with an increase in depth.

\begin{Lemma} The descendent correspondence for the  \label{p45} 
the ${T}$-depth $m$
theory of the tube implies the
descendent correspondence for the
 $\mathbf{T}$-depth $m+1$ theory of the cap.
\end{Lemma}

\begin{proof}
The result follows from the $\mathbf{T}$-equivariant
localization
formula for the cap in terms of the $T$-equivariant
theory of the tube (already used in \cite{part1}).
We first review the formula.

For the theory of stable pairs, consider the partition function
\begin{equation} \label{hgg66t}
{\mathsf Z}^{\mathsf{P}}_{d,\eta} 
\left(   \prod_{j=1}^k \tau_{i_j}([0])  \prod_{j'=1}^{k'} \tau_{i'_{j'}}([\infty])
\right)^{\mathsf{cap},\mathbf{T}} \ . 
\end{equation}
We will write the $\mathbf{T}$-equivariant localization formula
for \eqref{hgg66t}, as a sum over set partitions 
$$\mathsf{R}=(R_1,R_2,\ldots, R_{r(\mathsf{R})}), \ \ \ 
R_i\subset\{ 1, \ldots, k'\}$$
satisfying the following conditions
\begin{enumerate}
\item[$\bullet$] $R_i$ are nonempty and disjoint,
\item[$\bullet$] $R_1 \cup R_2 \cup \ldots \cup R_{r(\mathsf{R})} = \{ 1, \ldots, k'\}$,
\item[$\bullet$] $\text{min}\{j'\in R_i\} > \text{min}\{j'\in R_{i+1}\}$ .
\end{enumerate}
Let $m_i$ be the minimal index in $R_i$.
As a consequence of the third condition, $m_{r(\mathsf{R})}=1\in R_{r(\mathsf{R})}$.
The formula for the partition function \eqref{hgg66t} is
\begin{multline*}
 \sum_{\mathsf{R}}
s_3^{k'-r(\mathsf{R})}\ {\mathsf Z}^{\mathsf{P}}_{d,\eta^1} 
\left(   \prod_{j=1}^k \tau_{i_j}([0]) \right)^{\mathsf{cap},\mathbf{T}}
\frac{g^{\eta^1\tilde{\eta}^1}}{q^d}
{\mathsf Z}^{\mathsf{P}}_{d,\tilde{\eta}^1,\eta^2} 
\left(  \tau_{i'_{m_1}}(\mathsf{p}) \prod_{j'\in R_1^*} \tau_{i'_{j'}}(1) \right)^{\mathsf{tube},{T}} \\ \ \ \ \ \ \ \ \ \ \ \ \ \ \ \
\ \ \ \ \ \ \ \ \ \ \ \ \ \ \ \ \ \ \ \ \ \ \ \ \ \ \ \ \ \ \ 
\cdot \frac{g^{\eta^2\tilde{\eta}^2}}{q^d}
{\mathsf Z}^{\mathsf{P}}_{d,\tilde{\eta}^2,\eta^3} 
\left(  \tau_{i'_{m_2}}(\mathsf{p}) \prod_{j'\in R_2^*} \tau_{i'_{j'}}(1) \right)^{\mathsf{tube},{T}}\\ \ \ \ \ \ \  \ \ \ \  \ \  
\cdots \\
\ \ \ \ \ \ \ \ \ \ \ \ \ \  \ \ \ \ \ \ \ \ \ \ \ \ \ \ \ \ \ \ \ \
\hspace{50pt} 
\cdot \frac{g^{\eta^{r}\tilde{\eta}^r}}{q^d}
{\mathsf Z}^{\mathsf{P}}_{d,\tilde{\eta}^r,\eta} 
\left(  \tau_{i'_{m_r}}(\mathsf{p}) \prod_{j'\in R_r^*} \tau_{i'_{j'}}(1) \right)^{\mathsf{tube},{T}}, \hspace{-7pt}
\end{multline*}
where the metric term is
$$g^{\eta\tilde{\eta}} = (s_1s_2)^{\ell(\eta)}{(-1)^{|\eta|-\ell(\eta)}} 
\zz(\eta)\cdot \delta_{\eta,\tilde{\eta}}\ .$$
The above $\mathbf{T}$-equivariant formula is proven via
localization and the rubber calculus, see Section 7.2 of \cite{part1}.

For a partition function \eqref{hgg66t} of $\mathbf{T}$-depth
$m+1$, the right side of the $\mathbf{T}$-equivariant localization
formula is in terms of the $\mathbf{T}$-depth 0 theory of the
cap and the $T$-depth $m$ theory of the tube. 
%
Consider next the 
Gromov-Witten partition function,
\begin{equation} \label{gg66t}
{\mathsf Z}'^{\,\mathsf{GW}}_{d,\eta} 
\left(\   \overline{\prod_{j=1}^k \tau_{i_j}([0])  
\prod_{j'=1}^{k'} \tau_{i'_{j'}}([\infty])}\
\right)^{\mathsf{cap},\mathbf{T}} \ . 
\end{equation}
The $\mathbf{T}$-equivariant localization formula for
\eqref{gg66t} is  
\begin{multline*}
\hspace{-13pt}
 \sum_{\mathsf{R}}
s_3^{k'-r(\mathsf{R})}\ {\mathsf Z}'^{\,\mathsf{GW}}_{d,\eta^1} 
\left(\   \overline{\prod_{j=1}^k \tau_{i_j}([0])}\ \right)^{\mathsf{cap},\mathbf{T}}
\hspace{-5pt}
\frac{h^{\eta^1\tilde{\eta}^1}}{u^{-2\ell(\eta^1)}}
{\mathsf Z}'^{\,\mathsf{GW}}_{d,\tilde{\eta}^1,\eta^2} 
\left(  \overline{\tau_{i'_{m_1}}(\mathsf{p}) \prod_{j'\in R_1^*} 
\tau_{i'_{j'}}(1)} \right)^{\mathsf{tube},{T}} \\ \ \ \ \ \ \ \ \ \ \ \ \ \ \ \
\ \ \ \ \ \ \ \ \ \ \ \ \ \ \ \ \ \ \ \ \ \ \ \ \ \ \ \ \ \ \ 
\cdot \frac{h^{\eta^2\tilde{\eta}^2}}{u^{-2\ell(\eta^2)}}
{\mathsf Z}'^{\,\mathsf{GW}}_{d,\tilde{\eta}^2,\eta^3} 
\left(  \overline{\tau_{i'_{m_2}}(\mathsf{p}) \prod_{j'\in R_2^*} 
\tau_{i'_{j'}}(1)} \right)^{\mathsf{tube},{T}}\\ \ \ \ \ \ \  \ \ \ \  \ \  
\cdots \\
\ \ \ \ \ \ \ \ \ \ \ \ \ \  \ \ \ \ \ \ \ \ \ \ \ \ \ \ \ \ \ \ \ \
\hspace{50pt} 
\cdot \frac{h^{\eta^{r}\tilde{\eta}^r}}{u^{-2\ell(\eta^r)}}
{\mathsf Z}'^{\,\mathsf{GW}}_{d,\tilde{\eta}^r,\eta} 
\left(  \overline{\tau_{i'_{m_r}}(\mathsf{p}) \prod_{j'\in R_r^*} \tau_{i'_{j'}}(1)} \right)^{\mathsf{tube},{T}}, \hspace{-7pt}
\end{multline*}
where the metric is now
$$h^{\eta\tilde{\eta}} = (s_1s_2)^{\ell(\eta)} 
\zz(\eta)\cdot \delta_{\eta,\tilde{\eta}}\ .$$
The proof is again via standard localization and 
rubber calculus.

The
descendent correspondence of Conjecture \ref{ttt444} is formally
compatible with the above $\mathbf{T}$-equivariant localization formulas.
Since the right sides concern only 
the $\mathbf{T}$-depth 0 theory of the
cap and the $T$-depth $m$ theory of the tube,
Lemma \ref{p45} is an immediate consequence. 
\end{proof}

\subsection{Gromov-Witten side}
The stable pairs localization formula for \eqref{hgg66t} in 
Section \ref{ind2} was
explained in \cite{part1}. While the Gromov-Witten side is
parallel,
we present the first cases here to help the reader.

To start, we write the
 localization formula for 
$\mathbf{T}$-depth 1 series for the cap as 
\begin{multline*}
{\mathsf Z}'^{\, \mathsf{GW}}_{d,\eta} 
\left(\   \overline{\prod_{j=1}^k \tau_{i_j}([0]) \cdot \tau_{i'_1}([\infty])}
\ \right)^{\mathsf{cap},\mathbf{T}} = \\
{\mathsf Z}'^{\, \mathsf{GW}}_{d,\eta} 
\left(\   \overline{\prod_{j=1}^k \tau_{i_j}([0])} \cdot \overline{\tau_{i'_1}([\infty])}
\ \right)^{\mathsf{cap},\mathbf{T}}=
\\
\sum_{|\mu|=d}
\bW_\mu^{\mathsf{Vert}} \left(\ \overline{\prod_{j=1}^k \tau_{i_j}([0])}  \ 
   \right) \cdot
\bW_\mu^{(0,0)} \cdot
\mathsf{S}^{\mu}_{\eta}(\overline{\tau_{i'_1}}) \ ,
\end{multline*}
where the rubber term on the right is 
\begin{eqnarray*}
\mathsf{S}^\mu_\eta(\overline{\tau_{i'_1}}) & = &   
\sum_{g} u^{2g-2}
\left\langle \mu \ \left| \ \frac{ s_3 \overline{\tau_{i'_1}}}{s_3-\psi_0}  \ \right|\ \eta 
\right\rangle_{g,d}^{
\sim}. \\
\end{eqnarray*}
Here, $\mathsf{W}^{\mathsf{Vert}}_\mu$ and $\mathsf{W}^{0,0}_\mu$
denote the Gromov-Witten vertex and edge terms.

The rubber term simplifies
via  the
topological recursion relation for $\psi_0$ after
writing
\begin{equation}\label{nhhk}
\frac{s_3}{s_3-\psi_0} = 1 + \frac{\psi_0}{s_3-\psi_0}\ .
\end{equation}
We find the relation
\begin{eqnarray*}
\mathsf{S}^\mu_\eta(\overline{\tau_{i'_1}}) 
& = &
\sum_{|\tilde{\eta}|=d} \mathsf{S}^\mu_{\tilde{\eta}} 
\cdot \frac{h^{\tilde{\eta}\tilde{\eta}}}{u^{-2\ell(\tilde{\eta})}} \cdot 
 {\mathsf Z}'^{\, \mathsf{GW}}_{d,\tilde{\eta},\eta} 
\left(   \overline{\tau_{i_1'}([\infty])} \right)^{\mathsf{tube},T}\
\end{eqnarray*}
where the rubber term on the right is
\begin{eqnarray*}
\mathsf{S}^\mu_\eta & = &   
\sum_{g} u^{2g-2}
\left\langle \mu \ \left| \ \frac{ 1}{s_3-\psi_0}  \ \right|\ \eta 
\right\rangle_{g,d}^{
\sim}. \\
\end{eqnarray*}
The leading  $1$ on the right side of \eqref{nhhk} corresponds to
the degenerate leading term of $\mathsf{S}^\mu_{\tilde{\eta}}$.
The topological recursion applied to the $\psi_0$ prefactor
of the second term produces the rest of $\mathsf{S}^\mu_{\tilde{\eta}}$.
We have also used here the identification of the log tangent
bundle on the destabilized cap.

After reassembling the localization formula, we find
\begin{multline*}
{\mathsf Z}'^{\, \mathsf{GW}}_{d,\eta} 
\left(\ \overline{\prod_{j=1}^k \tau_{i_j}([0]) \cdot \tau_{i'_1}([\infty])}
\ \right)^{\mathsf{cap},\mathbf{T}} = 
\\
\sum_{|\tilde{\eta}|=d}
{\mathsf Z}'^{\, \mathsf{GW}}_{d,\tilde{\eta}} 
\left(\  \overline{\prod_{j=1}^k \tau_{i_j}([0])} \
\right)^{\mathsf{cap},\mathbf{T}}
\cdot \frac{h^{\tilde{\eta}\tilde{\eta}}}{u^{-2\ell(\tilde{\eta})}} \cdot 
 {\mathsf Z}'^{\, \mathsf{GW}}_{d,\tilde{\eta},\eta} 
\left(   \overline{\tau_{i_1'}([\infty])} \right)^{\mathsf{tube},T}
\end{multline*}
which is equivalent to the first case of the Gromov-Witten
formula of Section \ref{ind2}.

The higher cases of the Gromov-Witten localization formula
of Section \ref{ind2} are proven by expanding 
definition \eqref{gtte4} of the descendent correspondence and
following the rubber calculus. Consider
\begin{multline*}
{\mathsf Z}'^{\, \mathsf{GW}}_{d,\eta} 
\left(\ \overline{\prod_{j=1}^k \tau_{i_j}([0]) \cdot \tau_{i'_1}([\infty])
\tau_{i'_2}([\infty])}
\ \right)^{\mathsf{cap},\mathbf{T}}
= \\
{\mathsf Z}'^{\, \mathsf{GW}}_{d,\eta} 
\left(\   \overline{\prod_{j=1}^k \tau_{i_j}([0])} \cdot 
\overline{\tau_{i'_1}([\infty])
\tau_{i'_2}([\infty])
}
\ \right)^{\mathsf{cap},\mathbf{T}},
\end{multline*}
where we have 
\begin{multline} \label{naaa}
\overline{\tau_{i'_1}([\infty])
\tau_{i'_2}([\infty])
} =
s_3\sum_{\alpha} 
\tau_{\alpha}(\widetilde{K}_{(i_1'+1,i_2'+1),\alpha}\cdot [\infty]) 
\\ +
\sum_{\delta} 
\tau_{\delta}(\widetilde{K}_{(i_1'+1),\delta}\cdot [\infty])
\cdot
\sum_{\epsilon} 
\tau_{\epsilon}(\widetilde{K}_{(i_2'+1),\epsilon}\cdot [\infty])
\end{multline}
by definition.
The first summand on the right of \eqref{naaa} is obtained from
 the set partition
$\{1,2\}$ and the second term from the set partition $\{1\}\cup \{2\}$.
After applying localization and the rubber calculus to the
$\{1,2\}$ term, we obtain
the $\{1,2\}$ term of 
\begin{equation}\label{hohoho}
\sum_{|\tilde{\eta}|=d}
s_3{\mathsf Z}'^{\, \mathsf{GW}}_{d,\tilde{\eta}} 
\left(\  \overline{\prod_{j=1}^k \tau_{i_j}([0])} \
\right)^{\mathsf{cap},\mathbf{T}}
\cdot \frac{h^{\tilde{\eta}\tilde{\eta}}}{u^{-2\ell(\tilde{\eta})}} \cdot 
 {\mathsf Z}'^{\, \mathsf{GW}}_{d,\tilde{\eta},\eta} 
\left(   \overline{\tau_{i_1'}(\mathsf{p})
\tau_{i_2'}(1)
} \right)^{\mathsf{tube},T}\ .
\end{equation}
After applying localization and the rubber calculus to the
$\{1\} \cup \{2\}$ term of \eqref{naaa}, we obtain
the $\{1\} \cup \{2\}$ term of \eqref{hohoho}
plus the full series
\begin{multline} \label{pzz2}
\sum_{|\tilde{\mu}|,|\tilde{\eta}|=d}
{\mathsf Z}'^{\,\mathsf{GW}}_{d,\tilde{\mu}} 
\left(\ \overline{\prod_{j=1}^k \tau_{i_j}([0])} 
\ \right)^{\mathsf{cap},\mathbf{T}}
\cdot \frac{h^{\tilde{\mu}\tilde{\mu}}}{u^{-2\ell(\tilde{\mu})}} \cdot 
 {\mathsf Z}'^{\,\mathsf{GW}}_{d,\tilde{\mu},\tilde{\eta}} 
\left(
\overline{\tau_{i'_{2}}(\mathsf{p})}    \right)^{\mathsf{tube},T} 
\\ 
\cdot \frac{h^{\tilde{\eta}\tilde{\eta}}}{u^{-2\ell(\tilde{\eta})}} \cdot 
 {\mathsf Z}'^{\,\mathsf{GW}}_{d,\tilde{\eta},{\eta}} 
\left(\overline{\tau_{i'_{1}}(\mathsf{p})}    \right)^{\mathsf{tube},T} \ . \ \ \ \ \ 
\end{multline}
Combining \eqref{hohoho} and \eqref{pzz2} exactly
yields the Gromov-Witten formula of Section \ref{ind2} for
2 insertions over $\infty$.

\subsection{Proof of Theorem \ref{ppp444}}\label{jajaja}
Lemmas \ref{nndd}--\ref{p45}
together provide an induction which establishes
the descendent correspondence for
the $\mathbf{T}$-depth $m$ theory of
the cap for all $m$. \qed

\vspace{10pt}
Since the classes of the $\mathbf{T}$-fixed points $0,\infty \in \Pp$
generate $H_{\mathbf{T}}^*(\Pp, \mathbb{Z})$
after localization, Theorem \ref{ppp444}
is a $\mathbf{T}$-equivariant correspondence for
the full descendent theory of the cap.

\subsection{$K3$ surfaces}
\label{kkk333}
For a surface $S$,
following the notation of 
Section \ref{rr}, let 
$$\mathbf{P}_S= \mathbf{P}(L_0 \oplus L_\infty) \rightarrow S, \ \ \
 S_i= \mathbf{P}(L_i) \subset \mathbf{P}_S\ .$$

\begin{Proposition}
\label{kk33} 
Let $S$ be a nonsingular projective $K3$ surface.
For classes $\gamma_i \in H^{*}(S,\mathbb{Q})$,
we have 
$$\ZZ_{\mathsf{P}}\Big(\mathbf{P}_S/S_\infty   ;q\ \Big|  
{\tau_{\alpha_1-1}(\gamma_1)\cdots
\tau_{\alpha_{\ell}-1}(\gamma_{\ell})} \ \Big| \ \mu
\Big)_\beta\in \mathbb{Q}(q)$$
and the correspondence
\begin{multline*}
(-q)^{-d_\beta/2}\ZZ_{\mathsf{P}}\Big(\mathbf{P}_S/S_\infty   ;q\ \Big|  
{\tau_{\alpha_1-1}(\gamma_1)\cdots
\tau_{\alpha_{\ell}-1}(\gamma_{\ell})} \ \Big| \ \mu
\Big)_\beta \\ =
(-iu)^{d_\beta+\ell(\mu)-|\mu|}\ZZ'_{\mathsf{GW}}\Big(\mathbf{P}_S/S_\infty
;u\ \Big|   
\ \overline{\tau_{a_1-1}(\gamma_1)\cdots
\tau_{\alpha_{\ell}-1}(\gamma_{\ell})}
\ \Big| \ \mu\Big)_\beta 
\end{multline*}
under the variable change $-q=e^{iu}$.
\end{Proposition}

\begin{proof}
If the cohomology insertions $\gamma_i$ are
supported on $S_0$, then the above correspondence is
proven
in Proposition 26 of \cite{PPDC}.
The support hypotheses for $\gamma_i$ were  
needed there since, for the $\mathbf{T}$-equivariant cap,
only the correspondence for descendents of the
non-relative point had been proven in \cite{PPDC}.
Theorem \ref{ppp444} now removes the need for the support
hypothesis.
The proof of Proposition 26 together with Theorem \ref{ppp444}
yields the result.
\end{proof}

\section{The geometry $\Pp\times \com \times
\Pp \ /\ \mathbf{P}^1  \times \com$}
\label{xxx2}

\subsection{Overview} \label{erst}
Let $Y$ denote the 
the quasi-projective variety
$\Pp\times \com \times \Pp$.
For clarity, we will denote the first factor by $\mathbf{P}^1_1$ and the
third factor by $\mathbf{P}^1_3$.
Let
$$\pi_1 : Y \rightarrow \mathbf{P}^1_1, \ \ \ 
\pi_3 : Y \rightarrow \mathbf{P}^1_3
$$
denote the projections onto to the first and last factors.

The variety $Y$
admits an action of the 3-torus 
$$\mathbf{T}=\com_1^* \times \com_2^* \times \com_3^*\ .$$
The first factor $\com_1^*$ of $\mathbf{T}$ acts on
$\mathbf{P}^1_1$ with fixed points $0,\infty\in \mathbf{P}^1_1$
with tangent weights $-s_1,s_1$ respectively.
The factor $\com^*_2$ acts on $\com$ with
fixed point $0\in \com$ with tangent weight $-s_2$.
Finally, $\com_3^*$ acts on
$\mathbf{P}^1_3$ with fixed points $0,\infty\in \mathbf{P}^1_3$
with tangent weights $-s_3,s_3$ respectively.

Define the divisors
 $Y_0,Y_\infty \subset Y$ to be the fibers of $\pi_3$ over $0,\infty\in 
\mathbf{P}^1_3$,
$$
Y_0 =  \mathbf{P}^1_1\times
\com\times \{0 \}, \ \ \ 
Y_\infty = \mathbf{P}^1_1 \times
\com\times \{\infty\} .$$
Both $Y_0$ and $Y_\infty$ are preserved by the $\mathbf{T}$-action.
Let 
$$[0],[\infty]\in H^2_{\mathbf{T}}(Y,\mathbb{Q})$$
denote the classes of $Y_0$ and $Y_\infty$ respectively.

The projection $\pi_1$ is equivariant with respect to the
projection of $\mathbf{T}$ onto $\com^*_1$.
We will view
$$\theta_j,\theta'_{j'}\in  H^*_{\com_1^*}(\mathbf{P}^1_1,\mathbb{Q})$$
as classes in $H^*_{\mathbf{T}}(Y,\mathbb{Q})$ via pull-back
by $\pi_1$.

Since $Y_\infty$ is preserved
by the $\mathbf{T}$-action, we can define
\begin{multline}\label{hhhw}
{\mathsf Z}^{\mathsf{P}}_{\beta,\eta} 
\left(   \prod_{j=1}^k \tau_{i_j}(\theta_j[0]) \
\prod_{j'=1}^{{k'}} \tau_{i'_{j'}}(\theta_{j'}'[\infty])
\right)^{Y/Y_\infty,\mathbf{T}}
 = \\ 
\sum _{n\in \Z }q^{n}
\int _{[P_{n} (Y/Y_\infty,\beta)_\eta]^{vir}}
  \prod_{j=1}^k \tau_{i_j}(\theta_j[0]) \
\prod_{j'=1}^{k'} \tau_{i'_{j'}}(\theta'_{j'}[\infty])
\ ,
\end{multline}
by $\mathbf{T}$-equivariant residues.
Here, $\beta \in H_2(Y,\mathbb{Z})$ is a curve
class (specified by degrees along the two $\Pp$-factors),
and
$\eta$ is a boundary condition along $Y_\infty$.
The parallel Gromov-Witten partition function is
\begin{multline}\label{hhhgw}
{\mathsf Z}'^{\,\mathsf{GW}}_{\beta,\eta} 
\left(\   \overline{\prod_{j=1}^k \tau_{i_j}(\theta_j[0]) \
\prod_{j'=1}^{{k'}} \tau_{i'_{j'}}(\theta_{j'}'[\infty])}\
\right)^{Y/Y_\infty,\mathbf{T}}
 =\\ 
\sum _{g\in \Z } u^{2g-2}
\int _{[\overline{M}_{g,\star}' (Y/Y_\infty,\beta)_\eta]^{vir}}
  \overline{\prod_{j=1}^k \tau_{i_j}(\theta_j[0]) \
\prod_{j'=1}^{k'} \tau_{i'_{j'}}(\theta_{j'}'[\infty])}
\ .
\end{multline}

Our goal here is to prove the relative descendent
correspondence of Conjecture \ref{ttt444}
for the fully $\mathbf{T}$-equivariant partition
functions \eqref{hhhw} and \eqref{hhhgw}.

\begin{Theorem}
\label{ppp555} 
For the relative geometry $Y/Y_\infty$, we have
$$\ZZ^{\mathsf{P}}_{\beta,\eta}\Big(
\prod_{j=1}^k \tau_{i_j}(\theta_j[0]) \
\prod_{j'=1}^{{k'}} \tau_{i'_{j'}}(\theta_{j'}'[\infty])
\Big)^{Y/Y_\infty,\mathbf{T}} \ \in \mathbb{Q}(q,s_1,s_2,s_3)$$
and the correspondence 
\begin{multline*}
(-q)^{-d_\beta/2}\ZZ^{\mathsf{P}}_{\beta,\eta}\Big(
\prod_{j=1}^k \tau_{i_j}(\theta_j[0]) \
\prod_{j'=1}^{{k'}} \tau_{i'_{j'}}(\theta_{j'}'[\infty])
\Big)^{Y/Y_\infty,\mathbf{T}} \\ = 
(-iu)^{d_\beta+\ell(\eta)-|\eta|}\ZZ'^{\,\mathsf{GW}}_{\beta,\eta}\Big(\
\overline{\prod_{j=1}^k \tau_{i_j}(\theta_j[0]) \
\prod_{j'=1}^{k'} \tau_{i'_{j'}}(\theta_{j'}'[\infty])}\
\Big) ^{Y/Y_\infty,\mathbf{T}}
\end{multline*}
under the variable change $-q=e^{iu}$.
\end{Theorem}

The proof of Theorem \ref{ppp555}, given in Sections
\ref{depin} -- \ref{jjcc3}, again proceeds
by induction on the depth of
the descendent theories.
A study of capped rubber for the geometry
$Y/ Y_0\cup Y_\infty$ is required for the 
base case of the induction.

\subsection{Depth induction}
\label{depin}
The proof of Theorem \ref{ppp444} can be exactly followed
to establish Theorem \ref{ppp555}. To start, we define
the two notions of depth for the geometry $Y$.

Let $S\subset Y$ be the relative divisor $\cup_i \pi_3^{-1}(p_i)$
associated
to the points $p_1,\ldots, p_r\in \mathbf{P}^1_3$.
Let $$T=\com_1^* \times \com_2^*\subset \mathbf{T}$$
be the first two factors of the 3-torus.
We consider the $T$-equivariant stable pairs theory of $Y/S$.
The $T$-{\em depth} $m$ theory of $Y/S$ consists of all
$T$-equivariant series
\begin{equation}
\label{hkkqqq}
{\mathsf Z}^{\mathsf{P}}_{\beta,\eta^1,\dots,\eta^r} 
\left( \prod_{{j'}=1}^{k'}
\tau_{i'_{j'}}(\theta_{j'}'\cdot \mathsf{1})  \  \prod_{j=1}^k \tau_{i_j}(\theta_j\cdot \mathsf{p})   
\right)^{Y/S,T}
\end{equation}
where $k' \leq m$.
Here, $\mathsf{p}\in H^2(Y,\mathbb{Z})$ denotes the
class of a fiber of $\pi_3$,
and the $\eta^i$ are partitions determining the relative conditions 
along $\pi^{-1}(p_i)$.
Following exactly the proof of Lemma \ref{nndd}, we obtain the
following result.

\begin{Lemma}\label{nnnddd}
The descendent correspondence for the $T$-depth $m$ theory of 
$Y/Y_\infty$ implies the descendent correspondence for the  $T$-depth $m$ theory of
the $Y/Y_0\cup Y_\infty$. \qed
\end{Lemma}

The stable $\mathbf{T}$-{\em depth} $m$ theory of $Y/Y_\infty$ 
consists of all the
$\mathbf{T}$-equivariant series
\begin{equation}
\label{hkkqq5}
{\mathsf Z}^{\mathsf{P}}_{\beta,\eta} 
\left( 
\prod_{j=1}^k \tau_{i_j}(\theta_j[0])  \  \prod_{{j'}=1}^{k'}
\tau_{i'_{j'}}(\theta_{j'}'[\infty])  
\right)^{Y/Y_\infty,\mathbf{T}}
\end{equation}
where $k' \leq m$.

The proofs of Lemmas \ref{dummm} and \ref{p45} are formal and
remain valid
for the the geometry $Y/Y_\infty$. As a result, we obtain
the following two results.

\begin{Lemma}\label{dummmq} The 
descendent correspondence for the $\mathbf{T}$-depth $m$
theory of $Y/Y_\infty$ implies the
descendent correspondence 
for  the $T$-depth $m$ theory of the $Y/Y_\infty$. \qed
\end{Lemma}

\begin{Lemma} The descendent correspondence for the  \label{p45q} 
the ${T}$-depth $m$
theory of the tube implies the
descendent correspondence for the
 $\mathbf{T}$-depth $m+1$ theory of the cap. \qed
\end{Lemma}

Lemmas \ref{nnnddd}--\ref{p45q} together establish a
recursion which reduces Theorem \ref{ppp555} to the
base case of the $\mathbf{T}$-depth 0 theory of $Y/Y_\infty$.

\subsection{$\mathbf{T}$-depth 0} \label{fvvf}
The last step in the proof of Theorem \ref{ppp555} is to
establish the descendent correspondence in 
the base case of $\mathbf{T}$-depth 0. 
 
\begin{Proposition}
\label{ppp555p} 
For the relative geometry $Y/Y_\infty$, we have
$$\ZZ^{\mathsf{P}}_{\beta,\eta}\Big(
\prod_{j=1}^k \tau_{i_j}(\theta_j[0])\Big)^{Y/Y_\infty,\mathbf{T}} \ \in \mathbb{Q}(q,s_1,s_2,s_3)$$
and the correspondence 
\begin{multline*}
(-q)^{-d_\beta/2}\ZZ^{\mathsf{P}}_{\beta,\eta}\Big(
\prod_{j=1}^k \tau_{i_j}(\theta_j[0]) 
\Big)^{Y/Y_\infty,\mathbf{T}} \\ =
(-iu)^{d_\beta+\ell(\eta)-|\eta|}\ZZ'^{\,\mathsf{GW}}_{d,\eta}\Big(\
\overline{\prod_{j=1}^k \tau_{i_j}(\theta_j[0])} \
\Big) ^{Y/Y_\infty,\mathbf{T}}
\end{multline*}
under the variable change $-q=e^{iu}$.
\end{Proposition}

We can write the partition function for $Y/Y_\infty$
via capped localization for both stable pairs and
Gromov-Witten theory. The capped contributions
over $Y_0$ are 2-leg capped toric descendent vertices and
satisfy the descendent correspondence by \cite{PPDC}.
The interesting new terms in the capped localization formula
occur over $Y_\infty$ --- the capped rubber contributions.
The capped rubber contributions carry {\em no descendent
insertions}.

To prove the correspondence for the capped rubber
contributions over $Y_\infty$, we follow the technique 
developed in \cite{moop} where the capped rubber
contributions for 
$$\mathcal{A}_n \times \Pp\ / \mathcal{A}_n \times \{\infty\}$$
over $\infty$ were studied.
Via the differential equations constructed in Sections 3.2 of
\cite{moop}, the analysis of Section 3.4 can be applied to
our capped rubber contributions. The proof of 
Lemma 6 of \cite{moop} is valid here. As a result
the correspondence for the capped rubber contributions of $Y/Y_\infty$
over $Y_\infty$ is equivalent to the following non-rubber
correspondence.

We consider the relative geometry $Y\ /\ Y_0\cup Y_\infty$
with respect to the 2-torus $T$-action by the
first two factors $T\subset \mathbf{T}$.
Let $\gamma \in H_{\com^*}^*(\mathbf{P}^1_1,\mathbb{Q})$
be the class of the fixed point $\infty \in \mathbf{P}^1_1$.

\begin{Proposition}
\label{ppp555pp} 
For the relative geometry $Y/ Y_0\cup Y_\infty$, we have
$$\ZZ^{\mathsf{P}}_{\beta,\nu,\mu}\Big(\tau_0(\gamma[0])
\Big)^{Y/Y_0\cup Y_\infty,{T}} \ \in \mathbb{Q}(q,s_1,s_2)$$
and the correspondence 
\begin{multline*}
(-q)^{-d_\beta/2}\ZZ^{\mathsf{P}}_{\beta,\nu,\mu}\Big(
 \tau_0(\gamma[0])
\Big)^{Y/Y_0\cup Y_\infty,{T}} \\ =
(-iu)^{d_\beta+\ell(\nu)-|\nu|+ \ell(\mu)-|\mu|}\ZZ'^{\,\mathsf{GW}}_{\beta,\nu,\mu}\Big(\
\overline{\tau_0(\gamma[0])} \
\Big) ^{Y/Y_0\cup Y_\infty,{T}}
\end{multline*}
under the variable change $-q=e^{iu}$.
\end{Proposition}

By basic properties of the descendent correspondence \cite{PPDC},
$$\overline{\tau_0(\gamma[0])}=\tau_0(\gamma[0])\ . $$
Proposition \ref{ppp555p}
is a consequence of Proposition \ref{ppp555pp} together
with the recursion of Lemmas \ref{nnnddd} - \ref{p45q}.
Hence the proof of Theorem \ref{ppp555} will be complete
once Proposition \ref{ppp555pp} is established.

\subsection{Proof of Proposition \ref{ppp555pp}} \label{jjcc3}
The curve class $\beta\in H_2(Y,\mathbb{Z})$ 
is a linear combination of the curves
$$[\mathbf{P}_1^1] = \mathbf{P}_1^1 \times \{0\} \times \{0\}, \ \ \ \ \ 
[\mathbf{P}_3^1] =   \{0\} \times \{0\} \times \mathbf{P}_3^1\ .$$
If $\beta$ is a multiple of $[\mathbf{P}_3^1]$, then Proposition 
\ref{ppp555pp} reduces immediately to the $T$-equivariant  
descendent correspondence of local curves \cite{PPDC}.

Let $\mathbf{Y}=\mathbf{P}_1^1 \times \mathbf{P}_2^1 \times \mathbf{P}_3^1$.
We view the projective variety $\mathbf{Y}$ as 
a $\mathbf{T}$-equivariant compactification of
the quasi-projective variety $Y$,
$$\mathbf{P}_1^1\times \com \times \mathbf{P}_3^1 \subset
\mathbf{P}_1^1 \times \mathbf{P}_2^1 \times \mathbf{P}_3^1  .$$
Let $\mathbf{Y}_0$ and $\mathbf{Y}_\infty$ be the $\pi_3$-fibers of
$\mathbf{Y}$ over $0,\infty \in  \mathbf{P}_3^1$. 
Our proof of Proposition \ref{ppp555pp} will be obtained
from studying the partition functions
\begin{equation} \label{hilu}
\ZZ^{\mathsf{P}}_{\beta,\nu,\mu}\Big(
 \tau_0(\gamma[0])
\Big)^{\mathbf{Y}/\mathbf{Y}_0\cup \mathbf{Y}_\infty,{T}}, \ \ \ \
\ZZ'^{\,\mathsf{GW}}_{\beta,\nu,\mu}\Big(\
\overline{\tau_0(\gamma[0])} \
\Big) ^{  \mathbf{Y}/\mathbf{Y}_0\cup \mathbf{Y}_\infty         ,{T}}
\end{equation}
for the compact relative geometry
$\mathbf{Y}/  \mathbf{Y}_0 \cup \mathbf{Y}_\infty$.
We will consider curve classes
$$\beta = d_1[\mathbf{P}_1^1] + d_3[\mathbf{P}_3^1] \in
H_2(\mathbf{Y},\mathbb{Z})$$
for which $d_1>0$ and $d_3\geq  0$.

If $d_3>0$,
the relative conditions $\nu$ and $\mu$
in \eqref{hilu} will be taken 
to be of a special form. The relative condition
$\nu$ 
is a partition of $d_3$  weighted by $H_T^*(\mathbf{P}_1^1 \times
\mathbf{P}_2^1,\mathbb{Q})$. We {\em require} the weights 
of all the parts $\nu_i$ to be
the pull-backs of the classes the $\com^*_1$-fixed points 
$0,\infty \in \mathbf{P}_1^1$
{\em except} for the
weight of the part $\nu_1$. For $\nu_1$, we take the weight
to be the class of one of the following $T$-fixed points:
$$(0,0), (\infty,0) \in \mathbf{P}_1^1 \times
\mathbf{P}_2^1 \ .$$
For $\mu$,
we {require} {\em all} weights to be 
the pull-backs of the classes of
$0,\infty \in \mathbf{P}_1^1$

The moduli space of stable pairs $P_n(\mathbf{Y}/\mathbf{Y}_0 \cup \mathbf{Y}_\infty,\beta)_{\nu,\mu}$ has virtual dimension
$2d_1+2d_3$ minus the constraints imposed by the boundary
conditions.
The number of constraints imposed by $\nu$ is $d_3+1$ and by $\mu$ is $d_3$.
Hence, the virtual dimension of the stable pairs space is
$$2d_1+2d_3 -2d_3-1\ .$$
The integrand $\tau_0(\gamma[0])$ imposes another
constraint, so the virtual dimension of the integrals
in the  stable pairs
partition function \eqref{hilu}
is 
$2d_1 -2$. The parallel dimension analysis for the Gromov-Witten 
partition function \eqref{hilu} also yields $2d_1-2$.

\begin{Lemma} \label{keed}
For $d_3>0$ with our special boundary conditions $\nu$ and $\mu$,
we have
$$\ZZ^{\mathsf{P}}_{\beta,\nu,\mu}\Big(\tau_0(\gamma[0])
\Big)^{\mathbf{Y}/\mathbf{Y}_0 \cup \mathbf{Y}_\infty,{T}} \ \in \mathbb{Q}(q,s_1,s_2)$$
and the correspondence 
\begin{multline*}
(-q)^{-d_\beta/2}\ZZ^{\mathsf{P}}_{\beta,\nu,\mu}\Big(
 \tau_0(\gamma[0])
\Big)^{\mathbf{Y}/\mathbf{Y}_0 \cup \mathbf{Y}_\infty,{T}} \\ =
(-iu)^{d_\beta+\ell(\nu)-|\nu|+ \ell(\mu)-|\mu|}\ZZ'^{\,\mathsf{GW}}_{\beta,\nu,\mu}\Big(\
\overline{\tau_0(\gamma[0])} \
\Big) ^{\mathbf{Y}/\mathbf{Y}_0 \cup \mathbf{Y}_\infty,{T}}
\end{multline*}
under the variable change $-q=e^{iu}$.
\end{Lemma}

\begin{proof}
We can assume $d_1>0$, then $2d_1-2\geq 0$.
If $2d_1-2>0$, the both the stable pairs and Gromov-Witten 
partition functions vanish by the compactness of the 
geometry.
If $2d_1-2=0$, then both partition functions are
independent of the equivariant parameters $s_1$ and $s_2$.
The required matching then follows from Theorem \ref{ttt9999}.
\end{proof}

We can apply $T$-equivariant localization to both sides
of the correspondence of Lemma \ref{keed}. Via the action of
the second factor of $T$, the $T$-equivariant contributions occur
with $\mathbf{P}_2^1$ coordinate either $0$ or $\infty$ (remember
the curve class $\beta$ is degree 0 over $\mathbf{P}_2^1$).
The localization contributions where the entire curve
$\beta$ and all the boundary conditions lie  over $0\in \mathbf{P}_2^1$
yield{\footnote{Up to a harmless $s_2$ factor obtained
from the cohomology weight of the part $\nu_1$.}} the residue invariants 
appearing in Proposition \ref{ppp555pp}. 
All the other terms in the localization formula can be expressed
as the residue invariants of Proposition \ref{ppp555pp} (over
$0$ or $\infty\in \mathbf{P}_2^1$) 
with lesser $\beta$. Hence the $T$-equivariant localization
relation applied to Lemma \ref{keed} reduces 
Proposition \ref{ppp555pp} to the case where $d_3=0$.

To prove the $d_3=0$ case of Proposition \ref{ppp555pp}, we 
replace Lemma \ref{keed} with a different partition function.
Let 
$$\gamma_0 \in H^*_T(\mathbf{P}_1^1\times \mathbf{P}_2^1,\mathbb{Q})$$
be the class of the point $(\infty,0)$.
Alternatively, $\gamma_0$ is the intersection of $\gamma$ with the
divisor over $0\in \mathbf{P}_2^1$.
Hence, $\gamma_0$ restricted to $\mathbf{P}^1_1 \times \{0\} \times
\mathbf{P}^1_3$ is $-s_2\gamma$.

\begin{Lemma} \label{keed2}
For $d_3=0$,
we have
$$\ZZ^{\mathsf{P}}_{\beta,\emptyset,\emptyset}\Big(\tau_0(\gamma_0[0])
\Big)^{\mathbf{Y}/\mathbf{Y}_0 \cup \mathbf{Y}_\infty,{T}} \ \in \mathbb{Q}(q,s_1,s_2)$$
and the correspondence 
\begin{multline*}
(-q)^{-d_\beta/2}\ZZ^{\mathsf{P}}_{\beta,\emptyset,\emptyset}\Big(
 \tau_0(\gamma_0[0])
\Big)^{\mathbf{Y}/\mathbf{Y}_0 \cup \mathbf{Y}_\infty,{T}} \\ =
(-iu)^{d_\beta}\ZZ'^{\,\mathsf{GW}}_{\beta,\emptyset,\emptyset}\Big(\
\overline{\tau_0(\gamma_0[0])} \
\Big) ^{\mathbf{Y}/\mathbf{Y}_0 \cup \mathbf{Y}_\infty,{T}}
\end{multline*}
under the variable change $-q=e^{iu}$.
\end{Lemma}

\begin{proof}
The dimension analysis used in the proof of Lemma \ref{keed}
is also valid here and yields the result.
\end{proof}

Finally, we can apply $T$-equivariant localization to both sides
of the correspondence of Lemma \ref{keed2}. 
The localization contributions where the entire curve
$\beta$  lies  over $0\in \mathbf{P}_2^1$
yield{\footnote{Up to a harmless $s_2$ factor obtained
from $\gamma_0$.}} the residue invariants 
appearing in Proposition \ref{ppp555pp}. 
All the other terms in the localization formula 
include unconstrained curves over $\infty\in \mathbf{P}_2^1$
with positive $[\mathbf{P}_1^1]$ components --- all such
contributions vanish.{\footnote{The proof can be obtained
inductively
from the geometry $\mathbf{Y}/\mathbf{Y}_0 \cup \mathbf{Y}_\infty$
by considering the integrand $\tau_0(\gamma_0)$.
We leave the details
as an exercise for the reader.}}
The $T$-equivariant localization
relation applied to Lemma \ref{keed2} 
completes the proof of Proposition \ref{ppp555pp}
\qed

\vspace{10pt}
Proposition \ref{ppp555pp} was the last step in the
proof of  
Proposition \ref{ppp555p}.  The proof of Proposition
\ref{ppp555p} completes
the proof of Theorem \ref{ppp555}.

\section{Bi-relative residue theories} \label{xxx3}
\subsection{Overview}
In order to prove Theorem \ref{qqq111}, we must
first extend
Theorem \ref{ttt999} to non-toric surfaces $S$ which are
projective bundles over higher genus curves (as discussed
in Section \ref{ntss}). Our strategy is to extend Proposition
\ref{aaa999} to such surfaces. The extension of Theorem \ref{ttt999} 
then follows as a consequence.

In order to extend Proposition \ref{aaa999} to
projective bundles $S$ over higher genus curves, we will
degenerate $S$. To carry out the argument,
we will require a descendent correspondence for a particular residue theory 
in a bi-relative 3-fold geometry.

\subsection{Bi-relative geometry}
Following the notation of Section \ref{jjcc3}, let
$$\mathbf{Y}= \mathbf{P}_1^1\times \mathbf{P}_2^1 \times
\mathbf{P}_3^1\ , \ \ \ 
\mathbf{Y}_\infty = 
\mathbf{P}_1^1\times \mathbf{P}_2^1 \times
\{\infty\}\ ,
$$
and let 
$\mathbf{D}_\infty\subset \mathbf{Y}$ be the divisor
$$\mathbf{D}_\infty = \mathbf{P}_1^1\times \{\infty \} \times \mathbf{P}_3^1\ .$$
We will consider the bi-relative 3-fold geometry
\begin{equation}\label{dcct2}
\mathbf{Y} \ / \ \mathbf{Y}_\infty \cup \mathbf{D}_\infty \ .
\end{equation}
Since the divisors $\mathbf{Y}_\infty$ and $\mathbf{D}_\infty$
intersect, the full stable pairs and Gromov-Witten theories of
the geometry \eqref{dcct2} are not described by standard
relative techniques \cite{IP,LR}. 

Fortunately, we are only interested here in a corner of the bi-relative geometry
\eqref{dcct2} which can be approached by standard relative
geometry --- the residue theory over $0\in \mathbf{P}_2^1$.
To define the residues over $0\in \mathbf{P}_2^1$, curves
intersecting $\mathbf{Y}_\infty \cap \mathbf{D}_\infty$ do {\em not}
arise, so the standard relative methods are sufficient.

The descendent correspondence
for residue theory of \eqref{dcct2} over $0\in \mathbf{P}_2^1$
will play a crucial role in the extension of Proposition \ref{aaa999} and
Theorem \ref{ttt999}.

\subsection{Construction I} \label{kex1}
Consider the moduli space of stable pairs
$P_n(\mathbf{Y}/\mathbf{Y}_\infty,\beta)_\eta$ with curve class
$$\beta=d_1[\mathbf{P}^1_1]+d_2[\mathbf{P}^1_2] + d_3[\mathbf{P}^1_3]\ $$
and $\com^*_1\times \com^*_2$-equivariant
relative condition $\eta$ along $\mathbf{Y}_\infty$
with cohomology weights supported on
$$\mathbf{P}_1^1 \times \{0\} \times \{\infty\} \subset 
\mathbf{Y}_\infty \ .$$  
Define the open set
$$V_{n,\beta,\eta} \subset P_n(\mathbf{Y}/\mathbf{Y}_\infty,\beta)_\eta$$ 
to be the locus of stable pairs for which
$\mathbf{D}_\infty$ is not a zero divisor.

More precisely, a stable pair in the relative
geometry $\mathbf{Y}/\mathbf{Y}_\infty$
is supported on a destabilization $\widetilde{\mathbf{Y}}$
of $\mathbf{Y}$
along $\mathbf{Y}_\infty$.
The original divisor $\mathbf{D}_\infty$ degenerates to
the reducible divisor
$$\widetilde{\mathbf{D}}_\infty= \pi_2^{-1}(\infty) \subset \widetilde{\mathbf{Y}}, \ \ \ \
\pi_2: \widetilde{\mathbf{Y}}\rightarrow \mathbf{P}_2^1 \ .$$
We define $V_{n,\beta,\eta}$ to be the open
set of stable pairs for which $\widetilde{\mathbf{D}}_\infty$
is not a zero divisor.{\footnote{The moduli space
$P_n(\mathbf{Y}/\mathbf{Y}_\infty,\beta)_\eta$ is {\em not} relative to 
$\widetilde{\mathbf{D}}_\infty$, so transversality along 
$\widetilde{\mathbf{D}}_\infty$ is a non-trivial condition.
There is no bubbling along $\widetilde{\mathbf{D}}_\infty$.}}
In other words,
the stable pair is {\em transverse} to 
$\widetilde{\mathbf{D}}_\infty$.
Via the intersection with $\widetilde{\mathbf{D}}_\infty$, we
obtain a canonical map{\footnote{The map involves possible stabilization.}}
$$\epsilon: V_{n,\beta,\eta} \rightarrow \text{Hilb}(\mathbf{P}_1^1
\times \mathbf{P}_3^1\, /\, \mathbf{P}_1^1\times \{\infty\}, d_2)\ . $$
Here, $\text{Hilb}(\mathbf{P}_1^1
\times \mathbf{P}_3^1\, /\, \mathbf{P}_1^1\times \{\infty\}, d_2)$
is the Hilbert scheme{\footnote{The 
Hilbert scheme of points of a surface relative
to a divisor is a special case of the
relative ideal sheaf moduli for 3-folds. See
 \cite{Iman} for a discussion and study.}} of $d_2$ points of the surface
$\mathbf{P}_1^1
\times \mathbf{P}_3^1$ relative to the divisor
$\mathbf{P}_1^1
\times \{\infty \}$. 

The original 3-torus $\mathbf{T}$ acting on $\mathbf{Y}$ acts
on $V_{n,\beta,\eta}$.
While $V_{n,\beta,\eta}$ is certainly not compact, the 
$\com^*_2$-fixed
point locus is compact ---
all features of the stable pair occur on
$$\widetilde{\mathbf{D}}_0=\pi_2^{-1}(0) \subset \widetilde{\mathbf{Y}}\ .$$  
A $\com^*_2$-fixed stable pair in $V_{n,\beta,\eta}$
meets $\widetilde{\mathbf{D}}_\infty$
transversely. On $\widetilde{\mathbf{Y}} \setminus 
\widetilde{\mathbf{D}}_0$, $\com^*_2$-fixed stable pairs
are simply the pull-backs of $0$-dimensional subschemes
of $\widetilde{\mathbf{D}}_\infty$.
All components of positive degree over $\mathbf{P}_1^1 \times \mathbf{P}_3^1$
of  $\com_2^*$-fixed curves associated to stable pairs
in $V_{n,\beta,\eta}$ 
lie over $0\in \mathbf{P}_2^1$.

Let $\theta_j, \theta'_{j'} \in H^*_{\com^*_1}(\mathbf{P}_1^1,\mathbb{Q})$
be as in Section \ref{erst}. Let
$$[0,0],[0,\infty] \in H^*_{\com_2^*\times \com_3^*}(\mathbf{P}_2^1\times \mathbf{P}_3^1)$$
denote the classes of the points $(0,0)$ and $(0,\infty)$
respectively.
Let $$\phi \in H^*_{\com_1^*\times \com_3^*}(
\text{Hilb}(\mathbf{P}_1^1
\times \mathbf{P}_3^1\, /\, \mathbf{P}_1^1\times \{\infty\}, d_2),\mathbb{Q})\ .$$
We define the uncapped residue descendent series 
\begin{multline*}
{\mathsf V}^{\mathsf{P}}_{\beta}
\left(   \prod_{j=1}^k \tau_{i_j}(\theta_j[0,0]) \
\prod_{j'=1}^{{k'}} \tau_{i'_{j'}}(\theta_{j'}'[0,\infty])
\right)^{\mathbf{Y}/\mathbf{Y}_\infty\cup\mathbf{D}_\infty, \,  \mathbf{T}}_{\eta,\phi}
 = \\ 
\sum _{n\in \Z }q^{n}
\int _{[V_{n,\beta,\eta}]^{vir}}
  \prod_{j=1}^k \tau_{i_j}(\theta_j[0,0]) \
\prod_{j'=1}^{k'} \tau_{i'_{j'}}(\theta'_{j'}[0,\infty]) \cup 
\epsilon^*(\phi)
\ 
\end{multline*}
by $\mathbf{T}$-equivariant residues.

\subsection{Construction II}
Next, we consider the moduli of stable
pairs for the relative geometry
\begin{equation}\label{pss39}
\mathbf{Y} \  /  \ \mathbf{Y}_\infty \cup \mathbf{D}_\infty \ .
\end{equation}
with curve classes $d_2[\mathbf{P}_2^1]$.
Since $$[\mathbf{P}_2^1]\cdot \mathbf{Y}_\infty=0\ ,$$
the curves never meet
$\mathbf{Y}_\infty$. So the delicate study of
geometry 
relative to the singularities of 
$\mathbf{Y}_\infty \cup  \mathbf{D}_\infty$ can be completely avoided.
The moduli space
$$P_{n}(\mathbf{Y} /  \mathbf{Y}_\infty \cup \mathbf{D}_\infty,
d_2[\mathbf{P}_2^1])\ $$
is easily constructed.
The projections of
the curves to 
$$\mathbf{P}_1^1 \times \{0\}\times \mathbf{P}_3^1$$
 are never allowed to meet 
$$\mathbf{P}_1^1 \times \{0\} \times \{\infty\} \subset 
\mathbf{Y}_\infty \ .$$ 
Bubbling
occurs along $\mathbf{Y}_\infty$ to keep the projections away.
The points of the resulting
moduli space corresponds
to stable pairs which meet 
$\widetilde{\mathbf{D}}_\infty$ away from the intersection with 
$\mathbf{Y}_\infty$. Hence,
the deformation theory and virtual class are standard.

The boundary conditions 
along ${\mathbf{D}}_\infty$ are defined via the 
canonical map
$$\epsilon: 
P_{n}(\mathbf{Y} /  \mathbf{Y}_\infty \cup \mathbf{D}_\infty,
d_2[\mathbf{P}_2^1])
 \rightarrow \text{Hilb}(\mathbf{P}_1^1
\times \mathbf{P}_3^1\, /\, \mathbf{P}_1^1\times \{\infty\}, d_2)\ . $$
While any element of the cohomology of $\text{Hilb}(\mathbf{P}_1^1
\times \mathbf{P}_3^1\, /\, \mathbf{P}_1^1\times \{\infty\}, d_2)$
imposes a boundary condition, special elements
corresponding to the Nakajima basis of the cohomology of the
Hilbert scheme of points in the absolute case will
play a distinguished role.

Let $\mu$  
be partition of $d_2$ weighted by the cohomology of
the surface $\mathbf{P}_1^1\times \mathbf{P}_3^1$. 
Explicitly,
\begin{equation}\label{k33k} 
\mu=\{(\mu_1,\omega_1), \ldots, (\mu_\ell,\omega_\ell)\}, \ \ 
d_2 = \sum_{i=1}^\ell \mu_i, \ \  \omega_i \in
H_{\com^*_1\times \com^*_3}^*(\mathbf{P}_1^1\times \mathbf{P}_3^1,\mathbb{Q}).
\end{equation}
Such a weighted partition determines an element
$$\mathsf{Nak}(\mu) \in H^*_{\com^*_1\times \com^*_3}( \text{Hilb}(\mathbf{P}_1^1
\times \mathbf{P}_3^1\, /\, \mathbf{P}_1^1\times \{\infty\}, d_2),
\mathbb{Q})$$
by the following construction. 
Recall
$$\big(\mathbf{P}_1^1\times \mathbf{P}_3^1\, /\, \mathbf{P}_1^1
\times \{\infty\}\big)^{\ell} \rightarrow
\mathbf{P}_1^1\times \mathbf{P}_3^1 $$
is the space of ordered points in the relative surface
geometry, see Section \ref{diagclas}.
The cohomology weights $\omega_i$ pull-back canonically to
the space of points  
$\big(\mathbf{P}_1^1\times \mathbf{P}_3^1\, /\, \mathbf{P}_1^1
\times \{\infty\}\big)^{\ell}$.
Let
$$\mathsf{C}_\mu \subset \big(\mathbf{P}_1^1\times \mathbf{P}_3^1\, /\,
\mathbf{P}_1^1
\times \{\infty\}\big)^{\ell}  \times
\text{Hilb}\big(\mathbf{P}_1^1
\times \mathbf{P}_3^1\, /\, \mathbf{P}_1^1\times \{\infty\}, d_2\big)$$
be the closure of the locus of distinct points
in $(\mathbf{P}_1^1\times \mathbf{P}_3^1\, / \, \mathbf{P}_1^1
\times \{\infty\})^{\ell}$ carrying punctual subschemes of
lengths $\mu_1, \ldots, \mu_{\ell(\mu)}$.
Let
$$\mathsf{Nak}(\mu) = \frac{1}{\zz(\mu)} \text{pr}_{2*}\big(  \mathsf{C}_\mu
\cdot \text{pr}_1^*(\omega_1\cup \ldots \cup \omega_\ell)\big)$$
with respect to the projections of
$$(\mathbf{P}_1^1\times \mathbf{P}_3^1\, /\, \mathbf{P}_1^1
\times \{\infty\})^{\ell}  \times
\text{Hilb}(\mathbf{P}_1^1
\times \mathbf{P}_3^1\, /\, \mathbf{P}_1^1\times \{\infty\}, d_2)$$
onto the first and second factors.

Let $\mathbf{D}_0\subset \mathbf{Y}$ be the divisor
lying over $0\in \mathbf{P}_2^1$.
We can also consider the rubber moduli spaces of stable pairs
$$P_n(\mathbf{Y}/\mathbf{Y}_\infty \cup \mathbf{D}_0 \cup \mathbf{D}_\infty, d_2[\mathbf{P}_2^1])^\sim$$
which arises in the boundary of 
$P_{n}(\mathbf{Y} /  \mathbf{Y}_\infty \cup \mathbf{D}_\infty,
d_2[\mathbf{P}_2^1])$ over $\mathbf{D}_\infty$.
In addition to the boundary map $\epsilon_\infty$ associated
to $\mathbf{D}_\infty$, there is also a boundary map
$$\epsilon_0: 
P_{n}(\mathbf{Y} /  \mathbf{Y}_\infty \cup \mathbf{D}_0 \cup \mathbf{D}_\infty,
d_2[\mathbf{P}_2^1])^\sim
 \rightarrow \text{Hilb}(\mathbf{P}_1^1
\times \mathbf{P}_3^1\ /\ \mathbf{P}_1^1\times \{\infty\}, d_2) $$
obtained by the intersection with $\widetilde{\mathbf{D}}_0$.

As in Section \ref{rubcon}, we have the cotangent line classes
$$\Psi_0,\Psi_\infty\in H^2_{\com^*_1\times \com^*_3}
({P_n(\mathbf{Y} /  \mathbf{Y}_\infty \cup \mathbf{D}_0\cup \mathbf{D}_\infty,
d_2[\mathbf{P}_2^1]
)}^\sim,
{\mathbb Q})\ .$$
Define the rubber series
\begin{multline*}
{\mathsf R}^{\mathsf{P}}_{d_2[\mathbf{P}_2^1]} \left(\frac{1}{-\Phi_0+s_2}\right
)^{\mathbf{T}}_{\phi,\mu}
= \\
\sum_{n\in \mathbb{Z}} q^n 
\int_{[{P_n(\mathbf{Y} /  \mathbf{Y}_\infty \cup \mathbf{D}_0\cup \mathbf{D}_\infty,
d_2[\mathbf{P}_2^1]
)}^\sim]^{vir}} \frac{1}{-\Phi_0+s_2} \cdot
\epsilon_0^*(\phi) \cup \epsilon_\infty^*(\mathsf{Nak}(\mu)) \ .
\end{multline*}
Here, $\phi 
\in H^*_{\com^*_1\times \com^*_3}( \text{Hilb}(\mathbf{P}_1^1
\times \mathbf{P}_3^1\ /\ \mathbf{P}_1^1\times \{\infty\}, d_2),
\mathbb{Q})$
is an arbitrary class.

\subsection{Definition of the bi-relative residue} \label{kex3}
We define the bi-relative capped descendent residue theory
\begin{equation*}
{\mathsf C}^{\mathsf{P}}_{0} 
\left(   \prod_{j=1}^k \tau_{i_j}(\theta_j[0,0]) \
\prod_{j'=1}^{{k'}} \tau_{i'_{j'}}(\theta_{j'}'[0,\infty]),\beta
\right)^{\mathbf{Y}/\mathbf{Y}_\infty\cup\mathbf{D}_\infty, \,  \mathbf{T}}_{\eta,\mu}
\end{equation*}
by the formula
\begin{multline*}
\sum_{i} {\mathsf V}^{\mathsf{P}}_{\beta} 
\left(   \prod_{j=1}^k \tau_{i_j}(\theta_j[0,0]) \
\prod_{j'=1}^{{k'}} \tau_{i'_{j'}}(\theta_{j'}'[0,\infty])
\right)^{\mathbf{Y}/\mathbf{Y}_\infty\cup\mathbf{D}_\infty, \,  \mathbf{T}}_{\eta,\phi_i} \\
\cdot 
q^{-d_2}\ 
{\mathsf R}^{\mathsf{P}}_{d_2[\mathbf{P}_2^1]} \left(\frac{1}{-\Phi_0+s_2}
\right)_{\phi_i^\vee,\mu}^{\mathbf{T}}
\ 
\end{multline*}
where the sum is over the components
of a $\com^*_1 \times \com^*_3$-equivariant
K\"unneth decomposition
$$\sum_{i=1}^f \phi_i \otimes \phi_i^\vee = [\Delta] \in
H_{\com^*_1\times \com^*_3}( \text{Hilb} \times \text{Hilb},\mathbb{Q})$$
 of the diagonal of $\text{Hilb}(\mathbf{P}_1^1
\times \mathbf{P}_3^1\ /\ \mathbf{P}_1^1\times \{\infty\}, d_2)$.

\subsection{Motivation}
We have given above a rigorous definition of the bi-relative capped
descendent residue theory. If we had a complete definition of  
the stable pairs theory of the bi-relative geometry
$\mathbf{Y}/  \mathbf{Y}_\infty \cup \mathbf{D}_\infty$, the definition
of 
\begin{equation}\label{gtbbn}
{\mathsf C}^{\mathsf{P}}_{0} 
\left(   \prod_{j=1}^k \tau_{i_j}(\theta_j[0,0]) \
\prod_{j'=1}^{{k'}} \tau_{i'_{j'}}(\theta_{j'}'[0,\infty]),\beta
\right)^{\mathbf{Y}/\mathbf{Y}_\infty\cup\mathbf{D}_\infty, \,  \mathbf{T}}_{\eta,\mu}
\end{equation}
as a capped residue theory would be immediate.
Since we are interested in the residue theory over $0\in \mathbf{P}_2^1$,
the stable pairs do not interact with the singularities of
$\mathbf{Y}_\infty \cup \mathbf{D}_\infty$, and we are
able to define \eqref{gtbbn} by hand.

\subsection{Gromov-Witten theory}
Following every step of the stable pairs construction,
we can also define a bi-relative capped
descendent residue theory for stable maps,
\begin{equation}\label{gtbbnn}
{\mathsf C}^{\mathsf{GW}}_{0} 
\left(   \prod_{j=1}^k \tau_{i_j}(\theta_j[0,0]) \
\prod_{j'=1}^{{k'}} \tau_{i'_{j'}}(\theta_{j'}'[0,\infty]),\beta
\right)^{\mathbf{Y}/\mathbf{Y}_\infty\cup\mathbf{D}_\infty, \,  \mathbf{T}}_{\eta,\mu} \ .
\end{equation}
Moreover, the depth induction techniques of Sections
\ref{xxx1} -- \ref{xxx2} applied to both  the descendent
insertions and to the parts of $\mu$ yield the
following correspondence.

\begin{Theorem}
\label{ppp666} 
We have
$$\mathsf{C}^{\mathsf{P}}_{0}\Big(
\prod_{j=1}^k \tau_{i_j}(\theta_j[0,0]) \
\prod_{j'=1}^{{k'}} \tau_{i'_{j'}}(\theta_{j'}'[0,\infty]),\beta
\Big)^{\mathbf{Y}/
\mathbf{Y}_\infty\cup\mathbf{D}_\infty,\mathbf{T}}_{\eta,\mu} \ \in \mathbb{Q}(q,s_1,s_2,s_3)$$
and the correspondence 
\begin{multline*}
(-q)^{-d_\beta/2}\mathsf{C}^{\mathsf{P}}_{0}\Big(
\prod_{j=1}^k \tau_{i_j}(\theta_j[0,0]) \
\prod_{j'=1}^{{k'}} \tau_{i'_{j'}}(\theta_{j'}'[0,\infty]),\beta
\Big)^{\mathbf{Y}/\mathbf{Y}_\infty\cup\mathbf{D}_\infty,\mathbf{T}}_{\eta,\mu} = 
\\
(-iu)^{d_\beta+\ell(\eta)-|\eta|+\ell(\mu)-|\mu|}\mathsf{C}^{\mathsf{GW}}_{0}\Big(\
\overline{\prod_{j=1}^k \tau_{i_j}(\theta_j[0,0]) \
\prod_{j'=1}^{k'} \tau_{i'_{j'}}(\theta_{j'}'[0,\infty])}\ ,\beta
\Big) ^{\mathbf{Y}/\mathbf{Y}_\infty\cup\mathbf{D}_\infty,\mathbf{T}}_{\eta,\mu}
\end{multline*}
under the variable change $-q=e^{iu}$.
\end{Theorem}

\begin{proof}
We take the relative condition $\mu$ of the form \eqref{k33k}
to have cohomology weights
$$\omega_i = \  \gamma_i[0] \ \ \text{or}  \ \ \gamma_i[\infty]$$
where $\gamma_i\in H^*_{\com^*_1}(\mathbf{P}_1^1,\mathbb{Q})$
and $[0],[\infty]\in  H^*_{\com^*_3}(\mathbf{P}_3^1,\mathbb{Q})$
are the classes of the $\com^*_3$-fixed points.

To prove Theorem \ref{ppp666},
we exactly follow the depth induction used in the proof of
Theorems \ref{ppp444} and \ref{ppp555}. The depth count has two
components:
\begin{enumerate}
\item[$\bullet$] the number of descendent insertions of the
 form $\tau_{i'_{j'}}(\theta_{j'}'[0,\infty])$,
\item[$\bullet$] the number of parts of $\mu$ with weights of
the form $\gamma[\infty]$.
\end{enumerate}

There is no difficulty in including the parts of $\mu$ over
$\infty\in \mathbf{P}_3^1$ in the $\mathbf{T}$-equivariant
localization formula
of Lemma \ref{p45}. The descendents over $\infty\in \mathbf{P}_3^1$
were used to rigidify the rubber --- we can also use the
parts of $\mu$ to rigidify the rubber.
The outcome is a reduction of Theorem \ref{ppp666} to the
base case where all the descendent insertions and parts of $\mu$
lie over $0\in \mathbf{P}_3^1$. 

Theorem \ref{ppp666} in the base case concerns only 3-leg
descendent vertices at the points 
$(0,0,0),(\infty,0,0) \in \mathbf{Y}$
and the capped rubber contributions over $\infty\in \mathbf{P}_3^1$.
The GW/P correspondence for the 3-leg descendent
vertex has been established in \cite{PPDC}. The correspondence
for the capped
rubber of $\infty \in \mathbf{P}_3^1$ been treated already in  
Section \ref{fvvf} via Proposition \ref{ppp555pp}.
\end{proof}

\subsection{Degeneration}\label{kdd9}
Let $S$ be a nonsingular projective surface equipped with two line bundles
$L_0$ and $L_\infty$. Let
$$\pi:\mathbf{P}_S \rightarrow S$$
be the $\mathbf{P}^1$-bundle obtained from the
projectivization of the sum $L_0 \oplus L_\infty$.
The fiberwise $\com^*$-action on $\mathbf{P}_S$ leaves the divisor
 $$S_\infty = \mathbf{P}(L_\infty) \subset \mathbf{P}_S\ $$
invariant.
Let $C\subset S$ be a nonsingular curve, and let
$$\mathbf{P}_C = \pi^{-1}(C) \subset \mathbf{P}_S\ .$$

Via the fiberwise $\com^*$-action, we
can define capped bi-relative residue theories 
for the
geometry
$$\mathbf{P}_S\, /\, \mathbf{P}_C \cup S_\infty$$
for stable pairs and stable maps.
The constructions of Sections \ref{kex1}\,--\,\ref{kex3} apply here: only the
fiberwise $\com^*_2$-action was needed there to define the
bi-relative residue theories. We therefore have capped bi-relative
residue theories
\begin{equation}\label{lqq2}
{\mathsf C}^{\mathsf{P}}_{0} 
\left(   \prod_{j=1}^k \tau_{i_j}(\gamma_j) , \beta
\right)^{\mathbf{P}_S/\mathbf{P}_C \cup {S}_\infty, \,  \com^*}_{\eta,\mu}\ , \ \ \ \
{\mathsf C}^{\mathsf{GW}}_{0} 
\left(\   \overline{\prod_{j=1}^k \tau_{i_j}(\gamma_j)} 
,\beta\right)^{\mathbf{P}_S/\mathbf{P}_C \cup {S}_\infty, \,  \com^*}_{\eta,\mu}\ 
\end{equation}
where $\gamma_1,\ldots, \gamma_k \in H^*_{\com^*}(\mathbf{P}_S,\mathbb{Q})$
are classes supported on $S_0$,
$\eta$ is a $\com^*$-equivariant boundary condition along
$\mathbf{P}_C$ with support on $\mathbf{P}_C \cap S_0$,
and 
$$\text{Nak}(\mu)
\in H^*( \text{Hilb}(S\, /\, C, |\mu|),
\mathbb{Q})$$
is a Nakajima element.

The capped bi-relative residue theories occur naturally
in the degeneration formula.
Let
$$
\pi:\mathfrak{S}\to \Delta
$$
be a nonsingular $3$-fold fibered over an irreducible
nonsingular base curve $\Delta$. Let $S$ be a nonsingular
fiber, and let
$$
A \cup_{C} B
$$
be a reducible special fiber consisting of two nonsingular
surfaces intersecting transversally along a nonsingular
surface $C$.
Let
$$\mathfrak{L}_0, \mathfrak{L}_\infty \rightarrow \mathfrak{S}$$
be two line bundles.
The degeneration
\begin{equation}\label{lzxz}
 \mathbf{P}_{\mathfrak{S}} = \mathbf{P}(\mathfrak{L}_0\oplus \mathfrak{L}_\infty) \rightarrow B
\end{equation}
is a nonsingular 4-fold with a reducible fiber
$$\mathbf{P}_{S_1} \cup_{\mathbf{P}_C} \mathbf{P}_{S_2}\ .$$
Let $[\mathbf{P}] \in H_2(\mathbf{P}_{\mathfrak{S}},\mathbb{Z})$
be the curve class of the $\mathbf{P}^1$-fiber.

To write the degeneration formula corresponding to the
geometry \eqref{lzxz}, we require the following notation:
\begin{enumerate}
\item[$\bullet$]
Let 
$\beta = \beta_{S_0} + d [\mathbf{P}] \in H_2(\mathbf{P}_S,\mathbb{Z})$ where
$\beta_{S_0} \in H_2(S_0,\mathbb{Z})$.
\item[$\bullet$]
Let $\gamma_1, \ldots,\gamma_k\in H^*_{\com^*}(\mathbf{P}_{\mathfrak{S}},\mathbb{Q})$
be classes supported on $\mathfrak{S}_0$.
\item[$\bullet$] Let $\mu$ be a partition of $d$ with 
cohomology weights lying in $H^*(\mathfrak{S}_\infty,\mathbb{Q})$.
\end{enumerate}
The degeneration formula for stable pairs is
\begin{multline*}
\mathsf{C}^{\mathsf{P}}_0\Big( \prod_{j=1}^k \tau_{i_j}(\gamma_j),\beta \Big)_\mu   
^{\mathbf{P}_S/S_\infty,\com^*}
=\\
{\mathlarger{\mathlarger{\mathlarger{\sum}}}}\  
\mathsf{C}^{\mathsf{P}}_0\Big(   
\prod_{j\in J_1}\tau_{i_j}(\gamma_j),\beta_1\Big)_{\eta,\mu^1}
^{\mathbf{P}_A/\mathbf{P}_C\cup A_\infty, \com^*}\ 
(-1)^{|\eta|-\ell(\eta)}
\zz(\eta)q^{-|\eta|} 
\\
\cdot \mathsf{C}^{\mathsf{P}}_0\Big(  
{\prod_{i\in J_2}\tau_{i_j}(\gamma_j),\beta_2}\Big)_{\eta^\vee,\mu^2}
^{\mathbf{P}_B/\mathbf{P}_C\cup B_\infty, \com^*}
\ .
\end{multline*}
The sum is over all distributions of descendents,
distributions of $\mu$,
 and curve class splittings 
$$J_1\cup J_2 = \{1, \ldots, k\}, \ \  \mu= \mu^1\cup \mu^2, \ \
 \beta=\beta_1+\beta_2,$$
where we have
$$\beta_1 = \beta_{A_0}+d_1[\mathbf{P}], \ \ \beta_2=\beta_{B_0}+d_2
[\mathbf{P}]$$
with $\beta_{A_0} \in H_2(A_0,\mathbb{Z})$,
$\beta_{B_0}\in H_2(B_0,\mathbb{Z})$, and  $d_1+d_2=d$.
The sum is also over a basis $\eta$ of $H^*_{\com^*}(\mathbf{P}_C\setminus
\mathbf{P}_C\cap \mathfrak{S}_\infty,\mathbb{Q})$
supported on $\mathbf{P}_C\cap \mathfrak{S}_0$.

The above degeneration formula is straightforward consequence
of the standard degeneration formula for stable pairs
residue theories
and the definition of the 
bi-relative integrals.{\footnote{Since the curves of the
residue theory with positive $S$-degree lie in $S_0$ before degeneration,
the limits lie in $A_0$ and $B_0$ after degeneration.}}
We leave the details to the reader.

The degeneration formula for Gromov-Witten theory
takes a parallel form,
\begin{multline*}
\mathsf{C}^{\mathsf{GW}}_0\Big( \overline{\prod_{j=1}^k \tau_{i_j}(\gamma_j)},\beta \Big)_\mu   
^{\mathbf{P}_S/S_\infty,\com^*}
=\\
{\mathlarger{\mathlarger{\mathlarger{\sum}}}}\  
\mathsf{C}^{\mathsf{GW}}_0\Big(   
\overline{\prod_{j\in J_1}\tau_{i_j}(\gamma_j)},\beta_1\Big)_{\eta,\mu^1}
^{\mathbf{P}_A/\mathbf{P}_C\cup A_\infty, \com^*}\ 
\zz(\eta)u^{2\ell(\eta)}
\\
\cdot \mathsf{C}^{\mathsf{GW}}_0\Big(  
\overline{{\prod_{i\in J_2}\tau_{i_j}(\gamma_j)}},\beta_2 \Big)_{\eta^\vee,\mu^2}
^{\mathbf{P}_B/\mathbf{P}_C\cup B_\infty, \com^*}
\, 
\end{multline*}
with the same summation conventions.
The correspondence
$${\prod_{j=1}^k \tau_{i_j}(\gamma_j)} \mapsto
\overline{\prod_{j=1}^k \tau_{i_j}(\gamma_j)}$$
is defined via the conventions of Sections \ref{pwwf} for relative
geometries. The relative diagonals and log tangent bundle
are used.

The above degeneration formulas are compatible with the
natural generalization of Conjecture \ref{ttt444} for capped 
bi-relative residue theories.

\begin{Conjecture} \label{ttt555}
 For the theories \eqref{lqq2},
we have
$$\mathsf{C}^{\mathsf{P}}_{0}\Big(
\prod_{j=1}^k \tau_{i_j}(\gamma_j),\beta
\Big)^{\mathbf{P}_S/\mathbf{P}_C \cup S_\infty, \com^*}
_{\eta,\mu} \ \in \mathbb{Q}(q,t)$$
and the correspondence 
\begin{multline*}
(-q)^{-d_\beta/2}\mathsf{C}^{\mathsf{P}}_{0}\Big(
\prod_{j=1}^k \tau_{i_j}(\gamma_j),\beta
\Big)^{{\mathbf{P}_S/\mathbf{P}_C \cup S_\infty, \com^*}}_{\eta,\mu} = 
\\
(-iu)^{d_\beta+\ell(\eta)-|\eta|+\ell(\mu)-|\mu|}\mathsf{C}^{\mathsf{GW}}_{0}\Big(\
\overline{\prod_{j=1}^k \tau_{i_j}(\gamma_j)},\beta
\Big) ^{{\mathbf{P}_S/\mathbf{P}_C \cup S_\infty, \com^*}}_{\eta,\mu}
\end{multline*}
under the variable change $-q=e^{iu}$.
\end{Conjecture}

The conditions imposed on $\beta$, $\gamma_j$, $\mu$, and $\eta$
in Conjecture \ref{ttt555}
are as discussed for the degeneration formula.

\subsection{Review}
Theorems \ref{ppp444}\,--\,\ref{ppp666} are parallel results.
The strategy of depth induction is the main idea in the
proof of Theorem \ref{ppp444} for descendents on the cap.
The base case is the correspondence for the 1-leg 
capped descendent vertex of \cite{PPDC}.
For the relative geometry 
$\Pp\times \com \times
\Pp \ /\ \mathbf{P}^1  \times \com$, the
same depth induction is valid, but the base case,
settled in Proposition \ref{ppp555p}, is new.
Finally, for the bi-relative geometry of Theorem \ref{ppp666},
the relative constraints along $\mathbf{D}_\infty$ are new.
Fortunately, the relative insertions fit into the
original depth induction.

Theorem \ref{ppp666} and the degeneration
formula of Section \ref{kdd9}  are the main technical results
which will be needed to study descendent
correspondences for projective bundles over
curves.

\section{Projective bundles over higher genus curves}
\label{phgc}

\subsection{Overview}
Let $C$ be an nonsingular projective curve of genus $g$
equipped with a rank 2 vector bundle $\Lambda\rightarrow C$ and two line bundles
$$L^C_{0},  L^C_\infty \rightarrow C\ .$$
Let $S$ be the nonsingular projective surface obtained
by the projectivization of $E$,
$$ S = \mathbf{P}(E) \stackrel{\epsilon}{\longrightarrow} X\ .$$
The 
projective bundle
\begin{equation}\label{kxx5}
\mathbf{P}_S= \mathbf{P}(\epsilon^*L^C_0 \oplus \epsilon^*L^C_\infty) \rightarrow S
\end{equation}
admits sections 
$$ S_i= \mathbf{P}(\epsilon^*L^C_i) \subset \mathbf{P}_S\ .$$
We will establish here
the 
relative descendent correspondence 
of Conjecture \ref{ttt444}
for $\mathbf{P}_S/S_\infty$.

Relative projective bundle geometries over toric surfaces
were studied in Section \ref{n3n3}. We follow here the conventions
and constructions of Section \ref{n3n3}.
Let 
$$\Gamma=(\gamma_1, \ldots, \gamma_\ell), \ \ \ \gamma_i \in H^*(S,\mathbb{Q})\ . $$ 
Since $S$ has odd cohomology, the classes $\gamma_i$ may be of odd
(real) degree.

We consider capped contributions over $S_0$ in curve class
$$\beta = \beta_0 +d[\mathbf{P}], \ \ \ \beta_0\in \text{Eff}(S_0).$$ Let $\mu$ be a
boundary condition along $S_\infty$ with $|\mu|=d$,
$$\mu=\{(\mu_1,\omega_1), \ldots, (\mu_{\ell(\mu)}, \omega_{\ell(\mu)})\}$$
with $\omega_i\in H^*(S_\infty,\mathbb{Q})$.
Again, the classes $\omega_i$ may be of odd (real) degree.

\begin{Proposition}
\label{bbb999}
The $\com^*$-equivariant descendent correspondence 
for the capped contributions over $S_0$ holds for
the
geometry \eqref{kxx5}:  
$$\mathsf{C}^{\mathsf{P}}_0\left(\tau_\alpha(\Gamma_0), \beta\right)_\mu 
^{\mathbf{P}_S/S_\infty,\com^*}
\in
\mathbb{Q}(q,t)$$ and we have
\begin{equation*}
(-q)^{-d_{\beta}/2}\,
\mathsf{C}^{\mathsf{P}}_0\left(\tau_\alpha(\Gamma_0), \beta\right)_\mu
^{\mathbf{P}_S/S_\infty,\com^*}
 =
(-iu)^{d_{\beta}+\ell(\mu)-|\mu|}\,
\mathsf{C}^{\mathsf{GW}}_0\Big(\overline{\tau_\alpha(\Gamma_0)}, \beta\Big)_\mu
^{\mathbf{P}_S/S_\infty,\com^*}
\end{equation*}
under the variable change $-q=e^{iu}$.
\end{Proposition}

The proof of Proposition \ref{bbb999} will be given in Sections \ref{ttaa7}--\ref{togg4}. 
In the toric case studied in Section \ref{n3n3}, Proposition \ref{aaa999}
was shown to formally imply Theorem \ref{ttt999}. For the
geometry \eqref{kxx5}, Proposition \ref{bbb999} implies
the descendent correspondence by the same argument.

\begin{Theorem}
\label{xxx999} 
For the relative geometry $\mathbf{P}_S/S_\infty$
associated to \eqref{kxx5}
and classes
 $\gamma_i \in H^{*}_{\com^*}(\mathbf{P}_S,\mathbb{Q})$, we have 
$$\ZZ_{\mathsf{P}}\Big(\mathbf{P}_S/S_\infty   ;q\ \Big|  
{\tau_{\alpha_1-1}(\gamma_1)\cdots
\tau_{\alpha_{\ell}-1}(\gamma_{\ell})} \ \Big| \ \mu
\Big)_\beta^{\com^*}\in \mathbb{Q}(q,t)$$
and the correspondence
\begin{multline*}
(-q)^{-d_\beta/2}\ZZ_{\mathsf{P}}\Big(\mathbf{P}_S/S_\infty   ;q\ \Big|  
{\tau_{\alpha_1-1}(\gamma_1)\cdots
\tau_{\alpha_{\ell}-1}(\gamma_{\ell})} \ \Big| \ \mu
\Big)_\beta^{\com^*} \\ =
(-iu)^{d_\beta+\ell(\mu)-|\mu|}\ZZ'_{\mathsf{GW}}\Big(\mathbf{P}_S/S_\infty
;u\ \Big|   
\ \overline{\tau_{a_1-1}(\gamma_1)\cdots
\tau_{\alpha_{\ell}-1}(\gamma_{\ell})}
\ \Big| \ \mu\Big)_\beta^{\com^*} 
\end{multline*}
under the variable change $-q=e^{iu}$. \qed
\end{Theorem}

The parallel descendent correspondence holds when the projective
bundle geometry $\mathbf{P}_S/ S_0 \cup S_\infty$
is taken relative to both sections.

\subsection{Torus actions} \label{ttaa7}
If $\Lambda$ splits as a sum of line bundles on $C$,
\begin{equation}\label{ksss}
\Lambda=\Lambda_0 \oplus \Lambda_\infty\ ,
\end{equation}
then $S=\mathbf{P}(\Lambda)$ admits a fiberwise $\com^*$-action
by scaling. The 3-fold 
\begin{equation}\label{kssss}
\mathbf{P}_S = \mathbf{P}(\epsilon^*L_0^C \oplus \epsilon^*L_\infty^C)
\end{equation}
then carries a 2-dimensional torus action
$$\com_1^* \times \com_2^* \times \mathbf{P}_S \rightarrow \mathbf{P}_S$$
where $\com_1^*$ is the scaling associated to the splitting
\eqref{ksss} and $\com^*_2$ is the scaling associated to the
splitting \eqref{kssss}.

In case $\Lambda$  splits, we will prove the natural
$\com_1^*\times \com_2^*$-equivariant
lift of Proposition \ref{bbb999}:
$$\mathsf{C}^{\mathsf{P}}_0\left(\tau_\alpha(\Gamma_0), \beta\right)_\mu 
^{\mathbf{P}_S/S_\infty,\com^*_1\times \com^*_2}
\in
\mathbb{Q}(q,s)$$ and we have
\begin{multline*}
(-q)^{-d_{\beta}/2}\,
\mathsf{C}^{\mathsf{P}}_0\left(\tau_\alpha(\Gamma_0), \beta\right)_\mu
^{\mathbf{P}_S/S_\infty,\com^*_1\times \com^*_2}
 = \\
(-iu)^{d_{\beta}+\ell(\mu)-|\mu|}\,
\mathsf{C}^{\mathsf{GW}}_0\Big(\overline{\tau_\alpha(\Gamma_0)}, \beta\Big)_\mu
^{\mathbf{P}_S/S_\infty,\com^*_1\times \com^*_2}
\end{multline*}
under the variable change $-q=e^{iu}$.

Since every rank 2 bundle $\Lambda$ is deformation equivalent to
a split bundle, we can assume $\Lambda$ is split in the proof of
Proposition \ref{bbb999}.
We will prove the above  $\com_1^*\times \com_2^*$-equivariant
correspondence (which of course then implies the $\com_2^*$-equivariant
statement of Proposition \ref{bbb999}).

\subsection{Invertibility}
Before proving Proposition \ref{bbb999},
we will require an auxiliary result 
for the capped residue theory
\begin{equation}\label{jxxq}
\mathsf{C}^{\mathsf{P}}_0\Big( \prod_{j=1}^k \tau_{i_j}(\theta_j[0,0]),
\ d[\mathbf{P}^1_3]
\Big)_{\eta,\emptyset}^{\mathbf{Y}/\mathbf{Y}_\infty\cup 
\mathbf{D}_\infty,\mathbf{T}}
\end{equation}
derived from the analysis of stable pairs descendents in
\cite{PPstat}.

Let ${\mathcal P}(d,2)$ be the set of
pairs of partitions $(\eta^0,\eta^\infty)$
satisfying
$$|\eta^0| + |\eta^\infty| = d .$$
We define the boundary condition 
$\eta=(\eta^0,\eta^\infty)$ of \eqref{jxxq} 
by weighting the parts of $\eta^0$ with the
class 
of the point $(0,0,\infty)\in \mathbf{Y}_\infty$
 and the parts of
$\eta^\infty$ with the class of the point $(\infty,0,\infty)\in \mathbf{Y}_\infty$.

%

Let ${\mathbb Q}^{{\mathcal P}(d,2)
}$ denote the linear
space of functions from ${\mathcal P}(d,2)$ to the field
${\mathbb Q}(q,s_1,s_2,s_3)$.
Let $\mathsf{p_0},\mathsf{p_\infty} \in 
H^*_\mathbf{\com^*_1}(\mathbf{P}_1^1
,\mathbb{Q})$
be the classes of the fixed points $0,\infty \in \mathbf{P}^1_1$.
Let 
\begin{equation}\label{gvv12}
\tilde{\tau}(\mathsf{p})= \sum_{i=0}^\infty c^0_i \tau_i(\mathsf{p}_0[0,0])
+\sum_{i=0}^\infty c^\infty_i \tau_i(\mathsf{p}_\infty[0,0])
\end{equation}
be a {\em finite}
linear combination of descendents.
For $\w\geq 0$, define a function on ${\mathcal P}(d,2)$ by:
$$\gamma_{\w}: {\mathcal P}(d,2) \rarr {\mathbb Q}(q,s_1,s_2,s_3), \ \ \
\eta \mapsto  
\mathsf{C}^{\mathsf{P}}_0\Big( 
\tilde{\tau}(\mathsf{p})^\w
,
\ d[\mathbf{P}^1_3]
\Big)_{\eta,\emptyset}^{\mathbf{Y}/\mathbf{Y}_\infty\cup 
\mathbf{D}_\infty,\mathbf{T}} \ . $$
Here, $\mathsf{C}^{\mathsf{P}}_0$ is defined by a multilinear expansion of the
insertion $\tilde{\tau}(\mathsf{p})^\w$.

\begin{Lemma} \label{cakeshop}
For $d\geq 0$, there exists a linear combination $\tilde{\tau}(\mathsf{p})$ for which
the set of functions,  
$$\{\gamma_{0}, \gamma_{1}, \gamma_{2}, \dots \},$$
spans ${\mathbb Q}^{{\mathcal P}(d,2)}$.
\end{Lemma}

\begin{proof}
The spanning statement concerns the linear algebra of the field of rational functions
${\mathbb Q}(q,s_1,s_2,s_3)$. We must prove non-degeneracy of the 
set $\{\gamma_{0}, \gamma_{1}, \gamma_{2}, \dots \}$.
We will require only the leading $q$ term of
$$\mathsf{C}^{\mathsf{P}}_0\Big( 
\tilde{\tau}(\mathsf{p})^\w
,
\ d[\mathbf{P}^1_3]
\Big)_{\eta,\emptyset}^{\mathbf{Y}/\mathbf{Y}_\infty\cup 
\mathbf{D}_\infty,\mathbf{T}}\ .$$
The matter is then an assertion about the 
cohomology of the Hilbert scheme of points of $\mathbf{Y}_\infty$.

After changing the basis of $\eta$ to the 
$\com_1^* \times \com_2^*$-fixed points
of the Hilbert scheme of points of $\mathbf{Y}_\infty$ at $(0,0,\infty) $ and $(\infty,0,\infty) $,
the action of $\tilde{\tau}(\mathsf{p})$
is determined in Section 1.2 of \cite{PPstat}.
The operator ${\tau}_k(\mathsf{p}_0[0,0])$
is diagonal in the fixed point basis with 
eigenvalues given by symmetric functions 
in the weights{\footnote{The weights
depend only on $s_1$ and $s_2$.}} of the structure sheaf of the
fixed point of the Hilbert scheme of $\mathbf{Y}_\infty$ 
at $(0,0,\infty)$.
Modulo lower order symmetric functions,
the eigenvalue of ${\tau}_k(\mathsf{p}_0[0,0])$
is simply the $k^{th}$ power sum.

Since the power sums determine all symmetric functions,
we can find a finite linear combination
$$\tilde{\tau}(\mathsf{p})= \sum_{i=0}^\infty c^0_i \tau_i(\mathsf{p}_0[0,0])
+\sum_{i=0}^\infty c^\infty_i \tau_i(\mathsf{p}_\infty[0,0])
$$ 
with distinct eigenvalues on the 
$\com_1^* \times \com_2^*$-fixed points
of the Hilbert scheme of $\mathbf{Y}_\infty$ at $(0,0,\infty)$ and 
$(\infty,0,\infty)$.
By the Vandermonde determinant, the Lemma is proven.
\end{proof}

Lemma \ref{cakeshop} is similar to Lemma 5.6  of \cite{vir}.
The parallel result for
$$\mathsf{C}^{\mathsf{GW}}_0\Big( 
\overline{\tilde{\tau}(\mathsf{p})^\w}
,
\ d[\mathbf{P}^1_3]
\Big)_{\eta,\emptyset}^{\mathbf{Y}/\mathbf{Y}_\infty\cup 
\mathbf{D}_\infty,\mathbf{T}}\ $$
follows from the correspondence of Theorem \ref{ppp666}.

We have used the full $\mathbf{T}$-action to prove 
Lemma \ref{cakeshop}. However, the weight $s_3$ of third factor $\com^*_3$
of $\mathbf{T}$
is not needed in the argument.
Hence, Lemma \ref{cakeshop} holds for 
$$\gamma_{\w}: {\mathcal P}(d,2) \rarr {\mathbb Q}(q,s_1,s_2), \ \ \
\eta \mapsto  
\mathsf{C}^{\mathsf{P}}_0\Big( 
\tilde{\tau}(\mathsf{p})^\w
,
\ d[\mathbf{P}^1_3]
\Big)_{\eta,\emptyset}^{\mathbf{Y}/\mathbf{Y}_\infty\cup 
\mathbf{D}_\infty,\com^*_1\times \com^*_2} \ . $$

\subsection{Even theory}
We first prove Proposition \ref{bbb999} in case all
the cohomology insertions $\gamma_j$ 
and all the cohomology weights $\omega_i$ are
of even (real) degree.

If the underlying curve
$C$ is $\mathbf{P}^1$, Proposition \ref{bbb999} specializes to
Proposition \ref{aaa999} and is established. Consider a 
fiber $F$ of 
\begin{equation}\label{vppq}
\epsilon: R=\mathbf{P}(\Lambda_0\oplus \Lambda_\infty) \rightarrow \mathbf{P}^1.
\end{equation}
We can degenerate $R$ to the normal cone of $F$,
$$R\ \leadsto \   R\cup_{F} \mathbf{P}^1 \times \mathbf{P}^1\ .$$
We can degenerate the line bundles
\begin{equation} \label{gxcs}
\Lambda_0,\ \Lambda_\infty,\  \epsilon^*L^C_0, \ \epsilon^*L^C_\infty
\end{equation}
so the restrictions to $\mathbf{P}^1\times \mathbf{P}^1$ are
all trivial.
By the restriction of Theorem \ref{ppp666} 
to the subtorus $\com_1^*\times \com_2^*$, Lemma \ref{cakeshop} restricted to
$\com_1^*\times \com_2^*$, and the
compatibility of the descendent correspondence
with the degeneration formula, we conclude
Conjecture \ref{ttt555} holds $\com^*_1\times \com^*_2$-equivariantly
for $\mathbf{P}_R/\mathbf{P}_F\cup R_\infty$.
Repeating the argument for another fiber $F'$ of $\epsilon$ shows
Conjecture \ref{ttt555} holds $\com^*_1\times \com^*_2$-equivariantly
for 
$\mathbf{P}_R/\mathbf{P}_{F\cup F'} \cup R_\infty$.

Next suppose $E$ is a genus $1$ carrying line bundles
$\Lambda_0$, $\Lambda_\infty$, $L_0^E$, and $L_\infty^E$ with
 $$\epsilon: S=\mathbf{P}(\Lambda_0\oplus \Lambda_\infty) \rightarrow  E\ .$$
We can degenerate
$E$ to a nodal rational curve. The line bundles carried by $E$
can be taken to specialize to line bundles on the nodal curve.
Since there is no vanishing even cohomology for the degeneration,
we conclude 
Conjecture \ref{ttt555} holds $\com^*_1\times \com^*_2$-equivariantly
for the genus 1 case $\mathbf{P}_S/S_\infty$ as a consequence
of the genus 0 case $\mathbf{P}_R/\mathbf{P}_{F\cup F'} \cup R_\infty$.

Since we know Conjecture \ref{ttt555} holds 
$\com^*_1\times \com^*_2$-equivariantly
for the genus 1 case $\mathbf{P}_S/S_\infty$, degeneration to
the normal cone to fibers of $\epsilon$
and Lemma \ref{cakeshop} restricted to $\com^*_1\times \com^*_2$
prove Conjecture
\ref{ttt555} holds $\com^*_1\times \com^*_2$-equivariantly
for the genus 1 cases
\begin{equation}\label{fvv3} 
\mathbf{P}_S/\mathbf{P}_F\cup S_\infty\ , \ \
\mathbf{P}_S/\mathbf{P}_{F\cup F'} \cup S_\infty\ .
\end{equation}

Finally, if $C$ is curve of arbitrary genus $g$, we can degenerate
$C$ to a chain of $g$ elliptic curves. Since there
is no vanishing cohomology, we deduce Conjecture
\ref{ttt555} in the even case from the geometries \eqref{fvv3}. \qed
\label{evven}

\subsection{Odd theory}\label{godd}
\subsubsection{Reduction to genus 1}
Suppose $C$ is a genus $g$ curve carrying line bundles
$\Lambda_0$, $\Lambda_\infty$, $L_0^C$, and $L_\infty^C$ with
 $$\epsilon: S=\mathbf{P}(E_0\oplus E_\infty) \rightarrow  C\ .$$
We now consider 
Conjecture \ref{ttt555} equivariantly with respect
to $\com_1^*\times \com_2^*$ for
classes $\gamma_j$ and $\omega_i$ in full generality.

Since the degeneration of
$C$ to a chain of genus 1 curves has no
vanishing cohomology, we may assume $C$ is of genus 1.
By the $\com_1^* \times \com_2^*$-equivariant
methods of Section \ref{evven}, Conjecture \ref{ttt555}
for \eqref{fvv3} follows from  Conjecture \ref{ttt555}
for $\mathbf{P}_S/S_\infty$.

We may further simplify the geometry by degenerating $C$ to
the normal cone to a point $p\in C$,
$$C \ \leadsto \ C \cup_p \mathbf{P}^1\ $$
and requiring the line bundles
$
\Lambda_0,\ \Lambda_\infty,\  \epsilon^*L^C_0, \ \epsilon^*L^C_\infty
$
to be trivial on $C$ in the limit. We then obtain
a degeneration
$$S\ \leadsto \   \mathbf{P}^1 \times C \ \cup_{F}\  R\ ,$$
where $R$ is of the form \eqref{vppq}.
Since Conjecture \ref{ttt555}
has been already proven $\com_1^*\times \com_2^*$-equivariantly
for 
$\mathbf{P}_R/\mathbf{P}_F\cup R_\infty$,
we need only prove Conjecture \ref{ttt555} 
holds $\com_1^*\times \com_2^*$-equivariantly
for the
special case 
\begin{equation}\label{spcase}
\epsilon: S= \mathbf{P}(\mathcal{O}_E \oplus \mathcal{O}_E) \rightarrow E,
\ \ \ L_0^E= \mathcal{O}_E, \ \ L_\infty^E=\mathcal{O}_E, \ \ g(E)=1\ .
\end{equation}
Explicitly, the geometry is 
$$\mathbf{P}_S/S_\infty = 
\mathbf{P}^1\times \mathbf{P}^1 \times E \ / \
\mathbf{P}^1\times \{\infty \} \times E
 .$$

\subsection{Proof of Proposition \ref{bbb999}} \label{togg4}
Consider
the stable pair and Gromov-Witten  theories 
\begin{equation}\label{lssz}
\mathsf{C}^{\mathsf{P}}_0\left(\tau_\alpha(\Gamma_0), \beta\right)_\mu
^{\mathbf{P}_S/S_\infty,\com^*_1\times \com^*_2}\ , \ \ \
\mathsf{C}^{\mathsf{GW}}_0\Big(\overline{\tau_\alpha(\Gamma_0)}, \beta\Big)_\mu
^{\mathbf{P}_S/S_\infty,\com^*_1\times \com^*_2}
\end{equation}
for the genus 1 geometry \eqref{spcase}.
Both 
are uniquely determined from the corresponding
even theories
by the following four properties:
\begin{enumerate}
\item[(i)] Algebraicity of the virtual class,
\item[(ii)] Degeneration formulas for the relative theory in the
presence of odd cohomology,
 \item[(iii)] Monodromy invariance of the relative theory,
\item[(iv)] Elliptic vanishing relations.
\end{enumerate}
The properties (i)-(iv) were used in \cite{vir} to determine the
full relative Gromov-Witten descendents of target curves in terms
of the descendents of even classes.

The results of Section 5 of \cite{vir}  are entirely formal and
apply verbatim to the theories \eqref{lssz}. 
Lemma \ref{cakeshop} replaces Lemma 5.6 of \cite{vir}.
Let 
$$\mathsf{L}_0, \mathsf{L}_\infty \in H^*_{\com_1^*\times \com_2^*}(
 \mathbf{P}^1_1\times \mathbf{P}^1_2 \times E,\mathbb{Q})$$
be the classes of the curves
$$\{0\} \times \{0\} \times E \ \ \ \text{and} \ \ \
 \{\infty\} \times \{0\} \times E
$$
respectively.
For $\gamma \in H^*(E,\mathbb{Q})$, 
let{\footnote{
Here, the coefficients $c_i^0$ and $c_i^\infty$
are
taken so Lemma \ref{cakeshop} is valid.
When $\gamma$ is a class of a point in $E$, we recover
the $\com_1^* \times \com_2^*$ specialization of \eqref{gvv12}.}}
$$\tilde{\tau}(\gamma)= \sum_{i=0}^\infty c^0_i \tau_i(
\mathsf{L}_0\gamma)
+\sum_{i=0}^\infty c^\infty_i \tau_i(\mathsf{L}_\infty\gamma).
$$ 
We start, as in (5.11) of \cite{vir}, by studying the
descendents
\begin{equation}\label{lssz2}
\mathsf{C}^{\mathsf{P}}_0\left(
\tilde{\tau}(\alpha) \tilde{\tau}(\beta)
, d_1[\mathbf{P}^1_1]+d_3[E]\right)_{\eta,\emptyset}
^{\mathbf{P}_S/S_\infty,\com^*_1\times \com^*_2}\ ,
\end{equation}
$$
\mathsf{C}^{\mathsf{GW}}_0\Big(
{\overline{\tilde{\tau}(\alpha) \tilde{\tau}(\beta)}}
, d_1[\mathbf{P}^1_1]+d_3[E]
\Big)
_{\eta,\emptyset}
^{\mathbf{P}_S/S_\infty,\com^*_1\times \com^*_2}$$
for the geometry \eqref{spcase} where
$$\alpha,\beta \in H^1(E,\mathbb{Q}), \ \ \alpha\cup \beta=1\ . $$ 
Exactly following \cite{vir}, the
descendents \eqref{lssz2}
are determined from the even theory. 
Since the relations (i-iv) 
respect the descendent correspondence, we deduce Proposition \ref{bbb999}
from the even case proven in Section \ref{evven}.
Also, the rationality of the even theory implies
the rationality of the full theory.
When the invertibility
of Lemma \ref{cakeshop} is applied here, an induction on
$d_1$ is necessary.

The method developed in Section 5 of \cite{vir} proceeds
to handle all descendent insertions. For the case studied in \cite{vir},
the descendents 
$$\tau_n(\gamma), \ \ \ \gamma\in H^*(E,\mathbb{Q})$$ 
are labelled only by the integer $n$. Here, the insertions
are of the form
$$\tau_n(\mathsf{L}_0\gamma), \ \ \tau_n(\mathsf{L}_\infty \gamma), \ \
(n, \mathsf{L}_0\gamma), \ \ (n,\mathsf{L}_\infty\gamma), \ \ \ \ 
\gamma\in H^*(E,\mathbb{Q}) $$
where the latter two types{\footnote{
The relative conditions
{\em must} be treated on the same footing as the descendent insertions
as the relative weights may also be odd.}}
are
relative conditions of $\mu$. The insertion labelling is
the only difference. The reduction to the even descendents
exactly follows
Section 5 of \cite{vir}.  \qed

\subsection{Products $S\times C$}
Let $S$ be a nonsingular projective toric surface equipped
with the action of a 2-dimensional torus $T$.
Let $C$ be a nonsingular projective curve of genus $g$.
A simpler result than Theorem \ref{xxx999} is the
following.

\begin{Proposition}
\label{yyy999} 
For $\gamma_i \in H^{*}_{T}(S\times C,\mathbb{Q})$,
we have 
$$\ZZ_{\mathsf{P}}\Big(S\times C   ;q\ \Big|\  
{\tau_{\alpha_1-1}(\gamma_1)\cdots
\tau_{\alpha_{\ell}-1}(\gamma_{\ell})} \ 
\Big)^T_\beta\in \mathbb{Q}(q,s_1,s_2)$$
and the correspondence
\begin{multline*}
(-q)^{-d_\beta/2}\ZZ_{\mathsf{P}}\Big(S\times C   ;q\ \Big|  
\ {\tau_{\alpha_1-1}(\gamma_1)\cdots
\tau_{\alpha_{\ell}-1}(\gamma_{\ell})} \ 
\Big)^T_\beta \\ =
(-iu)^{d_\beta}\ZZ'_{\mathsf{GW}}\Big(S\times C
;u\ \Big|   
\ \overline{\tau_{a_1-1}(\gamma_1)\cdots
\tau_{\alpha_{\ell}-1}(\gamma_{\ell})}
\ \Big)^T_\beta 
\end{multline*}
under the variable change $-q=e^{iu}$.
\end{Proposition}

\begin{proof}
Let $p_1,\ldots,p_n$ denote the $T$-fixed points of $S$.
By considering localization for stable pairs and stable maps on
$S\times C$
with respect to the torus $T$,
we can reduce the descendent correspondence to a local result for
$\com^2 \times C$ with caps in the two $\com^2$ directions.
The localization formula is in terms of $n$ such
capped $\com^2\times C$ geometries (connected by
simple capped edge geometries).

Consider the capped geometry $\com^2\times C$. 
If all the descendent
insertions $\gamma_i$  have even (real) cohomological
degree, we can reduce
to the case where $g(C)=0$ by our standard degeneration and
relative arguments. Crucial here is a tri-relative residue theory
for 
\begin{equation}\label{frood}
\com^2 \times \proj^1 \ /\ \com^2\times \{\infty\}
\end{equation}
 defined
by completely parallel constructions to the bi-relative case
considered in Section \ref{xxx3}. The tri-relative
geometry has caps in the two $\com^2$ directions on \eqref{frood}.
The proof of the GW/P descendent correspondence for the
tri-relative cap \eqref{frood} exactly follows the proof
Theorem \ref{ppp444} with the two relative directions corresponding
to $\com^2$ handled as in the proof Theorem \ref{ppp666}.

To control the odd descendents,
we follow the strategy of Section \ref{godd} (and \cite{vir}).
We reduce to the case where $g(C)=1$ and
express the full theory in terms of the even theory.
Lemma \ref{cakeshop} for $T$ still applies.
  
We leave the straightforward details here to the reader.
The capped edges are 1-leg geometries and are easily 
handled.
\end{proof}

\section{Proof of Theorem \ref{qqq111}}
\label{laman}
\subsection{Overview}
We have now proven the GW/P descendent correspondences in
sufficiently many geometries to deduce Theorem \ref{qqq111}.
The idea is to degenerate the complete intersection by factoring
the equations. We present the proof carefully for the
quintic in $\mathbf{P}^4$ following the scheme of \cite{mptop}.
The argument for general Fano and
Calabi-Yau complete intersections in products of projective
spaces
is identical.

\subsection{Simple theories}
A {\em complete intersection pair} $(V,W)$ is a 
nonsingular complete intersection 
$$ V \subset \PP^{n_1} \times \cdots \times \PP^{n_m}\ .$$
together with a nonsingular divisor $W\subset V$
cut out by a hypersurface in $\PP^{n_1} \times \cdots \times \PP^{n_m}$.
In particular, 
$$W \subset 
\PP^{n_1} \times \cdots \times \PP^{n_m}$$
is also a complete intersection.

A class $\gamma\in H^*(V,\Q)$ is {\em simple} if $\gamma$ lies in the image of
the restriction map
$$H^*(\PP^{n_1} \times \cdots \times \PP^{n_m},\Q) \rarr H^*(V,\Q).$$
If $V$ is nonsingular complete intersection of dimension 3,
the simple cohomology of $V$ equals the even cohomology by 
the Lefschetz results.

The  simple Gromov-Witten and stable pairs theories of $V$ consist of 
the integrals of descendents of simple classes. Similarly, the
simple Gromov-Witten and stable pairs theories of the 
relative geometry $V/W$ consists of integrals of
descendents of simple classes with {\em no restrictions on the 
cohomology classes of $W$ in the relative constraints}.

\subsection{Fano and Calabi-Yau hypersurfaces in $\proj^4$}
\label{zzwe}
\subsubsection{Notation}
The following notation for curves, surfaces, and 3-folds will be convenient
for our study:
\begin{enumerate}
\item[(i)] let $C_{d_1,d_2}\subset \proj^3$ be a nonsingular complete intersection
 of type $(d_1,d_2)$,
\item[(ii)] let $S_d \subset \proj^3$ be a nonsingular surface of degree $d$,
\item[(iii)] let $T_d \subset \proj^4$ be a nonsingular 3-fold of degree $d$.
\end{enumerate}
Finally, let $\proj^3[d_1,d_2]$ be the blow-up of $\proj^3$ along $C_{d_1,d_2}$.

\subsubsection{Degeneration for the quintic}
Let $\mathsf{Z}^{\mathsf{P}}(T_5^\star)$ denote the simple stable pairs descendent theory of the Calabi-Yau
quintic 3-fold.{\footnote{Here and below, the superscript
$\star$ will indicate simple theories.}}  Factoring the quintic equation yields a
degeneration:
\begin{equation}\label{mnnv}
T_5 \ \leadsto\ T_4 \cup_{S_4} \proj^3[4,5]
\end{equation}
where $S_4\subset T_4$ is a linear section and
$S_4\subset \proj^3[4,5]$ is the strict transform of
a quartic containing $C_{4,5}\subset \proj^3$.
See Section 0.5.4
 of \cite{mptop} for a detailed construction of the degeneration
\eqref{mnnv}.
The degeneration formula expresses 
$\mathsf{Z}_{\mathsf{P}}(T_5^\star)$ in terms of the
relative theories of the special fibers. We write the
relation schematically as
$$\mathsf{Z}_{\mathsf{P}}(T_5^\star) \ \leadsto \ \
\mathsf{Z}_{\mathsf{P}}(T_4^\star/S_4) \ \text{and} \  
\mathsf{Z}_{\mathsf{P}}(\proj^3[4,5]/S_4)\ .$$
Similarly, in Gromov-Witten theory, we have
the determination
 $$\mathsf{Z}'_{\mathsf{GW}}(T_5^\star) \ \leadsto \ \
\mathsf{Z}'_{\mathsf{GW}}(T_4^\star/S_4), \ \text{and} \
\mathsf{Z}'_{\mathsf{GW}}(\proj^3[4,5]/S_4)\ .$$
By the compatibility of the descendent correspondence with
degeneration, the descendent correspondences for
$T_4^\star/S_4$ and $\proj^3[4,5]/S_4$
imply Theorem \ref{qqq111} for $T_5$.

\subsubsection{Descendent correspondence for
$\proj^3[4,5]/S_4$} \label{bbll}
Let us start with $\proj^3[4,5]/S_4$. Degeneration to the
normal cone of $$S_4\subset \proj^3[4,5]$$ yields the
determination{\footnote{When no superscript appears
on the partition function, the statement is understood
to hold for both stable pairs and Gromov-Witten theory.}}
\begin{equation}\label{dter}
\mathsf{Z}(\proj^3[4,5]) \ \leadsto \ \
\mathsf{Z}(\proj^3[4,5]/S_4) \ \text{and} \  \mathsf{Z}(
\mathbf{P}_{S_4}/S_4)\ .
\end{equation}
We know the
descendent correspondence for all projective bundle
geometries $\proj_{S_4}/S_4$ by Proposition \ref{kk33}.
By the invertibility of Proposition 6 of \cite{PP2},
the determination \eqref{dter} can be reversed,
\begin{equation*}
\mathsf{Z}(\proj^3[4,5]/S_4) \ \leadsto \ \
\mathsf{Z}(\proj^3[4,5]) \ \text{and} \  \mathsf{Z}(
\mathbf{P}_{S_4}/S_4)\ .
\end{equation*}
Hence, the descendent correspondence for 
$\proj^3[4,5]$ implies the descendent correspondence
for $\proj^3[4,5]/S_4$.

An approach to the blow-up $\proj^3[4,5]$ is explained in
Section 3.1 of \cite{mptop}. Let $S_4\subset \proj^3$
contain $C_{4,5}$. Degeneration to the normal cone of
$S_4\subset \proj^3$ yields
\begin{equation}\label{zbbb}
\proj^3 \ \leadsto \ \proj^3 \cup_{S_4} {\proj}_{S_4}
\end{equation}
for the projective bundle geometry 
\begin{equation}\label{fragg}
\pi:{\proj}_{S_4}= \proj(\mathcal{O}_{S_4} \oplus\mathcal{O}_{S_4}(4))\rightarrow S_4, \ \ \ {\proj}_{S_4}/(S_4)_\infty\ .
\end{equation}
The original curve $C_{4,5} \subset \proj^3$
has limit in $(S_4)_0 \subset {\proj}_{S_4}$.
After
blowing up the degeneration \eqref{zbbb} along the moving curve $C_{4,5}$,
we obtain
$$\proj^3[4,5] \ \leadsto \ \proj^3 \cup_{S_4} X, \ \ \ \ 
X= \mathcal{B}_{C_{4,5}}({\mathbf{P}}_{S_4})\ .$$
Here $X$ is the blow-up of
${\mathbf{P}}_{S_4}$ along $C_{4,5} \subset (S_4)_0$.
In order to prove the descendent correspondence for
$\proj^3[4,5]$, we need only prove the descendent
correspondence for $\proj^3/S_4$ and
$X/S_4$.
Using the established descendent
correspondences for the toric variety $\proj^3$ and
projective bundles
over $S_4$ and the invertibility of Proposition 6 of \cite{PP2},
we need only prove the descendent correspondence for
$X$.

To study $X$, we repeat the blow-up construction of the
previous paragraph.
Let $\proj_{S_4}$ be the projective bundle \eqref{fragg}, and
let $$\proj_{C_{4,5}} \subset \proj_{S_4}$$ be the divisor lying
over $C_{4,5}\subset S_4$ via $\pi$.
Degeneration to the normal cone of
$\proj_{C_{4,5}} \subset \proj_{S_4}$
yields
\begin{equation}\label{zbbbb}
\proj_{S_4} \ \leadsto \ \proj_{S_4} \cup_{\proj_{C_{4,5}}} {\proj}_{\proj_{C_{4,5}}}\ .
\end{equation}
for the projective bundle geometry 
\begin{equation}\label{fraggg}
{\proj}_{\proj_{C_{4,5}}}= \proj(\mathcal{O}_{C_{4,5}} \oplus\mathcal{O}_{C_{4,5}}(5))\times_{C_{4,5}}
\proj(\mathcal{O}_{C_{4,5}} \oplus\mathcal{O}_{C_{4,5}}(4))\rightarrow {C_{4,5}}
\end{equation}
relative to 
$${\proj_{C_{4,5}}} = \proj(\mathcal{O}_{C_{4,5}}(5)) 
\times_{C_{4,5}}
\proj(\mathcal{O}_{C_{4,5}} \oplus\mathcal{O}_{C_{4,5}}(4))\ .$$
The original curve $C_{4,5} \subset (S_4)_0\subset \proj_{S_4}$
has limit equal to
$$ C_{4,5}= \proj(\mathcal{O}_{C_{4,5}}) 
\times_{C_{4,5}}
\proj(\mathcal{O}_{C_{4,5}})
\subset \proj(\mathcal{O}_{C_{4,5}} \oplus\mathcal{O}_{C_{4,5}}(5))\times_{C_{4,5}}
\proj(\mathcal{O}_{C_{4,5}} \oplus\mathcal{O}_{C_{4,5}}(4))\ .$$
After
blowing up the degeneration \eqref{zbbb} along the moving curve $C_{4,5}$,
we obtain
$$X \ \leadsto \ \mathbf{P}_{S_4} \cup_{\proj_{C_{4,5}}} Y\ ,$$
$$ 
Y= \mathcal{B}_{C_{4,5}}(
\proj(\mathcal{O}_{C_{4,5}} \oplus\mathcal{O}_{C_{4,5}}(5))\times_{C_{4,5}}
\proj(\mathcal{O}_{C_{4,5}} \oplus\mathcal{O}_{C_{4,5}}(4))
)\ .$$
In order to prove the descendent correspondence for
$Y$, we need only prove the descendent
correspondence for $\mathbf{P}_{S_4} / {\proj_{C_{4,5}}}$ and
$Y/\proj_{C_{4,5}}$.

As we have seen in \eqref{fraggg}, $\proj_{\proj_{C_{4,5}}}$
is a projective bundle over $\proj_{C_{4,5}}$ of the
form required by Theorem \ref{xxx999}. Hence, the
descendent correspondence holds for
$\mathbf{P}_{S_4} / {\proj_{C_{4,5}}}$ by inverting
\begin{equation*}
\mathsf{Z}(\proj_{S_4}) \ \leadsto \ \
\mathsf{Z}(\proj_{S_4}/\proj_{C_{4,5}}) \ \text{and} \  \mathsf{Z}(
\mathbf{P}_{\proj_{C_{4,5}}}/S_4)\ . 
\end{equation*}
The invertibility is possible again
by Proposition 6 of \cite{PP2}.
Similarly, 
the descendent correspondence  for $Y$ implies the
descendent correspondence for 
$Y/\proj_{C_{4,5}}$.

The last step in proving the descendent correspondence
for $\proj^3[4,5]/S_4$ is to prove the descendent
correspondence for $Y$.
The 3-fold $Y$ is a bundle over $C_{4,5}$ with fiber
equal to $\mathcal{B}_p(\proj^1\times \proj^1)$, 
the blow-up of $\proj^1\times \proj^1$ at a point.
By further degeneration arguments{\footnote{We can
degenerate 
$$C_{4,5}\leadsto C_{4,5}\cup \proj^1$$ 
and require all the twisting of $Y$ to lie over $\proj^1$.}},
we can reduce
to the case $$\mathcal{B}_p(\proj^1\times \proj^1) \times C_{4,5}$$
covered by Proposition \ref{yyy999}. \qed

\subsection{Proof of Theorem \ref{qqq111}}
We turn now to $T_4^\star/S_4$. The normal cone degeneration
$$T_4 \ \leadsto \ T_4 \cup_{S_4} \mathbf{P}_{S_4}$$
and invertibility yields the determination
\begin{equation*}
\mathsf{Z}(T_4^\star/S_4) \ \leadsto \ \
\mathsf{Z}(T_4^\star) \ \text{and} \  \mathsf{Z}(
\mathbf{P}_{S_4}/S_4)\ .
\end{equation*}
Hence, the descendent correspondence
for $T_5^\star$ follows from the descendent correspondence for $T_4^\star$. 

By factoring the quartic equation defining $T_4 \subset \proj^3$,
degree reduction can be continued. The full reduction
scheme for the quintic is:

\begin{eqnarray*}
\mathsf{Z}(T_5^\star) & \leadsto &  \mathsf{Z}(T^\star_4/S_4) \
\text{and}\  \mathsf{Z}(\proj^3[4,5]/S_4), \\
\mathsf{Z}(T_4^\star/S_4)&  \leadsto &
 \mathsf{Z}(T_4^\star) \ \text{and} \ \mathsf{Z}(\proj_{S_4}/ S_4),\\
\mathsf{Z}(T_4^\star) & \leadsto &  \mathsf{Z}(T^\star_3/S_3) \
\text{and}\  \mathsf{Z}(\proj^3[3,4]/S_3), \\
\mathsf{Z}(T_3^\star/S_3)&  \leadsto &
 \mathsf{Z}(T_3^\star) \ \text{and} \ \mathsf{Z}(\proj_{S_3}/ S_3),\\
\mathsf{Z}(T_3^\star) & \leadsto &  \mathsf{Z}(T^\star_2/S_2) \
\text{and}\  \mathsf{Z}(\proj^3[2,3]/S_2), \\
\mathsf{Z}(T_2^\star/S_2)&  \leadsto &
 \mathsf{Z}(T_2^\star) \ \text{and} \ \mathsf{Z}(\proj_{S_2}/ S_2),\\
\mathsf{Z}(T_2^\star) & \leadsto &  \mathsf{Z}(T^\star_1/S_1) \
\text{and}\  \mathsf{Z}(\proj^3[1,2]/S_1), \\
\mathsf{Z}(T_1^\star/S_1)&  \leadsto &
 \mathsf{Z}(T_1^\star) \ \text{and} \ \mathsf{Z}(\proj_{S_1}/ S_1)\ .
\end{eqnarray*}

The end points of the scheme are $T_1^*$ (which is toric),
projective bundles over toric and $K3$ surfaces, and
blown-up projective spaces. The descendent correspondence
has been established for all the end points --- the blown-up
projective spaces are handled by the method of Section \ref{bbll}.



We have proven the GW/P descendent correspondence for the
even theories of all Fano hypersurfaces in $\proj^4$. 
We can degenerate all Fano and Calabi-Yau 3-fold complete intersections
by an identical factoring argument. The outcome is a proof of
Theorem \ref{qqq111}.\qed

\subsection{Proof of Corollary \ref{ccc000}}
To prove Corollary \ref{ccc000}, we start with the
descendent correspondence of Theorem \ref{qqq111}.
The initial term results of Theorem 2  of \cite{PPDC}
then imply the Corollary. \qed

\subsection{The Enriques Calabi-Yau}\label{et45}
As a further example,
we prove the GW/P correspondence for the Enriques Calabi-Yau 3-fold
studied in \cite{klemmm,newcal}.

Let $\sigma$ act freely on the product $K3 \times E$ by
an Enriques involution $\sigma_{K3}$ on the $K3$ and by -1
on the elliptic curve.
By definition, the quotient
\begin{equation*}
Q = \left( 
K3 \times E\right) / \langle \sigma 
\rangle
\end{equation*}
is an {\em Enriques Calabi-Yau 3-fold}.
Since $K3\times E$ carries a holomorphic 3-form invariant
under $\sigma$, the canonical class is trivial
$$\omega_Q = \mathcal{O}_Q.$$
By projection on the right,
\begin{equation}\label{ft4}
Q \rightarrow E/\langle -1 \rangle = \mathbf{P}^1
\end{equation}
is a $K3$ fibration with 4 double Enriques fibers.

Let $\text{inv}_{\mathbf{P}^1}$ be an involution of $\mathbf{P}^1$
with 2 fixed points.
Let $\tau$ act freely on the product $K3 \times {\mathbf{P}}^1$ by
$(\sigma_{K3},\text{inv}_{\mathbf{P}^1})$.
Let
$$R = \left( 
K3 \times \mathbf{P}^1 \right) / \langle \tau 
\rangle$$
denote the quotient.
By projecting left,  
\begin{equation}\label{kx33}
R \rightarrow K3/\langle \sigma_{K3} \rangle=S
\end{equation}
is a projective bundle over the Enriques surface $S$.
Two sections of the bundle \eqref{kx33} are obtained from the fixed
points of $\text{inv}_{\mathbf{P}^1}$.
By projecting right,
$$R \rightarrow \mathbf{P}^1/\langle \text{inv}_{\mathbf{P}^1}\rangle$$
is a $K3$ fibration with  2 double Enriques
fibers.

By degenerating the $K3$ fibration \eqref{ft4}, we find a
degeneration of the Enriques Calabi-Yau $Q$,
$$Q \ \leadsto \ R \cup _{K3} R$$
where the intersection $K3$ is a common fiber, see \cite[Section 1.4]
{newcal}.
Hence the GW/P correspondence for $Q$ is reduced
to the GW/P descendent correspondence for $R/K3$. The latter
reduces to the GW/P descendent correspondence for $R$
by degeneration to the normal cone and Proposition \ref{kk33}.

The Enriques surface $S$ degenerates{\footnote{The second homologies
$H_2(S,\mathbb{Z})$ and $H_2(Q,\mathbb{Z})$ have
 2-torsion for the Enriques surface and the Enriques Calabi-Yau.
We specify here the curve class only mod torsion.
Data about the torsion refinement is lost in the degeneration of 
$S$. 
}}
to a union along an elliptic curve
of a $\mathbf{P}^1$-bundle over an elliptic curve
and the rational elliptic surface, see \cite[Section 1.3]{newcal}.
 We use the corresponding degeneration of \eqref{kx33}
to prove the GW/P descendent correspondence for 
$R$.
We obtain the following result.

\begin{Proposition}
\label{ccc111xxx} 
Let $Q$ be the Enriques Calabi-Yau, and let
$\beta \in H_2(Q,\mathbb{Z})/ \mathsf{tor}$ be a curve class.
Then,
$$\ZZ_{\mathsf{P}}\Big(Q;q 
\Big)_\beta \ \in \mathbb{Q}(q)\ ,$$
and we have the correspondence
$$
\ZZ_{\mathsf{P}}\Big(Q;q\Big)_\beta  =
\ZZ'_{\mathsf{GW}}\Big(Q;u\Big)_\beta 
$$
under the variable change $-q=e^{iu}$.
\end{Proposition}

\label{ecy}

\vspace{+4 pt}
\noindent Departement Mathematik \\
\noindent ETH Z\"urich \hfill  \\
\noindent rahul@math.ethz.ch  \\

\vspace{+2 pt}
\noindent
Department of Mathematics \hfill Department of Mathematics\\ 
Harvard University \hfill  Massachusetts Institute of Technology\\
apixton@math.harvard.edu \hfill apixton@mit.edu

\end{document}